\documentclass[11pt]{amsart}

\usepackage{amssymb}
\usepackage[all,cmtip]{xy}
\usepackage{mathrsfs}
\usepackage[foot]{amsaddr}

\newtheorem{theorem}{Theorem}[section]
\newtheorem{lem}[theorem]{Lemma}

\newtheorem{cor}[theorem]{Corollary}
\newtheorem{prop}[theorem]{Proposition}

\theoremstyle{definition}
\newtheorem{defi}[theorem]{Definition}
\newtheorem{example}[theorem]{Example}

\theoremstyle{remark}
\newtheorem{rem}[theorem]{Remark}

\newcommand{\BC}{\mathbb{C}}            %% field of complex numbers
\newcommand{\BZ}{\mathbb{Z}}             %% ring of natural numbers  
\newcommand{\BN}{\mathbb{N}}            %% ring of postive integers

\newcommand{\Glie}{\mathfrak{g}}        %% Lie algebra
\newcommand{\Hlie}{\mathfrak{h}}          %% Cartan subalgebra
\newcommand{\BGG}{\mathcal{O}}      %% category BGG
\newcommand{\CW}{\mathcal{W}}     %% standard module
\newcommand{\SL}{\mathscr{L}}     %% asymptotic module

\newcommand{\BQ}{\mathbf{Q}}                %% root lattice
\newcommand{\BP}{\mathbf{P}}            %% weight lattice
\newcommand{\wt}{\mathrm{wt}}         %% the set of weights
\newcommand{\lwt}{\mathrm{wt}_{\ell}}   %% the set of l-weights

\newcommand{\CEl}{\mathcal{E}_{\ell}}        %% q-character ring
\newcommand{\qc}{\chi_{\mathrm{q}}}            %% q-character
\newcommand{\nqc}{\widetilde{\qc}}            %% normalized q-character
\newcommand{\CE}{\mathcal{E}}         %% character ring

\newcommand{\CL}{\mathcal{L}}       %% affine weights
\newcommand{\CR}{\mathcal{R}}    %% rational affine weights
\newcommand{\CD}{\mathcal{D}}    %% one-dim affine weights

\newcommand{\Be}{\mathbf{e}}         
\newcommand{\Bf}{\mathbf{f}}  
\newcommand{\Bm}{\mathbf{m}} 
\newcommand{\Bn}{\mathbf{n}}  
\newcommand{\Br}{\mathbf{r}}
\newcommand{\Bs}{\mathbf{s}}

\allowdisplaybreaks                %% equations allowed to be breaked
\title[Shifted Yangians and R-matrices]{Shifted Yangians and polynomial R-matrices}
\author{David Hernandez}
\address{DH: Universit\'e de Paris and Sorbonne Universit\'e, CNRS, IMJ-PRG, IUF, F-75006, Paris, France}
\email{david.hernandez@u-paris.fr}
\author{Huafeng Zhang}
\address{HZ: CNRS, UMR 8524-Laboratoire Paul Painlev\'e, Univ. Lille, F-59000 Lille, France}
\email{huafeng.zhang@univ-lille.fr}

\begin{document}

\begin{abstract}
We study the category $\BGG^{sh}$ of representations over a shifted 
Yangian. This category has a tensor product structure and contains 
distinguished modules, the positive prefundamental 
modules and the negative prefundamental modules. Motivated by the 
representation theory of the Borel subalgebra of a quantum affine 
algebra and by the relevance of quantum integrable systems in this 
context, we prove that tensor products of prefundamental 
modules with irreducible modules are either cyclic or co-cyclic. This implies the existence and 
uniqueness of morphisms, the R-matrices, for such tensor products. We prove the 
R-matrices are polynomial in the spectral parameter, and we establish 
functional relations for the R-matrices. As applications, we prove the Jordan--H\"older property in the category $\BGG^{sh}$. We also obtain a proof, uniform for any finite type, that any irreducible module 
factorizes through a truncated shifted Yangian.

\bigskip 



\end{abstract}
\maketitle
\tableofcontents

\section{Introduction}

Shifted Yangians, and their truncations, appeared for type $A$ in the context of 
the representation theory of finite $W$-algebras in the work of Brundan-Kleshchev \cite{BK}, 
then in the study of quantized affine Grassmannian slices by Kamnitzer-Webster-Weekes-Yacobi \cite{KWWY} for general types 
and in the study of quantized Coulomb branches of 3d $N = 4$ SUSY quiver gauge 
theories by Braverman-Finkelberg-Nakajima \cite{BFN} for simply-laced types and by Nakajima-Weekes \cite{NW} for non simply-laced types.

Fix a finite-dimensional complex simple Lie algebra $\Glie$.
The shifted Yangians $Y_\mu(\Glie)$ can be seen as variations of the ordinary Yangian 
$Y(\Glie)$ in its Drinfeld presentation, but depending on a coweight 
$\mu$ in the coweight lattice, denoted by $\BP^{\vee}$, of the underlying simple Lie algebra $\mathfrak{g}$. 
In the particular case $\mu = 0$, we recover the Yangian $Y_0(\Glie) = Y(\Glie)$. The representations of Yangians, and their trigonometric analogs the quantum affine algebras $U_q(\hat{\mathfrak{g}})$, have been under 
intense study since several decades. The truncated shifted Yangians, certain remarkable quotient of shifted Yangians,  depend on additional parameters, including a dominant coweight $\lambda$. 
These parameters $\lambda$ and 
$\mu$ can be interpreted as parameters for generalized slices of the affine Grassmannian $\overline{\mathcal{W}}_\mu^\lambda$ (usual slices when $\mu$
is dominant). These varieties are also Coulomb 
branches, symplectic dual to Nakajima quiver varieties, and the truncated shifted Yangians can be
seen as quantizations of these symplectic varieties.

For simply-laced types, representations of shifted Yangians and related Coulomb branches have been 
intensively studied in this context, see \cite{BK, KTWWY0, KTWWY} and references therein. 
For non simply-laced types, representations of quantizations of Coulomb branches have been
studied by Nakajima and Weekes \cite{NW} by using the method originally developed in \cite{Nak2} for simply-laced types (the reader may refer to the discussion in the Introduction of \cite{H0}).

One crucial property of shifted Yangians is the 
existence \cite{coproduct} of a family of algebra homomorphisms indexed by a pair of coweights $\mu$ and $\nu$,
$$\Delta_{\mu,\nu}: Y_{\mu+\nu}(\Glie) \longrightarrow Y_{\mu}(\Glie) \otimes Y_{\nu}(\Glie).$$ 
This is analog to the Drinfeld-Jimbo coproduct for
ordinary Yangians. 
These coproducts $\Delta_{\mu,\nu}$ induce a tensor product structure on a category
$$\BGG^{sh} = \bigoplus_{\mu\in \BP^{\vee}} \BGG_\mu$$ 
which is a sum of categories $\BGG_\mu$ of representations over the shifted Yangians $Y_\mu(\mathfrak{g})$ for various coweights $\mu$.

By \cite{BK, KTWWY0, KTWWY}, an irreducible module in category $\BGG^{sh}$ is determined by its {\it highest weight}, which is a tuple of ratios of monic polynomials in $u$, one ratio for each Dynkin node of $\Glie$. There is a natural $\BC$-action on the shifted Yangians by algebra automorphisms. Each module $V$ induces on the same underlying vector space a family of module structures $V(a)$, for $a \in \BC$ referred to as {\it spectral parameter}, such that $V(0) = V$. If $V$ is irreducible, then $V(a)$ remains irreducible and its highest weight is obtained from that of $V$ by the substitution $u \mapsto u-a$.

\smallskip

In a seemingly different direction, the category $\mathcal{O}$ of representations of the Borel subalgebra $U_q(\hat{\mathfrak{b}})$ of a quantum affine algebra $U_q(\hat{\mathfrak{g}})$ was 
introduced and studied by Jimbo and the first author in \cite{HJ}. One crucial point for the approach therein is an asymptotical procedure to construct certain remarkable simple representations, the prefundamental representations, as limits of finite-dimensional representations of $U_q(\hat{\mathfrak{g}})$. It was observed by the second author in \cite{Z} that category $\BGG^{sh}$ for shifted Yangians provides a Yangian counterpart of the prefundamental representations and their asymptotical procedure. The prefundamental representations play a similar role in category $\BGG^{sh}$ as the fundamental representations do in the subcategory of finite-dimensional representations of the ordinary Yangian \cite{CP0}, hence the terminology. 

\smallskip

The trigonometric analogs of shifted Yangians, namely the shifted quantum affine algebras, are other examples of shifted quantum groups. These algebras $U_q^\mu(\hat{\mathfrak{g}})$ were introduced by Finkelberg-Tsymbaliuk \cite{FT} as variations of the quantum affine algebras $U_q(\hat{\mathfrak{g}})$, for a quantization parameter $q\in\mathbb{C}^*$ which is not a root of unity. The shifted quantum affine algebras also admit remarkable truncations which are closely related to $K$-theoretical Coulomb branches. The approach to the representation theory of $U_q^\mu(\hat{\mathfrak{g}})$ developed by the first author in \cite{H0} is based on the relations with representations in the category $\mathcal{O}$ of the Borel algebra $U_q(\hat{\mathfrak{b}})$ and on associated quantum integrable systems. For instance, the study of $R$-matrices and transfer-matrices of $U_q(\hat{\mathfrak{g}})$ allow to give a proof, uniform for any finite type, that simple finite-dimensional representations of shifted quantum affine algebras $U_q^\mu(\hat{\mathfrak{g}})$ descend to a truncation.  

\medskip

In this spirit, it is natural to raise the question of the construction of $R$-matrices for representations of shifted quantum groups. It is the problem we address in this paper and from which we obtain several applications. The Drinfeld-Jimbo coproduct is only conjecturally known for shifted quantum affine algebras, that is why we work with shifted Yangians.

\medskip

More precisely, we study morphisms in category $\BGG^{sh}$ of the form 
$$\check{R}_{V,W}: V\otimes W \rightarrow W\otimes V$$
for a pair $(V, W)$ of irreducible representations of shifted Yangians. 

To state our main results in a neat way, call an irreducible module {\it positive} (respectively, {\it negative}) if its highest weight is a tuple of (respectively, inverses of) monic polynomials. When the total degree of these polynomials is one, these are the prefundamental modules \cite{Z} mentionned above. Positive modules are one-dimensional, while negative modules are infinite-dimensional except in the trivial case. 

\medskip

\noindent {\bf Construction of R-matrices.} Let $P$ be a positive module and $N$ be a negative module. Let $V$ be an {\it arbitrary} irreducible module in category $\BGG^{sh}$. As our first main result, we prove the following cyclicity and cocyclicity properties (Theorem \ref{thm: cyclicity  cocyclicity}):
\begin{itemize}
\item[(i)] The modules $P \otimes V$ and $V \otimes N$ are generated by tensor products of highest weight vectors. The modules $V \otimes P$ and $N \otimes V$ are cogenerated by tensor products of highest weight vectors.
\end{itemize} 
Here a module is cogenerated by a vector if this vector is contained in all nonzero submodules.  As a consequence, we obtain unique module morphisms sending a tensor product of highest weight vectors to the opposite tensor product, 
$$ \check{R}_{P,V}(a): P(a) \otimes V \longrightarrow V \otimes P(a)\quad \text{and}\quad \check{R}_{V,N}(a): V(a) \otimes N \longrightarrow N \otimes V(a).  $$
Another consequence of cyclicity property is that the $R$-matrices $\check{R}_{P,V}(a)$ and $\check{R}_{V,N}(a)$, viewed as vector-valued functions of $a$, are polynomial. 

Our cyclicity property differs from the case of finite-dimensional irreducible representations of ordinary quantum affine algebras and Yangians where cyclicity holds true for {\it generic} spectral parameters. Indeed, in the non-shifted case the failure of cyclicity is controlled by the poles of normalized $R$-matrices viewed as rational functions \cite{AK,FM,Kashiwara,H2,H5,GTLW} of the spectral parameter, which are rarely polynomial. 

\medskip

\noindent {\bf Properties of R-matrices.} The $R$-matrices being module morphisms, we are able to compute them for positive modules, much in the spirit of Jimbo \cite{Jimbo1}. 
It turns out that they are connected to the Gerasimov--Kharchev--Lebedev--Oblezin truncation series \cite{GKLO}, GKLO series for short, which are certain generating series of the shifted Yangians appearing in the definition of truncated shifted Yangians \cite{BK, KWWY, BFN}. Note that in the trigonometric case, the truncation series defining truncated shifted quantum 
affine algebras were related to limits of transfer-matrices associated to positive 
prefundamental representations (the $Q$-operators) in \cite{H0}.

Since a positive module $P$ is one-dimensional, we view $\check{R}_{P,V}(a)$ as a linear operator on $V$. We establish the following property in the \lq\lq positive case" :
\begin{itemize}
\item[(ii)] For $P$ a positive prefundamental module, the vector-valued polynomial function $a \mapsto \check{R}_{P,V}(a)$ from $\BC$ to $\mathrm{End} (V)$ satisfies an additive difference equation determined by the action 
of a GKLO series; Equation \eqref{equ: difference}.
\end{itemize}
In the case of finite-dimensional representations of the ordinary Yangian, there is a general construction of $R$-matrices by solving additive difference equations \cite{GTL,GTL1,GTLW}. The point (ii) can be seen as a reverse statement : first the $R$-matrices are shown to exist and then we find difference equations for them.

\smallskip

In the \lq\lq negative case", the following is the key technical result of this paper. 
\begin{itemize}
\item[(iii)] Let $V$ be a fundamental module equipped with a weight basis, and view $\check{R}_{V,N}(a)$ as a matrix whose entries are vector-valued polynomial functions from $\BC$ to $\mathrm{End} (N)$. Then the diagonal entry associated to the lowest weight basis vector of $V$ is the action of a GKLO series; Equation \eqref{equ: GKLO vs lowest diagonal entry}.
\end{itemize}
Here a fundamental module \cite{CP0} is a finite-dimensional irreducible module over the ordinary Yangian whose associated Drinfeld polynomials are of total degree one (as reminded above, it should not be confused with a prefundamental representation).

\medskip

\noindent {\bf Application I: truncation of irreducible modules.} We obtain a proof, uniform for any finite type, that any irreducible module in $\mathcal{O}^{sh}$ factorizes through a truncated shifted Yangian.
For $\Glie$ of simply-laced type, this can be derived from \cite{KTWWY0,KTWWY}, and then extended to non simply-laced types by \cite{NW, Nak2} where the classification for non simply-laced truncated shifted Yangians is reduced to the known classification in simply-laced types via geometric arguments. 
In the case of shifted quantum affine algebras, the result was established for finite-dimensional irreducible modules \cite[Theorem 12.9]{H0} by a method involving transfer-matrices of quantum integrable systems. 

We prove furthermore that if $\Glie$ is not of type $E_8$, then any highest $\ell$-weight module in category $\BGG^{sh}$ descends to a truncation, by realizing such a module as a quotient of the tensor product of a positive module with a negative module (Theorem \ref{thm: Weyl standard}).
 
\smallskip 
 
\noindent {\bf Application II: Jordan--H\"older property.} As another application, we prove that in category $\BGG^{sh}$ the tensor product of two (and hence finitely many) irreducible modules admits a finite Jordan--H\"older filtration. In other words, the full subcategory of $\mathcal{O}^{sh}$ consisting of modules with a finite Jordan--H\"older filtration is closed under tensor product. This seems surprising, at least to us, as the analog  category $\mathcal{O}$ for the Borel subalgebra \cite{HJ} does not satisfy this property.
  
\medskip

We expect the $R$-matrices introduced in this paper for shifted quantum groups will be further studied in the future, keeping in mind their importance for ordinary quantum groups. Also, an approach to $R$-matrices using algebraic versions of Maulik-Okounkov stable maps for the category $\mathcal{O}$ of Borel subalgebras was proposed in \cite{H1}. As this category $\mathcal{O}$ is closely related to the category $\mathcal{O}^{sh}$ for shifted quantum affine algebras, we expect such algebraic stable maps also to exist for shifted quantum groups. We also expect more applications of these $R$-matrices for the representation theory of shifted quantum groups, in particular to get advances on the Conjecture in \cite{H0} which states a parameterization of irreducible representations of non-simply laced truncated shifted quantum affine algebras in terms of Langlands dual $q$-characters.

 \medskip

The paper is organized as follows.

\smallskip

In Section \ref{secdeux}, we recall the basic properties of shifted Yangians and of the truncated shifted Yangians.  We also give a first estimation of the coproduct (Lemma \ref{lem: coproduct estimation}). 

In Section \ref{sec: rep shif}, we recall basic properties of representations of shifted Yangians, including the existence of Verma modules, the parameterization of irreducible modules in the category $\BGG^{sh}$, finite-dimensional irreducible modules, $q$-characters. 

In Section \ref{secquatre} we prove cyclicity and co-cyclicity properties for tensor products of prefundamental representations in category $\BGG^{sh}$
(Theorem \ref{thm: cyclicity  cocyclicity}),  which motivate our definitions of Weyl modules and standard modules (Definition \ref{defi: Weyl standard modules}). We describe these  modules when $\Glie$ is not of type $E_8$ (Theorem \ref{thm: Weyl standard}). 

In Section \ref{sec: one-dim R-matrix} we construct the R-matrices for suitable highest $\ell$-weight modules and establish their first properties in Theorem \ref{thm: R-matrix}, Propositions \ref{prop: polynomiality R-} and \ref{prop: one-dim R+}. We also get several results on the eigenvalues of certain of these $R$-matrices (Proposition \ref{prop: l-weight R-matrix+}).

In Section \ref{sec: sl2}, we focus on the case $\Glie = sl_2$ for which we prove the existence and uniqueness of factorization for all irreducible modules in category $\BGG^{sh}$ into tensor products of prefundamental modules and Kirillov-Reshetikhin modules (Theorem \ref{rem: uniqueness}). 

In Section \ref{sec: pre R} we compute diagonal entries of certain remarkable R-matrices, by relating it to one-dimensional $R$-matrices (Proposition \ref{prop: highest diagonal entry} and Theorem \ref{thm: TQ pre R}). The proof uses a refined estimation of the coproduct that we establish (Lemma \ref{lem: coproduct estimation refine}).

In Section \ref{sectrunc} we prove uniformly that any standard module (and so any irreducible module) factorizes through a truncated shifted Yangian (Theorem \ref{thm: truncation standard}). 

In Section  \ref{sec: truncation}, we establish the Jordan--H\"older property 
of the category $\BGG^{sh}$ (Theorem \ref{thm: Jordan-Holder}). We also get a uniform proof that a truncated shifted Yangian has only a finite number of irreducible representations (Theorem \ref{thm: finite truncation}).

\section{Shifted Yangians}\label{secdeux}
In this section we recall the basic properties of shifted Yangians from \cite{BFN, coproduct, Kn}. 
We review their definition, their standard gradings and their triangular decomposition. We recall the 
shift homomorphism and the coproduct for which we give an estimation (Lemma \ref{lem: coproduct estimation}). We also discuss the particular case of the ordinary Yangian as well as 
certain remarkable quotients, the truncated shifted Yangians.

\subsection{Definition and structure}
Fix $\Glie$ to be a complex finite-dimensional simple Lie algebra. Set $\BN := \BZ_{\geq 0}$. Let $\Hlie$ be a Cartan subalgebra of $\Glie$, and $I := \{1,2,\cdots,r\}$ be the set of Dynkin nodes. The dual space $\Hlie^*$ admits a basis of {\it simple roots} $(\alpha_i)_{i \in I}$ and a non-degenerate symmetric bilinear form $(,): \Hlie^* \times \Hlie^* \longrightarrow \BC$. For $i,j \in I$ set 
$$c_{ij} := \frac{2(\alpha_i,\alpha_j)}{(\alpha_i,\alpha_i)} \in \BZ,\quad d_{ij} := \frac{(\alpha_i,\alpha_j)}{2},\quad d_i := d_{ii}. $$
We assume that the $d_i \in \BZ_{>0}$ are coprime. For $i \in I$, the {\it fundamental weight} $\varpi_i \in \Hlie^*$, the {\it fundamental coweight} $\varpi_i^{\vee} \in \Hlie$, and the {\it simple coroot} $\alpha_i^{\vee} \in \Hlie$ are determined by the following equations for $j \in I$:
$$(\varpi_i, \alpha_j) = d_i\delta_{ij},\quad \langle \varpi_i^{\vee}, \alpha_j\rangle = \delta_{ij},\quad \langle \alpha_i^{\vee}, \alpha_j\rangle = c_{ij}, $$
where $\langle, \rangle: \Hlie \times \Hlie^* \longrightarrow \BC$ denotes the evaluation map.
 We shall need the coweight lattice, the root lattice, and some of their subsets defined as follows:
 \begin{align*}
 \mathrm{coweight\ lattice}\quad \BP^{\vee} &:= \bigoplus_{i\in I} \BZ \varpi_i^{\vee} \subset \Hlie,\quad \BQ_+^{\vee} := \bigoplus_{i\in I} \BN \alpha_i^{\vee};  \\
 \mathrm{root\ lattice}\quad \BQ &:= \bigoplus_{i\in I} \BZ \alpha_i \subset \Hlie^*,\quad \BQ_+ := \bigoplus_{i\in I} \BN \alpha_i,\quad \BQ_- := - \BQ_+.
 \end{align*}
A coweight means an element of the coweight lattice $\BP^{\vee}$. It is dominant (respectively antidominant) if all the coefficients of $\varpi_i^{\vee}$ belong to $\BN$ (respectively $- \BN$). On the other hand, a weight means an element of the dual space $\Hlie^*$. The weight lattice does not play any role in this paper. We let $\varpi$ denote fundamental weights, and reserve $\omega$ for highest $\ell$-weight vectors.
 
 For a coweight $\mu \in \BP^{\vee}$, the shifted Yangian $Y_{\mu}(\Glie)$ is the algebra with generators
 $$ x_{i,n}^{\pm},\quad \xi_{i,p} \quad \mathrm{for}\ (i, n, p) \in I \times \BN \times \BZ $$
 called Drinfeld generators, subject to the following relations: 
 \begin{gather}
[\xi_{i,p}, \xi_{j,q}] = 0, \quad [x_{i,m}^+,x_{j,n}^-] = \delta_{ij} \xi_{i,m+n}, \label{rel: Cartan}  \\
[\xi_{i,p+1}, x_{j,n}^{\pm}] - [\xi_{i,p}, x_{j,n+1}^{\pm}] = \pm d_{ij} (\xi_{i,p}x_{j,n}^{\pm} + x_{j,n}^{\pm} \xi_{i,p}), \label{rel: Cartan-Drinfeld}  \\
[x_{i,m+1}^{\pm}, x_{j,n}^{\pm}] - [x_{i,m}^{\pm}, x_{j,n+1}^{\pm}] = \pm d_{ij} (x_{i,m}^{\pm}x_{j,n}^{\pm} + x_{j,n}^{\pm} x_{i,m}^{\pm}), \label{rel: Drinfeld} \\
\mathrm{ad}_{x_{i,0}^{\pm}}^{1-c_{ij}} (x_{j,0}^{\pm}) = 0 \quad \mathrm{if}\ i \neq j,  \label{rel: Serre} \\
\xi_{i,-\langle \mu, \alpha_i \rangle - 1} = 1,\quad \xi_{i,p} = 0 \quad \mathrm{for}\ p < -\langle \mu, \alpha_i \rangle - 1.  \label{rel: shift}
\end{gather}
Here $\mathrm{ad}_x(y) := xy - yx$. Define the generating series for $i \in I$:
 \begin{equation}  \label{def: generators currents}
x_i^{\pm}(u) :=  \sum_{n\in \BN} x_{i,n}^{\pm} u^{-n-1},\quad \xi_i(u) :=  \sum_{p\in \BZ} \xi_{i,p} u^{-p-1} \in Y_{\mu}(\Glie)((u^{-1})).
\end{equation}
These are Laurent series in $u^{-1}$, with leading terms $x_{i,0}^{\pm} u^{-1}$ and $u^{\langle \mu, \alpha_i \rangle}$. 

\begin{rem}
Our generators $x_{i,n}^+, x_{i,n}^-$ and $\xi_{i,p}$ correspond to $E_i^{(n+1)}, F_i^{(n+1)}$ and $H_i^{(n+1)}$ in \cite{BFN, coproduct}. The zero-shifted Yangian $Y_0(\Glie)$ is the ordinary Yangian \cite{Dr} with deformation parameter $\hbar = 1$, which will also be denoted by $Y(\Glie)$.
\end{rem}

The algebra $Y_{\mu}(\Glie)$ is $\BQ$-graded, called {\it weight grading}, by declaring the weights of the generators $x_{i,n}^+, x_{i,n}^-$ and $\xi_{i,p}$ to be $\alpha_i, -\alpha_i$ and $0$. For $\beta \in \BQ$, let $Y_{\mu}(\Glie)_{\beta}$ denote the subspace of elements of weight $\beta$.
Setting $p = -\langle \mu, \alpha_i\rangle - 1$ in Eq.\eqref{rel: Cartan-Drinfeld} and noticing $\xi_{i,p} = 1$, one obtains the following Cartan relation
\begin{equation}  \label{rel: weight grading}
 [\xi_{i,-\langle \mu, \alpha_i\rangle}, x_{j,n}^{\pm}] = \pm (\alpha_i,\alpha_j) x_{j,n}^{\pm}.  
\end{equation}
So the weight grading is characterized alternatively: $x \in Y_{\mu}(\Glie)$ is of weight $\beta$ if and only if $[\xi_{i,-\langle \mu, \alpha_i\rangle}, x] = (\alpha_i,\beta) x$ for all $i \in I$. 

As in the case of ordinary Yangian (see for example \cite[\S 2.8]{GTL}), for $a \in \BC$ there is an algebra automorphism $\tau_a: Y_{\mu}(\Glie) \longrightarrow Y_{\mu}(\Glie)$ defined by
\begin{gather}  \label{rel: spectral shift}
 \tau_a: Y_{\mu}(\Glie) \longrightarrow Y_{\mu}(\Glie),\quad  x_j^{\pm}(u) \mapsto x_j^{\pm}(u-a),\quad \xi_j(u) \mapsto \xi_j(u-a).
\end{gather}
Note that $\tau_a \circ \tau_b = \tau_{a+b}$ for $a,b \in \BC$ and $\tau_0 = \mathrm{Id}$. Indeed $\tau_a$ can be obtained from the evaluation at $z=a$ of the following algebra homomorphism:
\begin{gather}  \label{rel: spectral shift formal}
    \tau_z: Y_{\mu}(\Glie) \longrightarrow Y_{\mu}(\Glie) \otimes \BC[z],\quad X_p \mapsto \sum_{n\in \BN} \binom{p}{n}  X_{p-n} \otimes z^n
\end{gather}
where $X \in \{x_i^{\pm}, \xi_i\}$ and $p \in \BZ$. It is understood that $x_{i,k}^{\pm} = 0$ for $k < 0$.

For $\zeta, \eta$ antidominant coweights, the following map extends uniquely to an algebra morphism $\iota_{\mu,\zeta,\eta}: Y_{\mu}(\Glie) \longrightarrow Y_{\mu+\zeta+\eta}(\Glie)$, called the {\it shift homomorphism}:
\begin{gather} \label{rel: shift map}
x_{i,n}^+ \mapsto x_{i,n-\langle\zeta,\alpha_i\rangle}^+, \quad x_{i,n}^- \mapsto x_{i,n-\langle\eta,\alpha_i\rangle}^-,\quad \xi_{i,p} \mapsto \xi_{i,p-\langle\zeta+\eta,\alpha_i\rangle}.
\end{gather}

The shifted Yangian $Y_{\mu}(\Glie)$ admits a {\it triangular decomposition}. Let us define three subalgebras by generating subsets. The first $Y_{\mu}^<(\Glie)$ is generated by the $x_{i,n}^-$, the second $Y_{\mu}^>(\Glie)$ generated by the $x_{i,n}^+$, and the third $Y_{\mu}^=(\Glie)$ generated by the $\xi_{i,p}$.
These subalgebras inherit from $Y_{\mu}(\Glie)$ the weight grading, the first two are graded by $\BQ_{\mp}$, and the third $\{0\}$. Set $$Y_{\mu}^+(\Glie) := Y_{\mu}^=(\Glie) Y_{\mu}^>(\Glie)\text{ and }Y_{\mu}^-(\Glie) := Y_{\mu}^<(\Glie) Y_{\mu}^=(\Glie);$$ these are subalgebras. The following result is a consequence of the PBW basis theorem \cite[Corollary 3.15]{coproduct}. 

\begin{theorem}\cite{coproduct} \label{thm: triangular decomposition}
All shift homomorphisms are injective. The multiplication map $Y_{\mu}^<(\Glie) \otimes Y_{\mu}^=(\Glie) \otimes Y_{\mu}^>(\Glie)  \longrightarrow Y_{\mu}(\Glie)$
is an isomorphism of vector spaces. $Y_{\mu}^{\pm}(\Glie)$ is the algebra generated by $x_{i,n}^{\pm}$ and $\xi_{i,p}$ for $(i, n,p) \in I \times \BN \times \BZ $ subject to Eq.\eqref{rel: shift}, the first half of Eq.\eqref{rel: Cartan}, and the $\pm$ part of Eqs.\eqref{rel: Cartan-Drinfeld}--\eqref{rel: Serre}.
\end{theorem}

It follows that the assignments $x_i^{\pm}(u) \mapsto x_i^{\pm}(u)$ and $\xi_i(u) \mapsto u^{-\langle \mu, \alpha_i \rangle} \xi_i(u)$ extend uniquely to four algebra isomorphisms 
\begin{gather}  \label{equ: isomorphism}
Y_0^>(\Glie) \cong Y_{\mu}^>(\Glie),\quad Y_0^<(\Glie) \cong Y_{\mu}^<(\Glie), \quad Y_0^{\pm}(\Glie) \cong Y_{\mu}^{\pm}(\Glie).
\end{gather}
The first two isomorphisms being independent of $\mu$, we omit $\mu$ from $Y_{\mu}^>(\Glie)$ and $Y_{\mu}^<(\Glie)$ when no confusion arises.

The next property of shifted Yangians is the {\it coproduct} of Drinfeld--Jimbo, which plays a key role in our study of representations.
We rephrase \cite[Theorem 4.8, Theorem 4.12, Proposition 4.14]{coproduct}. While \cite{coproduct} considered simply-laced types, the proofs apply in general as commented in \cite[Remark 3.2]{coproduct}.

\begin{theorem}\cite{coproduct}  \label{thm: coproduct}
There exists a unique family of algebra homomorphisms $$\Delta_{\mu,\nu}: Y_{\mu+\nu}(\Glie) \longrightarrow Y_{\mu}(\Glie) \otimes Y_{\nu}(\Glie)$$ 
for all coweights $\mu, \nu$ such that $\Delta_{0,0}$ is the coproduct of the ordinary Yangian and properties (i)--(ii) hold true.
\begin{itemize}
\item[(i)] For $\mu$ and $\nu$ antidominant, $i \in I,\ n < - \langle \mu,\alpha_i\rangle$ and $m < -\langle \nu,\alpha_i\rangle$:
\begin{gather} \label{rel: special coproduct}
\Delta_{\mu,\nu}(x_{i,n}^+) = x_{i,n}^+ \otimes 1,\quad \Delta_{\mu,\nu}(x_{i,m}^-) = 1 \otimes x_{i,m}^-.
\end{gather}
\item[(ii)] If $\zeta$ and $\eta$ are antidominant, then the following diagram commutes:
\begin{gather} \label{rel: coproduct vs shift}
 \xymatrixcolsep{6pc} \xymatrix{
Y_{\mu+\nu}(\Glie) \ar[d]^{\iota_{\mu+\nu,\zeta,\eta}} \ar[r]^{\Delta_{\mu,\nu}} & Y_{\mu}(\Glie) \otimes Y_{\nu}(\Glie) \ar[d]^{\iota_{\mu,\zeta,0} \otimes \iota_{\nu,0,\eta}} \\
Y_{\mu+\nu+\zeta+\eta}(\Glie) \ar[r]^{\Delta_{\mu+\zeta,\nu+\eta}}          & Y_{\mu+\zeta}(\Glie) \otimes Y_{\nu+\eta}(\Glie) } 
\end{gather}
\end{itemize}
Furthermore, if $\nu$ is antidominant, then the following diagram commutes:
\begin{gather} \label{rel: coproduct coass}
 \xymatrixcolsep{6pc} \xymatrix{
Y_{\mu+\nu+\rho}(\Glie) \ar[d]^{\Delta_{\mu,\nu+\rho}} \ar[r]^{\Delta_{\mu+\nu,\rho}} & Y_{\mu+\nu}(\Glie) \otimes Y_{\rho}(\Glie) \ar[d]^{\Delta_{\mu,\nu} \otimes \mathrm{Id}} \\
Y_{\mu}(\Glie) \otimes Y_{\nu+\rho}(\Glie) \ar[r]^{\mathrm{Id} \otimes \Delta_{\nu,\rho}}          & Y_{\mu}(\Glie) \otimes Y_{\nu}(\Glie) \otimes Y_{\rho}(\Glie) } 
\end{gather}
\end{theorem}

\subsection{The ordinary Yangian}  \label{ss: ordinary Yangian}
The ordinary Yangian $Y(\Glie)$ endowed with the coproduct $\Delta_{0,0} =: \Delta$ is a Hopf algebra, which contains the universal enveloping algebra $U(\Glie)$ as a Hopf subalgebra by identifying the $x_{i,0}^{\pm}$ with root vectors in the Lie algebra $\Glie$ associated to the roots $\pm \alpha_i$ so that $\alpha_i^{\vee} = \frac{1}{d_i} \xi_{i,0} \in \Hlie$. Let $R$ denote the set of positive roots of $\Glie$. One can extend $x_{i,0}^{\pm} =: x_{\alpha_i}^{\pm}$ to root vectors $x_{\gamma}^{\pm} \in \Glie_{\pm \gamma}$ for $\gamma \in R$ suitably normalized with respect to an invariant bilinear form of $\Glie$. Then the coproduct is determined by \cite{Dr} (see \cite[\S 4.2]{GNW} for a proof)
\begin{gather}  \label{eq: coproduct Yangian}
\Delta(\xi_{i,1}) = \xi_{i,1} \otimes 1 + 1 \otimes \xi_{i,1} + \xi_{i,0} \otimes \xi_{i,0} - \sum_{\gamma \in R} (\alpha_i,\gamma) x_{\gamma}^- \otimes x_{\gamma}^+.
\end{gather}
For $(\gamma, n) \in R \times \BN$, the root vector $x_{\gamma}^- \in \Glie_{-\gamma}$ proportional to an iterated commutator of the $x_{i,0}^-$, we choose exactly one of the $x_{i,0}^-$ in the commutator and replace it with $x_{i,n}^-$. It depends on the choice of $i$ and the position of $x_{i,0}^-$ in the commutator. We fix such a choice for all $(\gamma, n) \in R \times \BN$, and let $x_{\gamma,n}^-$ denote the resulting element in $Y^<(\Glie)$, called a {\it PBW variable} as in \cite[\S 3.12]{coproduct}.

\begin{rem}\label{rem: PBW}
For $\mu$ a coweight let us identify the subalgebra $Y_{\mu}^<(\Glie)$ with $Y^<(\Glie)$ via \eqref{equ: isomorphism} so that the PBW variables make sense for $Y_{\mu}^<(\Glie)$. By \cite[\S 3.12]{coproduct}, with respect to a total order on the set of PBW variables which is in natural bijection with $R \times \BN$, the ordered monomials in the PBW variables form a basis of $Y_{\mu}^<(\Glie)$. 
\end{rem}

We shall need the current algebra $\Glie[t]$; the Lie algebra $\Glie \otimes \BC[t]$ with bracket $$[x \otimes t^m, y \otimes t^n] = [x,y] \otimes t^{m+n}\quad \mathrm{for}\ x, y \in \Glie\quad \mathrm{and}\ m,n \in \BN. $$
It is bigraded by $\BN \times \BQ$: the weight grading comes from the adjoint action of $\Hlie \subset \Glie \subset \Glie[t]$; the $\BN$-grading is defined by declaring $x \otimes t^m$, for $x \in \Glie$ and $m \in \BN$, to be of degree $m$. Its universal enveloping algebra  $U(\Glie[t])$ is bi-graded by $\BN \times \BQ$.

The algebra $Y(\Glie)$ is $\BN$-filtrated, by declaring $Y(\Glie)^{\leq n}$, for $n \in \BN$, to be the linear subspace spanned by the monomials in the generators $x_{i,m}^{\pm}, \xi_{i,m}$ for which the sum of the indexes $m$ is at most $n$. The $\BN$-filtration on $Y(\Glie)$ is compatible with the weight grading, so that the associated grading $\mathrm{gr}_{\BN} Y(\Glie)$ is an algebra bigraded by $\BN \times \BQ$. Here by definition $Y(\Glie)^{\leq -1} = \{0\}$ and
$$ \mathrm{gr}_{\BN} Y(\Glie) := \bigoplus_{n\in \BN} Y(\Glie)^{\leq n}/Y(\Glie)^{\leq n-1}. $$
We have an isomorphism of $(\BN, \BQ)$-bigraded algebras 
\begin{gather}  \label{iso: filtration}
U(\Glie[t]) \longrightarrow \mathrm{gr}_{\BN} Y(\Glie),  \quad x_{\alpha_i}^{\pm} \otimes t^m \mapsto \overline{x_{i,m}^{\pm}}.
\end{gather}
Here $\overline{x_{i,m}^{\pm}}$  denote the images of $x_{i,m}^{\pm}$ under the projection 
$$Y(\Glie)^{\leq m} \longrightarrow Y(\Glie)^{\leq m} /Y(\Glie)^{\leq m-1} \subset \mathrm{gr}_{\BN} Y(\Glie).$$
This isomorphism appeared in \cite{PBW}. For a complete proof, see \cite[Theorem B.2]{FT1}.

\subsection{First coproduct estimation} 
A compact formula for the coproduct of the Drinfeld generators is unknown beyond $sl_2$ in \cite[\S 6.3]{coproduct}. Still some partial information on the weight space projection of coproduct is sufficient. 

Let us define the {\it height function} $h: \BQ_+ \longrightarrow \BN$ to be the additive function such that $h(\alpha_i) = 1$ for $i \in I$. In the following, when we write $h(\beta)$ or speak of height of a weight $\beta$, it is understood that $\beta \in \BQ_+$.

We shall need the notion of {\it principal part}. For $V$ a vector space, let $V[[u,u^{-1}]]$ denote the space of formal power series with coefficients in $V$, and $V[[u^{-1}]]$ its subspace of power series in $u^{-1}$. The principal part of a formal power series $f(u)$ in $V[[u,u^{-1}]]$, denoted by $\langle f(u) \rangle_+$, is a power series in $V[[u^{-1}]]$ defined by
$$ \left\langle \sum_{p\in \BZ} f_p u^{-p-1} \right\rangle_+ := \sum_{p\in \BN} f_p u^{-p-1}. $$
 This was denoted by $\underline{f(u)}$ in \cite[Lemma 5.13]{KTWWY0}. As an example,
$$ \left\langle \frac{g(u)}{u-a}\right\rangle_+ = \frac{g(a)}{u-a} \quad \mathrm{for}\ g(u) \in \BC[u]\ \mathrm{and}\ a \in \BC. $$  
As another example, the shift homomorphism of \eqref{rel: shift map} acts on $x_i^{\pm}(u)$ as:
$$ x_i^+(u) \mapsto \langle u^{-\langle \zeta,\alpha_i \rangle} x_i^+(u) \rangle_+,\quad x_i^-(u)  \mapsto \langle u^{-\langle \eta,\alpha_i \rangle} x_i^-(u) \rangle_+.  $$
Multiplication endows $V[[u,u^{-1}]]$ with a module structure over the polynomial algebra $\BC[u]$. It is compatible with taking principal part:
\begin{equation}  \label{asso: truncation}
\langle g \langle h f \rangle_+ \rangle_+ = \langle g h f\rangle_+ \quad \mathrm{for}\ g, h \in \BC[u]\ \mathrm{and}\ f \in V[[u,u^{-1}]].
\end{equation}

\begin{lem} \label{lem: coproduct estimation}
For all coweights $\mu$ and $\nu$, the coproduct $\Delta_{\mu,\nu}$ satisfies:
\begin{gather*}
\Delta_{\mu,\nu}(x_i^+(u)) \equiv  x_i^+(u) \otimes 1 +  \langle\xi_i(u) \otimes x_i^+(u) \rangle_+\ \mathrm{mod}.\sum_{h(\beta) > 0} Y_{\mu}^-(\Glie)_{-\beta} \otimes Y_{\nu}^+(\Glie)_{\beta+\alpha_i},  \\
\Delta_{\mu,\nu}(x_i^-(u)) \equiv 1 \otimes x_i^-(u) +  \langle x_i^-(u) \otimes \xi_i(u) \rangle_+\ \mathrm{mod}. \sum_{h(\beta) > 0} Y_{\mu}^-(\Glie)_{-\beta-\alpha_i} \otimes Y_{\nu}^+(\Glie)_{\beta}, \\
\Delta_{\mu,\nu}(\xi_i(u)) \equiv \xi_i(u) \otimes \xi_i(u) \ \mathrm{mod}. \sum_{h(\beta) > 0} Y_{\mu}^-(\Glie)_{-\beta} \otimes Y_{\nu}^+(\Glie)_{\beta}, \\
\Delta_{\mu,\nu}(\xi_{i,-\langle\mu+\nu, \alpha_i\rangle}) = \xi_{i,-\langle\mu, \alpha_i\rangle} \otimes 1 + 1 \otimes \xi_{i,-\langle\nu, \alpha_i\rangle}.
\end{gather*}
\end{lem}
In the first three relations, the notation \lq\lq$\mathrm{mod}. V$\rq\rq\ for $V \subset Y_{\mu}^-(\Glie) \otimes Y_{\nu}^+(\Glie)$ should be understood as \lq\lq modulo $V((u^{-1}))$\rq\rq. 
\begin{proof}
The case $\mu = \nu = 0$ follows from \cite[Lemma 1]{Kn}; while the statement of \cite{Kn} is weaker, its proof works in our situation; see also \cite[Prop.2.8]{CP0} and \cite[Prop.2.10]{GW}. For arbitrary coweights we use the zigzag arguments as in the proof of \cite[Theorem 4.12]{coproduct}. We shall treat only the first assertion, as the other three are parallel. Our goal is to prove the following relations, denoted by $P(n,\mu,\nu)$, for $n \in \BN$:
$$\Delta_{\mu,\nu}(x_{i,n}^+) \equiv x_{i,n}^+ \otimes 1 + \sum_{m\geq 0} \xi_{i,n-1-m} \otimes x_{i,m}^+\ \mathrm{mod}.\sum_{h(\beta) > 0} Y_{\mu}^-(\Glie)_{-\beta} \otimes Y_{\nu}^+(\Glie)_{\beta+\alpha_i}. $$
For $\zeta, \eta$ antidominant coweights, applying $\iota_{\mu,\zeta,0} \otimes \iota_{\nu,0,\eta}$ to $\Delta_{\mu,\nu}(x_{i,n}^+)$, from the commutative diagram \eqref{rel: coproduct vs shift} and injectivity of shift homomorphisms we deduce
$$P(n,\mu,\nu) \Longleftrightarrow P(n-\langle \zeta,\alpha_i\rangle, \mu+\zeta, \nu+\eta).$$
Let us choose $\zeta$ and $\eta$ such that $\mu + \zeta$ and $\nu + \eta$ are antidominant. 

If $n \geq -\langle \mu,\alpha_i\rangle$, then the above equivalence applied to $(n+\langle \mu,\alpha_i\rangle, 0, 0)$ and the antidominant coweights $\mu+\zeta, \nu + \eta$ gives 
$$P(n+\langle \mu,\alpha_i\rangle, 0, 0) \Longleftrightarrow P(n-\langle \zeta,\alpha_i\rangle, \mu+\zeta, \nu+\eta).$$
So $P(n,\mu,\nu) \Longleftrightarrow P(n+\langle \mu,\alpha_i\rangle, 0, 0)$, and the latter is true by \cite{Kn}. 

If $n < -\langle \mu, \alpha_i\rangle$, then $n  -\langle \zeta,\alpha_i\rangle < - \langle \mu+\zeta, \alpha_i\rangle$ and by Equation \eqref{rel: special coproduct}, 
$$ \Delta_{\mu+\zeta, \nu+\eta}(x_{i,n-\langle \zeta, \alpha_i\rangle}^+) =  x_{i,n-\langle \zeta, \alpha_i\rangle}^+ \otimes 1. $$
In the commutative diagram \eqref{rel: coproduct vs shift} put $x_{i,n}^+$ at the top-left corner. Then the element at the bottom-right corner is $x_{i,n-\langle \zeta, \alpha_i\rangle}^+ \otimes 1$. From the injectivity of the vertical maps we obtain that the element at the top-right corner is $x_{i,n}^+ \otimes 1$, namely, 
 $\Delta_{\mu,\nu}(x_{i,n}^+) = x_{i,n}^+ \otimes 1$.  For $m \geq 0$, we have $\xi_{i,n-m-1} = 0$ in $Y_{\mu}(\Glie)$ because 
 $$n-m-1 < -\langle \mu, \alpha_i\rangle - m -1 \leq -\langle \mu, \alpha_i\rangle - 1.$$
So the summation $\sum_m$ in $P(n,\mu,\nu)$ vanishes. This proves $P(n,\mu,\nu)$, with the second summation $\sum_{\beta}$ being zero.
\end{proof}
The zigzag arguments will reappear in this paper at the level of representations.

\subsection{Truncated shifted Yangians} These are quotients of shifted Yangians appearing first in type A \cite{BK} as finite W-algebras, in dominant case \cite{KWWY} as quantizations of slices in affine Grassmannians, and in the most general case \cite{BFN} as quantized Coulomb branches. Their definition involves the notion of $\ell$-weight.

To motivate Eq.\eqref{def: weight coweight} below, let $\mu$ be a coweight and $V$ be a $Y_{\mu}(\Glie)$-module. Then the actions of the $\xi_{i,p}$ on $V$ mutually commute. Suppose that $0 \neq v \in V$ is a common eigenvector with $e_{i,p}$ being the eigenvalue of $\xi_{i,p}$. We have
$$ \xi_i(u) v = e_i(u) v, \quad e_i(u) := \sum_{p\in \BZ} e_{i,p} u^{-p-1}. $$
From Eq.\eqref{rel: shift} we see that $e_i(u)$ is a Laurent series in $u^{-1}$ whose leading term is fixed to be $u^{\langle \mu, \alpha_i\rangle}$. So the {\it coweight} $\mu$ can be recovered from the $I$-tuple of Laurent series $(e_i(u))_{i\in I}$. The actions of the $\xi_{i,-\langle \mu, \alpha_i\rangle}$ on $v$ are encoded in the {\it weight} $\beta \in \Hlie^*$ defined below in the same way as Eq.\eqref{rel: weight grading}:
$$  \xi_{i,-\langle \mu, \alpha_i\rangle} v = (\alpha_i, \beta) v \quad \mathrm{where}\ \beta :=  \sum_{i\in I} e_{i,-\langle\mu,\alpha_i\rangle}  \frac{1}{d_i} \varpi_i.  $$
In notations of Subsection \ref{ss: Verma}, the vector $v$ is of $\ell$-weight $(e_i(u))_{i\in I}$ and weight $\beta$.

Consider the multiplicative group $\BC((u^{-1}))^{\times}$ of the field $\BC((u^{-1}))$ of Laurent series in $u^{-1}$. The set of $\ell$-weights, denoted by $\CL$, is the subset of the $I$-fold product group $\prod_{i \in I} \BC((u^{-1}))^{\times}$ consisting of $I$-tuples of Laurent series in $u^{-1}$ whose leading terms are of the form $u^k$ for $k \in \BZ$; it is clearly a subgroup. For $\Be \in \CL$ and $i \in I$, let $\Be_i(u)$ be the $i$-th component of $\Be$, and let $\Be_{i,p} \in \BC$, for $p \in \BZ$, be the coefficient of $u^{-p-1}$ in $\Be_i(u)$. By definition there exists a unique $k_i \in \BZ$ such that $$\Be_i(u) = \sum_{p\in \BZ} \Be_{i,p} u^{-p-1}\text{ with }\Be_{i,p} = 0\text { for $p<-k_i-1$} \text{ and }\Be_{i,-k_i-1} = 1.$$
Define the {\it weight} and the {\it coweight} of $\Be$ by
\begin{gather} \label{def: weight coweight}
\varpi(\Be) := \sum_{i\in I} \frac{\Be_{i,-k_i}}{d_i} \varpi_i \in \Hlie^*,\quad \varpi^{\vee}(\Be) := \sum_{i\in I} k_i \varpi_i^{\vee} \in \BP^{\vee}.
\end{gather}
This defines two morphisms of abelian groups
$\varpi: \CL \longrightarrow \Hlie^*$ and $\varpi^{\vee}: \CL  \longrightarrow \BP^{\vee}$.

\begin{defi}  \label{def: truncatable pair}
A pair $(\mu, \Br)$ of coweight $\mu$ and $\ell$-weight $\Br$ {\it truncatable} if
\begin{equation}  \label{equ: integrality}
\varpi^{\vee}(\Br) - \mu \in \BQ_+^{\vee}.
\end{equation}
\end{defi}
In this situation, let $m_i \in \BN$, for $i\in I$, be the coefficient of $\alpha_i^{\vee}$ in $\varpi^{\vee}(\Br) - \mu$. We have the Gerasimov--Kharchev--Lebedev--Oblezin (GKLO for short) series $A_i(u)$, a $Y_{\mu}^=(\Glie)$-valued Laurent series in $u^{-1}$ of leading term $u^{m_i}$ for $i\in I$, uniquely determined by the following equations  \cite[Lemma 2.1 with $\iota \hbar = 1$]{GKLO} (see also \cite[\S 4A]{KWWY} and \cite[(B.14)]{BFN}):
\begin{equation}  \label{equ: truncation}
\begin{split}
\xi_i(u) &=   \frac{\Br_i(u)}{A_i(u)A_i(u-d_i)} \prod_{j: c_{ji} < 0} \prod_{t=1}^{-c_{ji}} A_j(u-d_{ij}-t d_j) \\
           & = \frac{\Br_i(u)}{A_i(u) A_i(u-d_i)}  \prod_{j: c_{ji} = -1} A_j(u-\frac{1}{2} d_j)\prod_{j: c_{ji} = -2} A_j(u) A_j(u-1)  \\
 & \quad\quad \times  \prod_{j: c_{ji} = -3} A_j(u+\frac{1}{2}) A_j(u-\frac{1}{2}) A_j(u-\frac{3}{2}). 
\end{split}
\end{equation}
The second equation comes from the fact that $c_{ji} = -1$ implies $d_{ij} = -\frac{1}{2}d_j$, while $c_{ji} < -1$ implies $d_j = 1$ and $d_{ij} = \frac{1}{2} c_{ji}$. 

%Here we do not write the coefficients as $A_{i,p}$ because the latter will be reserved in Eq.\eqref{def: simple root} to be generalized simple roots. 
\begin{lem}\cite{GKLO}  \label{lem: affine coweight}
Let $(\mu, \Br)$ be truncatable. In the shifted Yangian $Y_{\mu}(\Glie)$ we have
\begin{equation*} 
A_i(u) x_{j,n}^- A_i(u)^{-1} =  x_{j,n}^- + d_i \delta_{ij} \sum_{k\geq 0} x_{i,n+k}^- u^{-k-1}.
\end{equation*}
\end{lem}
\begin{proof}
Write $A_i(u) = \sum_p a_{i,p} u^{-p-1}$. The following relation is a consequence of \cite[(2.12),(2.14)]{GKLO}; see \cite[Definition 4.1]{KTWWY} in simply-laced types:
\begin{gather*} 
  [a_{i,p+1}, x_{j,n}^-] - [a_{i,p}, x_{j,n+1}^-] = \delta_{ij} d_i x_{i,n}^- a_{i,p}.
\end{gather*}
Since $a_{i,p} = 0$ for $p <<0$, the above relation can be rewritten as
$$  [a_{i,p}, x_{j,n}^-] = d_i\delta_{ij}  \sum_{k\geq 0} x_{i,n+k}^- a_{i,p-k-1}. $$
Multiplying the above equality by $u^{-p-1}$ and summing over $p \in \BZ$, we get
$$ A_i(u) x_{j,n}^- - x_{j,n}^- A_i(u) =  d_i\delta_{ij} \sum_{k\geq 0} x_{i,n+k}^- u^{-k-1} A_i(u). $$
Right multiplying by $A_i(u)^{-1}$ gives the desired identity of the lemma. 
\end{proof}

\begin{defi}\cite{KWWY, BFN}  \label{def: truncated shifted Yangians}
For $(\mu,\Br) \in \BP^{\vee} \times \CL$ a truncatable pair, the {\it truncated shifted Yangian} $Y_{\mu}^{\Br}(\Glie)$ is the algebra defined as the quotient of the shifted Yangian $Y_{\mu}(\Glie)$ by the two-sided ideal generated by the coefficients of $\langle A_i(u)\rangle_+$ for $i \in I$.
\end{defi}

\begin{rem}  \label{rem: truncated shifted Yangian}
Assume each $\Br_i(u)$ is a monic polynomial of $u$. Our algebra $Y_{\mu}^{\Br}(\Glie)$ and series $A_i(u)$ correspond to $\widetilde{Y}_{\mu}^{\nu}(\Br)$ and $u^{m_i} A_i(u)$ in \cite[\S 3.4]{KPW} with $\nu = \varpi^{\vee}(\Br)$. The original truncated shifted Yangian, denoted by $Y_{\mu}^{\nu}(\Br)$, is defined to be the image of $Y_{\mu}^{\Br}(\Glie)$ under the so-called GKLO representation by difference operators; see \cite[Theorem 4.5]{KWWY} and \cite[Theorem B.15]{BFN}. Conjecturally \cite{KWWY} the quotient map is an isomorphism, and a proof in type A is available in \cite[Theorem 1.6]{KMW} and \cite[Theorem A.5]{KPW}. 
The reason why we drop the polynomiality of $\Br$ will be given in Section \ref{sectrunc}.
\end{rem}
\section{Representations of shifted Yangians}  \label{sec: rep shif}
We recall basic properties of representations of shifted Yangians: Verma modules, classification of irreducible modules in the category $\BGG^{sh}$, $q$-characters, finite-dimensional irreducible modules, and prefundamental modules. 

 Most of the definitions and results in this section are well-known, and were also known for representation theory of three classes of algebras: the ordinary Yangian \cite{Dr,CP0, Kn, GTL}; the upper Borel subalgebra of the quantum affine algebra $U_q(\hat{\Glie})$ for $q \in \BC^{\times}$ generic \cite{HJ,FH,H1}, which we refer to as {\it Borel algebra}; the shifted quantum affine algebras recently developed in \cite{FT, H0}. 
Their proofs work for shifted Yangians as well because the algebraic structures are common for these quantum groups.

\subsection{Verma modules} \label{ss: Verma}
We begin with some general remarks on the notions of weights and $\ell$-weights for modules over shifted Yangians. 
Fix $\mu$ a coweight. Let $V$ be a module over $Y_{\mu}(\Glie)$. For $\beta \in \Hlie^*$ and $\Bf \in \CL$, define
\begin{gather*}
 V_{\beta} := \{v \in V \ |\ \forall\ i \in I,\ \xi_{i,-\langle\mu,\alpha_i \rangle} v = (\alpha_i,\beta) v \}, \\
 V_{\Bf} :=   \{v \in V \ |\ \forall\ (i,p) \in I \times \BZ,\ \exists\ m \in \BN\ \mathrm{such\ that}\ (\xi_{i,p} - \Bf_{i,p})^m v = 0  \}.
\end{gather*}
If $V_{\beta}$ is nonzero, then it is called a {\it weight space} of weight $\beta$, and a nonzero vector $v \in V_{\beta}$ is called a {\it weight vector} of weight $\wt(v) := \beta$. If $V_{\Bf}$ is nonzero, then it is an {\it $\ell$-weight space} of $\ell$-weight $\Bf$ and similar conventions for $\ell$-weight vector and $\lwt(v)$.  Let $\wt(V) \subseteq \Hlie^*$ be the set of weights of $V$, and $\lwt(V) \subseteq \CL$ the set of $\ell$-weights.  
By Eqs.\eqref{rel: shift} and \eqref{rel: weight grading}, for $\Bf \in \lwt(V),\ \beta \in \wt(V)$ and $\alpha \in \BQ$, we have
$$\mu = \varpi^{\vee}(\Bf),\quad Y_{\mu}(\Glie)_{\alpha} V_{\beta} \subseteq V_{\alpha+\beta}.  $$ 

\begin{rem} \label{rem: spectral weight}
While the automorphism $\tau_a$ from \eqref{rel: spectral shift} preserves the weight grading on $Y_{\mu}(\Glie)$, this is not the case for modules. Define the weight $\tilde{\mu} \in \Hlie^*$ associated to $\mu$ and, by abuse of language, the group automorphism $\tau_a: \CL \longrightarrow \CL$ by
$$ \tilde{\mu} := \sum_{i\in I} \langle \mu, \alpha_i\rangle \frac{1}{d_i} \varpi_i,\quad \tau_a(\Bf) := (\Bf_i(u-a))_{i\in I} \quad \mathrm{for}\ \Bf \in \CL.  $$
 For $V$ a $Y_{\mu}(\Glie)$-module, the pullback module $\tau_a^*V$ is denoted by $V(a)$, with $a$ referred to as the {\it spectral parameter}.  We have 
\begin{align*}
V_{\beta} = V(a)_{\beta-a \tilde{\mu}} \quad \mathrm{for}\ \beta \in \wt(V),\quad V_{\Bf} = V(a)_{\tau_a(\Bf)} \quad \mathrm{for}\ \Bf \in \lwt(V).
\end{align*}
\end{rem}

Call $V$ a {\it weight module} if it is a direct sum of weight spaces. In such a module, any $\ell$-weight space $V_{\Bf}$ is contained in the weight space $V_{\varpi(\Bf)}$. We shall say that $V$ is weight graded by a subset $X \subset \Hlie^*$ if $V$ is a weight module and $\wt(V) \subset X$.

Call $V$ {\it top graded} if there exists $\lambda \in \Hlie^*$ such that $V$ is weigh graded by $\lambda+\BQ_-$ and $V_{\lambda}$ is one-dimensional. Clearly $\lambda$ is unique and $V_{\lambda}$ equals an $\ell$-weight space $V_{\Be}$ for a unique $\Be \in \CL$. We refer to $\lambda, V_{\lambda}, \Be$ and $V_{\Be}$ as the {\it top weight}, {\it top weight space}, {\it top $\ell$-weight} and {\it top $\ell$-weight space}.

Let $\Be \in \CL$ be of coweight $\mu$. 
The {\it Verma module} $M(\Be)$ is the $Y_{\mu}(\Glie)$-module defined by parabolic induction \cite[\S 3.3]{KTWWY0}
$$ M(\Be) := Y_{\mu}(\Glie) \otimes_{Y_{\mu}^+(\Glie)} \BC. $$
Here $\BC = \BC 1$ is viewed as a $Y_{\mu}^+(\Glie)$ by setting $x_i^+(u)1 = 0$ and $\xi_i(u)1 = \Be_i(u)1$.
The vector $\omega_{\Be} := 1 \otimes 1 \in M(\Be)$ is of weight $\varpi(\Be)$ and $\ell$-weight $\Be$. From the triangular decomposition of Theorem \ref{thm: triangular decomposition} and the weight grading on $Y_{\mu}^<(\Glie)$ we obtain that $M(\Be)$ is top graded with $\Be$ being the top $\ell$-weight. Moreover, the linear map $Y_{\mu}^<(\Glie) \longrightarrow M(\Be)$ sending $x \in Y_{\mu}^<(\Glie)$ to $x \omega_{\Be}$ is bijective.

By standard argument, the Verma module has a unique maximal submodule, the quotient by which is irreducible and denoted by $L(\Be)$. By abuse of language, still let $\omega_{\Be} \in L(\Be)$ denote the image of $\omega_{\Be} \in M(\Be)$ under the quotient.

Let $V$ be a $Y_{\mu}(\Glie)$-module and let $v$ be a nonzero vector of $V$. Call $v$ a {\it vector of highest $\ell$-weight $\Be$} if there exists a module morphism $M(\Be) \longrightarrow V$ sending $\omega_{\Be}$ to $v$. Namely, $\xi_i(u) v = \Be_i(u) v$ and $x_i^+(u) v = 0$ for $i \in I$.

\begin{defi}
Call $V$ a {\it module of highest $\ell$-weight $\Be$} if there exists a nonzero surjective module morphism $M(\Be) \longrightarrow V$.\end{defi}

Equivalently, $V$ is generated by a vector $v$ of highest $\ell$-weight $\Be$. It follows that $V$ is top graded with $\Be$ being the top $\ell$-weight.
In particular, $v$ is unique up to homothety, and there is a unique surjective module morphism $V \longrightarrow L(\Be)$ sending $v$ to $\omega_{\Be}$.

Recall the coproduct for $\mu, \nu$ coweights
$$\Delta_{\mu,\nu}: Y_{\mu+\nu}(\Glie) \longrightarrow Y_{\mu}(\Glie) \otimes Y_{\nu}(\Glie)$$ 
from Theorem \ref{thm: coproduct}. If $W$ and $V$ are modules over $Y_{\mu}(\Glie)$ and $Y_{\nu}(\Glie)$ respectively, then their tensor product $W \otimes V$ is naturally a module over $Y_{\mu+\nu}(\Glie)$. Since the coproduct respects the weight grading, we have 
$$W_{\alpha} \otimes V_{\beta} \subset (W \otimes V)_{\alpha+\beta}\quad\text{ for }\alpha, \beta \in \Hlie^*.$$ 
So, a tensor product of weight modules is still weight graded.

\begin{example}  \label{example: tensor product highest weight}
Let $\Be, \Bf \in \CL$. Consider the tensor product module $M(\Be) \otimes M(\Bf)$. From the coproduct of $\xi_i(u)$ and $x_i^+(u)$ in Lemma \ref{lem: coproduct estimation} we see that $\omega_{\Be} \otimes \omega_{\Bf}$ is of highest $\ell$-weight $\Be\Bf$. This implies that $L(\Be\Bf)$ is a subquotient of $L(\Be) \otimes L(\Bf)$.
\end{example}

Lowest $\ell$-weight vectors/modules can be defined by replacing $x_i^+(u)$ with $x_i^-(u)$.

\begin{example}  \label{example: V W}
Let $V$ be a $Y_{\nu}(\Glie)$-module containing a lowest $\ell$-weight vector $v_-$ and let $W$ be a $Y_{\mu}(\Glie)$-module  containing a highest $\ell$-weight vector $\omega$. Then for $j \in I,\ v \in V$ and $w \in W$, we have the following relations in the module $V\otimes W$ based on the coproduct estimation of Lemma \ref{lem: coproduct estimation}:
\begin{gather*}
 \xi_j(u) (v_- \otimes w) = \xi_j(u) v_- \otimes \xi_j(u) w, \quad x_j^-(u) (v_- \otimes w) = v_- \otimes x_j^-(u) w, \\
 x_j^-(u) (v \otimes \omega) = v \otimes x_j^-(u) \omega + \langle x_j^-(u) v \otimes \xi_j(u) \omega\rangle_+. 
\end{gather*}
In particular, if $w$ is an $\ell$-weight vector, then so is $v_- \otimes w$.
\end{example}

\subsection{Finite-dimensional irreducible modules} In this subsection we recall the classification of finite-dimensional irreducible modules over shifted Yangians from \cite[Theorem 3.5]{KTWWY0}. The result was proved in simply-laced types by reduction to $sl_2$ and applying \cite[\S 7.2]{BK}, so it works in general types. See \cite[Theorem 6.4]{H0} for a similar classification for shifted quantum affine algebras.

\begin{example}\cite[Remark 24]{Z} \label{ex: + prefund}
Let $(i,a) \in I \times \BC$. The {\it positive prefundamental module} $L_{i,a}^+$ is the one-dimensional $Y_{\varpi_i^{\vee}}(\Glie)$-module of highest $\ell$-weight (our sign convention is opposite to \cite[(10)]{Z} and agrees with \cite[Definition 3.7]{HJ})
\begin{equation}
\Psi_{i,a} := (\underbrace{1,\cdots,1}_{i-1}, u - a, \underbrace{1,\cdots,1}_{r-i}) \quad \quad \mathrm{prefundamental\ weight}.  \label{def: pre-fund}
\end{equation}
Our terminology follows \cite[Definition 3.4]{FH}. In the framework of representations of the Borel algebra \cite{HJ}, the positive prefundamental module is an infinite-dimensional irreducible module whose $\ell$-weights are rather simple, and it has important applications in quantum integrable systems (construction of Baxter's Q-operators \cite{Baxter} as transfer matrices of this modules, polynomiality of Q-operators). In another framework of representations of shifted quantum affine algebras, which is closer to our situation, the positive prefundamental modules are one-dimensional \cite[Example 4.12]{H0}.
\end{example}

For $(i, a) \in I \times \BC$ define the {\it fundamental $\ell$-weight} by
\begin{equation}
  Y_{i,a} := \frac{\Psi_{i,a-\frac{1}{2}d_i}}{\Psi_{i,a+\frac{1}{2}d_i}} \in \CL.  \label{def: fund l-weights}
\end{equation}
In the notations of \cite[\S 2.13]{CP0}, $L(Y_{i,a})$ is the finite-dimensional irreducible module over $Y(\Glie)$ with Drinfeld polynomials $P_i^+(u) = u - a - \frac{1}{2}d_i$ and $P_j^+(u) = 1$ for $j \neq i$. This justifies its name {\it fundamental module}.

\begin{theorem}\cite{BK,KTWWY0} \label{thm: finite-dim}
For  $\Be \in \CL$, the irreducible module $L(\Be)$ is finite-dimensional if and only if $\Be$ is a monomial of the $\Psi_{i,a}$ and $Y_{i,a}$ for $i \in I$ and $a \in \BC$. Furthermore, all finite-dimensional irreducible modules over shifted Yangians arise in this way.
\end{theorem}

\begin{example}  \label{ex: two-dim}
Fix $(i, a) \in I \times \BC$. Let $N_{i,a}$ be the irreducible module of highest $\ell$-weight 
$$Y_{i,a-\frac{1}{2}d_i} \prod_{j: c_{ij}<0} \Psi_{j,a-d_{ij}}.$$
It is realized on the vector space $\BC^2$ with basis $(e_1, e_2)$ such that the only nonzero actions of the generating series on the basis are
\begin{gather*}
\xi_j(u) e_1 = e_1 \begin{cases}
\frac{u-a+d_{ij}}{u-a} & \mathrm{if}\ c_{ij}\geq 0,\\
u-a+d_{ij} & \mathrm{if}\ c_{ij} < 0,
\end{cases} \quad  \xi_j(u) e_2 = e_2\begin{cases}
\frac{u-a-d_{ij}}{u-a} & \mathrm{if}\ c_{ij} \geq 0,\\
u-a-d_{ij} & \mathrm{if}\ c_{ij} < 0, 
\end{cases} \\
x_i^+(u) e_2 = \frac{1}{u-a} e_1, \quad   x_i^-(u) e_1 = \frac{d_i}{u-a} e_2.
\end{gather*}
Over the Borel algebra there is an infinite-dimensional irreducible module of similar highest $\ell$-weight \cite[\S 6.1.3]{HL}, denoted by $N_{i,a}^+$ in \cite[(6.2)]{Jimbo}, which gives rise to cluster mutations \cite{HL0,HL} and three-term Baxter's TQ relations for transfer matrices \cite[Prop.6.8]{Jimbo}. Over shifted quantum affine algebras, the irreducible module is two-dimensional \cite[Example 6.6]{H0}. 
\end{example}
The ratio of the $\ell$-weights of $N_{i,a}$ is a {\it generalized simple root}: 
\begin{equation}  
      A_{i,a} := \prod_{j \in I} \frac{\Psi_{j, a - d_{ij}}}{\Psi_{j, a+d_{ij}} } \in \CL. \label{def: simple root} 
\end{equation}
Notice that the $A_{i,a}$ for $(i,a) \in I \times \BC$ generate a free abelian subgroup of $\CL$.
Originally generalized simple roots were defined in \cite[(3.11),(4.8)]{FR} as certain evaluations of the universal R-matrix of $U_q(\hat{\Glie})$, and they were linked to $\ell$-weights therein. Similar formulas hold \cite[\S 5.5, Theorem 6.1]{H0} for shifted quantum affine algebras.  

A finite-dimensional irreducible $Y(\Glie)$-module is necessarily weight graded, as an integrable $\Glie$-module, and it is both of highest $\ell$-weight and of lowest $\ell$-weight. 

\begin{theorem}\cite{TG,Tan,Guay}  \label{thm: normalized R fin dim}
Let $U$ and $V$ be finite-dimensional irreducible $Y(\Glie)$-modules generated by highest $\ell$-weight vectors $\omega_1$ and $\omega_2$ respectively. Let $a, b \in \BC$.
\begin{itemize}
\item[(i)] There exist a tensor product of fundamental modules $T$ and an injective morphism from
$V$ to $T$ whose image contains a tensor product of highest 
$\ell$-weight vectors as well as a tensor product of lowest 
$\ell$-weight vectors.
\item[(ii)] There exists a finite subset $X$ of $\BC$ such that the module $U(a) \otimes V(b)$ is irreducible if $a-b \notin X$.
\item[(iii)] The assignment $\omega_1 \otimes \omega_2 \mapsto \omega_2 \otimes \omega_1$ extends uniquely to a linear map
$$ \check{R}_{U,V}(u): U \otimes V \longrightarrow V \otimes U \otimes \BC(u) $$
such that the evaluation at $u = a-b$ of the vector-valued rational function is a module morphism from $U(a) \otimes V(b)$ to $V(b) \otimes U(a)$, if $a-b$ is not a pole.
 \end{itemize}
\end{theorem}
We refer to \cite[Theorem 3.10]{Guay} for a proof of the theorem and for a discussion of relevant results for the quantum affine algebra.
Part (i) is a weaker version of the main results of \cite{TG,Tan}: such a tensor product can be chosen to have a unique irreducible submodule (of co-highest $\ell$-weight in the sense of Definition \ref{def: cocyclicity}). The vector-valued rational function $\check{R}_{U,V}(u)$ in Part (iii) is called {\it normalized R-matrix}. It is rarely polynomial, contrary to our R-matrices constructed later in Section \ref{sec: one-dim R-matrix}. 

As in \cite[\S 2.13]{GW}, set $\kappa := \frac{1}{2} \max(d_i: i\in I) h^{\vee}$ where $h^{\vee}$ is the dual Coxeter number of $\Glie$. One has the involution $i \mapsto \overline{i}$ of the set $I$ of Dynkin nodes of $\Glie$ induced by $w_0(\alpha_i) = -\alpha_{\overline{i}}$ where $w_0$ is the longest element of the Weyl group of $\Glie$. Define 
\begin{equation} \label{def: fund module}
V_i := L(Y_{\overline{i}, \frac{1}{2} d_i - \kappa}) \quad \mathrm{for}\ i \in I.
\end{equation}

\begin{lem} \cite[Prop.3.2]{CP0} \label{lem: lowest weight fund}
For $i \in I$, the lowest $\ell$-weight of $V_i$ is $Y_{i, \frac{1}{2}d_i}^{-1}$.
\end{lem}
The ordinary Yangian $Y(\Glie)$ is a Hopf algebra with antipode $S$. For $V$ be a $Y(\Glie)$-module, its {\it Hopf dual} is the $Y(\Glie)$-module structure on the linear dual $V^*$ defined by 
$$ (a f) (v) = f(S(a) v) \quad \mathrm{for}\ a \in Y(\Glie),\ f \in V^* \quad \mathrm{and}\ v \in V. $$
By the coproduct estimation of Lemma \ref{lem: coproduct estimation}, the dual $L(Y_{i,a})^*$ of a fundamental module $L(Y_{i,a})$ is of lowest $\ell$-weight $Y_{i,a}^{-1}$. The above lemma implies that $L(Y_{i,a})^*$ is the fundamental module $L(Y_{\overline{i}, a-\kappa})$; see \cite[Corollary 6.10]{FM} for similar arguments.

\subsection{Category $\BGG^{sh}$ and rationality} 
In this subsection we study of a category of representations of shifted Yangians, which appeared in \cite[\S 5]{KTWWY} in simply-laced types. 

For $\mu$ a coweight, define $\BGG_{\mu}$ to be the full subcategory of the category of $Y_{\mu}(\Glie)$-modules. An object of $\BGG_{\mu}$ is a $Y_{\mu}(\Glie)$-module $V$ such that: 
\begin{itemize}
\item[(O1)] it is a direct sum of finite-dimensional weight spaces;
\item[(O2)] there exist $\lambda_1,\lambda_2,\cdots,\lambda_n \in \Hlie^*$ such that 
$$\wt(V) \subseteq \bigcup_{j=1}^n (\lambda_j + \BQ_-).$$
\end{itemize}

\begin{rem}  \label{rem: derived algebra}
Assume $\mu$ dominant. The quotient of $Y_{\mu}(\Glie)$ by the ideal generated by the $[x,y]$ for $x, y \in Y_{\mu}(\Glie)$ is isomorphic to the polynomial algebra in finitely many variables $\xi_{i,p}$ for $i\in I$ and $-\langle \mu, \alpha_i\rangle \leq p<0$. A finite-dimensional module over the quotient algebra is in category $\BGG_{\mu}$ if and only if it is semi-simple. If $\mu \neq 0$, then the quotient algebra, as a polynomial algebra in at least one variable, admits finite-dimensional modules which are non semi-simple and therefore do not belong to category $\BGG_{\mu}$; see \cite[\S 5.1]{BK} for similar arguments. If $\mu = 0$, then a finite-dimensional $Y(\Glie)$-module is necessarily in category $\BGG_0$ viewed as an integrable $\Glie$-module.
\end{rem}

Category $\BGG_{\mu}$ is abelian. Let us describe its irreducible objects. The following rationality is well-known for quantum affine algebras \cite{H3} and Yangians \cite{GTL}.

\begin{lem} \label{lem: rationality Drinfeld}
Let $V$ be a $Y_{\mu}(\Glie)$-module which is a direct sum of finite-dimensional weight spaces. The generating series $x_i^{\pm}(u)$ and $\xi_i(u)$ restricted to each weight space of $V$ are rational in the sense that they are expansions at $\infty$ of rational functions of $u$ with values in finite-dimensional vector spaces.
\end{lem}
\begin{proof}
The rationality of the Laurent series $\xi_i(u)$ and $x_i^{\pm}(u)$ is proved in the same way as \cite[Prop.3.8]{H3}, \cite[Prop.3.6(i)]{GTL}: first one shows explicitly the rationality of the $x_i^{\pm}(u)$, which implies that of $\langle \xi_i(u)\rangle_+$; then $\xi_i(u)$ is $\langle \xi_i(u)\rangle_+$ plus a polynomial of $u$.
\end{proof}

Define $\CR$ to be the subgroup of $\CL$ generated by the $\Psi_{i,a}$. An element $\Be \in \CL$ belongs to $\CR$ if and only if all the components $\Be_i(u)$ are ratios of monic polynomials of $u$. Let $\CR_{\mu}$ be the set of $\Be \in \CR$ of coweight $\mu$.
\begin{theorem}   \label{thm: classification}
For $\mu$ a coweight, the $L(\Be)$ for $\Be \in \CR_{\mu}$ form the set of mutually non-isomorphic irreducible modules in category $\BGG_{\mu}$.  
\end{theorem}
\begin{proof}
Standard arguments based on the triangular decomposition and rationality of Lemma \ref{lem: rationality Drinfeld} show that any irreducible module in category $\BGG_{\mu}$ is of the form $L(\Be)$ for $\Be \in \CR_{\mu}$. It suffices to prove that $L(\Be)$ is in category $\BGG_{\mu}$ for $\Be \in \CR_{\mu}$.
Using repeatedly Eq.\eqref{rel: Drinfeld} as in \cite[\S 5, PROOF OF (b)]{CP}, we are reduced to show that for fixed $i \in I$ the vectors $x_{i,n}^- \omega_{\Be}$ with $n \in \BN$ span a finite-dimensional subspace of $L(\Be)$. Write $$\Be_i(u) = \frac{P(u)}{Q(u)}\text{ with $P(u)$ and $Q(u)$ monic polynomials}.$$ 
It suffices to prove the recurrence relation $\langle Q(u) x_i^-(u)\omega_{\Be} \rangle_+ = 0$. Indeed,
\begin{align*}
&x_{j,m}^+x_i^-(u) = x_i^-(u) x_{j,m}^+ + \delta_{ij} \sum_{n\geq 0} \xi_{i,m+n} u^{-n-1} = x_i^-(u) x_{j,m}^+ + \delta_{ij} \langle u^m \xi_i(u)\rangle_+, \\
&x_{j,m}^+\langle Q(u) x_i^-(u) \omega_{\Be} \rangle_+   =  \langle Q(u) x_{j,m}^+ x_i^-(u) \omega_{\Be} \rangle_+= \delta_{ij} \langle Q(u) \langle u^m \xi_i(u) \omega_{\Be} \rangle_+ \rangle_+ \\
& \quad = \delta_{ij} \langle u^m Q(u) \xi_i(u) \omega_{\Be} \rangle_+ = \delta_{ij} \left\langle u^m Q(u) \frac{P(u)}{Q(u)} \omega_{\Be} \right\rangle_+ = 0.
\end{align*} 
The power series $\langle Q(u) x_i^-(u) \omega_{\Be} \rangle_+$ is annihilated by all the $x_{j,m}^+$. If it is nonzero, then by applying the triangular decomposition to its coefficients we obtain a nonzero submodule of $L(\Be)$ weight graded by $\varpi(\Be)-\alpha_i+\BQ_-$, contradicting the irreducibility of $L(\Be)$.
\end{proof}

We define the {\it completed Grothendieck group} $K_0(\BGG_{\mu})$ as in \cite[\S 3.2]{HL}: its elements are formal sums $\sum_{\Be \in \CR_{\mu}} n_{\Be} [L(\Be)]$ of the symbols $[L(\Be)]$, for $\Be \in \CR_{\mu}$ and $n_{\Be} \in \BZ$ such that the direct sum of $Y_{\mu}(\Glie)$-modules $\oplus_{\Be \in \CR_{\mu}} L(\Be)^{\oplus |n_{\Be}|}$ is in category $\BGG_{\mu}$; addition is the usual one of formal sums. 
Let $V$ be in category $\BGG_{\mu}$. As in the case of Kac--Moody algebras \cite[\S 9.3]{Kac}, for $\Be \in \CR_{\mu}$ the multiplicity $m_{L(\Be),V} \in \BN$ of the irreducible module $L(\Be)$ in  $V$ makes sense, and we get a well-defined isomorphism class of $V$,
 $$[V] := \sum_{\Be \in \CR_{\mu}} m_{L(\Be),V} [L(\Be)] \in K_0(\BGG_{\mu}).$$ 

The coproduct $\Delta_{\mu,\nu}$ of Theorem \ref{thm: coproduct} induces a functor 
$$ \BGG_{\mu} \times \BGG_{\nu} \longrightarrow \BGG_{\mu+\nu},\quad (W, V) \mapsto W \otimes V. $$
Define the direct sum of abelian categories and its Grothendieck group
$$ \BGG^{sh} := \bigoplus_{\mu \in \BP^{\vee}} \BGG_{\mu},\quad K_0(\BGG^{sh}) := \bigoplus_{\mu \in \BP^{\vee}} K_0(\BGG_{\mu}). $$
Then the above functor extends to a tensor product functor
$$\otimes: \BGG^{sh} \times \BGG^{sh} \longrightarrow \BGG^{sh}.$$ 
The exactness of tensor product induces a group homomorphism
$$K_0(\BGG^{sh}) \times K_0(\BGG^{sh}) \longrightarrow K_0(\BGG^{sh}),\quad ([W],[V]) \mapsto [W \otimes V].$$ 
\begin{rem} \label{rem: monoidal}
Let $\BGG_-^{sh}$ denote the direct sum of the categories $\BGG_{\mu}$ for $\mu$ antidominant. Then the commutative diagram \eqref{rel: coproduct coass} implies that $(\BGG^{sh}_-, \otimes)$ is a monoidal category with trivial associators. It is unclear to us whether category $(\BGG^{sh}, \otimes)$ is monoidal because the coproducts fail to be co-associative for general coweights \cite[Remark 4.15]{coproduct}. 
\end{rem}

If $V$ is in category $\BGG_{\mu}$, then each weight space $V_{\beta}$ is a direct sum of $\ell$-weight spaces and each $\ell$-weight belongs to $\CR_{\mu}$ by Lemma \ref{lem: rationality Drinfeld}. 
Following Knight \cite{Kn}, we define the {\it q-character} of $V$ to be (we adopt the terminology of \cite{FR})
$$ \qc(V) := \sum_{\Bf \in \lwt(V)}  \dim(V_{\Bf}) \Bf \in \CEl. $$
 The target $\CEl$ is the set of formal sums $\sum_{\Bf \in \CR} n_{\Bf} \Bf$ of $\Bf \in \CR$ with integer coefficients $n_{\Bf}$ subject to the following conditions \cite[\S 3.4]{HJ}
\begin{itemize}
\item[(E1)] for each $\beta \in \Hlie^*$ the set $\{\Bf \in \CR\ |\ n_{\Bf} \neq 0,\ \varpi(\Bf) = \beta \}$ is finite;
\item[(E2)] there exist $\lambda_1,\lambda_2,\cdots,\lambda_m \in \Hlie^*$ such that 
$$\varpi(\Bf) \in \bigcup_{j=1}^m (\lambda_j + \BQ_-)\text{ if }n_{\Bf} \neq 0.$$
\end{itemize}
It is a ring: addition is the usual one of formal sums; multiplication is induced by that of $\CR$. One views $\CEl$ as a completion of the group ring $\BZ[\CR]$. 

 Since $\qc$ respects exact sequences, the assignment $[V] \mapsto \qc(V)$ extends uniquely to a group homomorphism  
 $$\qc: K_0(\BGG^{sh}) \longrightarrow \CEl$$ called the q-character map. The next result is proved in the same way as \cite[Theorem 2]{Kn} and \cite[Proposition 1]{FR}, based on the coproduct estimation of Lemma \ref{lem: coproduct estimation}.

\begin{theorem}\cite{Kn,FR}  \label{thm: q-char ring morphism}
The q-character map is an injective group homomorphism. Furthermore, for $W$ and $V$ in category $\BGG^{sh}$, we have 
$$\qc(W \otimes V) = \qc(W) \qc(V).$$
\end{theorem}

As an important consequence, the Grothendieck group $K_0(\BGG^{sh})$ endowed with the multiplication is a commutative ring: the associativity follows from that of the target ring $\CEl$, so does the commutativity as in \cite[Remark 3.13]{HJ}. In the case of shifted quantum affine algebras, the ring structure of the Grothendieck group is given by the fusion product of highest $\ell$-weight modules \cite[Theorem 5.4]{H0}.

For $V$ a top graded module in category $\BGG^{sh}$, we define its {\it normalized q-character} by
$$ \nqc(V) := \qc(V) \times \Be^{-1} \in \CEl$$
where $\Be$ is the top $\ell$-weight of $V$.  
In Example \ref{ex: two-dim} we have $\nqc(N_{i,a}) = 1 + A_{i,a}^{-1}$. A tensor product of top graded modules is still top graded, and the normalized q-characters are multiplicative with respect to tensor product as in Theorem \ref{thm: q-char ring morphism}.

We shall also need the notion of character which is defined in a standard way. Let $\CE$ denote the set of formal sums $\sum_{\lambda \in \Hlie^*} n_{\lambda}e^{\lambda}$ of the symbols $e^{\lambda}$ with integer coefficients $n_{\lambda}$ under the condition: there exist $\lambda_1, \lambda_2, \cdots, \lambda_m \in \Hlie^*$ such that $n_{\lambda} \neq 0$ implies $\lambda \in \cup_{j=1}^m (\lambda_j + \BQ_-)$. This is again a ring: addition is the usual one of formal sums; multiplication is induced by $e^{\lambda}e^{\mu} = e^{\lambda+\mu}$ for $\lambda, \mu \in \Hlie^*$. In particular, the weight map $\varpi: \CL \longrightarrow \Hlie^*$ induces a ring morphism 
$$ \varpi: \CEl \longrightarrow \CE,\quad \sum_{\Bf \in \CR} n_{\Bf}  \Bf \mapsto \sum_{\Bf \in \CR} n_{\Bf}  e^{\varpi(\Bf)}. $$
The character of a module $V$ in category $\BGG^{sh}$ is defined as 
$$ \chi(V) := \varpi(\qc(V)) = \sum_{\lambda \in \wt(V)} \dim (V_{\lambda} ) e^{\lambda}  \in \CE. $$
In Example \ref{ex: + prefund}, we have $\chi(L_{i,a}^+) = e^{-ad_i^{-1}\varpi_i}$ by Eq.\eqref{def: weight coweight}.

For $(i,a,k) \in I \times \BC \times \BN$, the Kirillov--Reshetikhin (KR for short) module  $W_{k,a}^{(i)}$ is the finite-dimensional irreducible $Y(\Glie)$-module of highest $\ell$-weight 
$$\frac{\Psi_{i,a-kd_i}}{\Psi_{i,a}} = Y_{i,a-\frac{1}{2}d_i} Y_{i,a-\frac{3}{2}d_i} \cdots Y_{i,a-\frac{2k-1}{2}d_i}. $$
Following \cite[Definition 3.4]{FH}, define the negative prefundamental module $L_{i,a}^-$ to be $L(\Psi_{i,a}^{-1})$ in category $\BGG_{-\varpi_i^{\vee}}$ for $(i,a) \in I \times \BC$. As in the case of the Borel algebra \cite{HJ}, it can be realized as a limit of KR modules \cite{Z}. 

\begin{prop}\cite{Z} \label{prop: existence - pref}
Fix $(i,a) \in I\times \BC$. As the integer $k \in \BN$ tends to infinity, the normalized q-character of the KR module $W_{k,a}^{(i)}$ converges to the normalized q-character of $L_{i,a}^-$ as a power series in $\BN[[A_{j,b}^{-1}]]_{j \in I, b \in \BC}$.
\end{prop}
\begin{proof}
We have constructed in \cite[Prop.23]{Z} a module $L$ in category $\BGG_{-\varpi_i^{\vee}}$ with q-character $\Psi_{i,a}^{-1} \lim_{k\rightarrow \infty} \nqc(W_{k,a}^{(i)})$. In particular $\Psi_{i,a}^{-1}$ is a highest $\ell$-weight of $L$ and $L_{i,a}^-$ is an irreducible subquotient of $L$. It suffices to show that $\Psi_{i,a} \qc(L)$ is bounded above by $\nqc(L_{i,a}^-)$, so that $L_{i,a}^- \cong L$. Since the former is the limit of $\nqc(W_{k,a}^{(i)})$, we are led to prove that $\nqc(W_{k,a}^{(i)})$ is bounded above by $\nqc(L_{i,a}^-)$ for $k \in \BN$. This follows by viewing $W_{k,a}^{(i)}$ as an irreducible subquotient of $L_{i,a-kd_i}^+ \otimes L_{i,a}^-$ and taking normalized q-characters. (Since $L_{i,a-kd_i}^+$ is one-dimensional, its normalized q-character is 1.)
\end{proof}

The character of a negative prefundamental module has a fermionic form \cite[Theorem 6.4]{HJ}. We shall need its product form, conjectured in \cite{MY} and partly proved recently in \cite{Lee}. While \cite{Lee} is about KR-modules over $U_q(\hat{\Glie})$, its main result holds true in the Yangian case by the functor of \cite{GTL} relating finite-dimensional modules over $U_q(\hat{\Glie})$ and $Y(\Glie)$. Recall that $R$ is the set of positive roots of $\Glie$. For $\gamma$ a positive root and for $i \in I$, by definition $\langle \varpi_i^{\vee}, \gamma\rangle$ is the coefficient of $\alpha_i$ in $\gamma$.

\begin{theorem}\cite{Lee} \label{thm: character formula}
Assume $\Glie$ is not of type $E_8$. For $(i,a) \in I \times \BC$ we have
$$ \chi(L_{i,a}^-) = e^{ad_i^{-1}\varpi_i} \prod_{\gamma \in R} \left( \frac{1}{1-e^{-\gamma}}\right)^{\langle \varpi_i^{\vee}, \gamma\rangle}. $$
\end{theorem}

\subsection{Examples in the $sl_2$-case} \label{ss: sl2}
For the simple Lie algebra $sl_2$, we omit the Dynkin node $1$ everywhere: $x_n^{\pm} = x_{1,n}^{\pm}$ and $\xi_p = \xi_{1,p}$ as generators; $N_a = N_{1,a}$ and $L_a^{\pm} = L_{1,a}^{\pm}$ as modules: $\Psi_a = u-a$  and $A_a = \frac{u-a+1}{u-a-1}$ as $\ell$-weights. We identify the coweight lattice with $\BZ$, so that $1$ is the fundamental coweight and $2$ is the simple coroot. Similarly, the set of weights is $\BC$, so that 1 is the fundamental weight and 2 the simple root. 

\begin{example}\cite[Proposition 2.6]{CP1} \label{ex: asym sl2}
Let $a, b \in \BC$. On the vector space with basis $(v_i)_{i \in \BN}$ there is a $Y(sl_2)$-module structure, denoted by $\SL_b^a$:
\begin{gather*}
x^+(u) v_i = \frac{1}{u-b+i-1}  v_{i-1}, \quad x^-(u) v_i = \frac{(b-a-i)(i+1)}{u-b+i} v_{i+1}, \\
\xi(u) v_i = \frac{(u-b-1)(u-a)}{(u-b+i-1)(u-b+i)} v_i.
\end{gather*}
Its normalized q-character is
$$ \nqc(\SL_b^a) = 1 + A_{b}^{-1} + A_{b}^{-1} A_{b-1}^{-1} +  A_{b}^{-1} A_{b-1}^{-1} A_{b-2}^{-1} +\cdots.$$
The vector $v_0$ generates an irreducible submodule, denoted by $L_b^a$, of highest $\ell$-weight $\frac{u-a}{u-b}$. We have $\SL_b^a = L_b^a$ if and only if $b-a \notin \BN$. When $b - a \in \BN$, 
$$ \nqc(L_b^a) = 1 + A_{b}^{-1} + A_{b}^{-1} A_{b-1}^{-1} +  A_{b}^{-1} A_{b-1}^{-1} A_{b-2}^{-1} +\cdots +  A_{b}^{-1} A_{b-1}^{-1} \cdots A_{a+1}^{-1}. $$
Let us define $ \Delta_b^a := \{k \in \BN\ |\ k < b - a \}$ if $b-a \in \BN$, and $\Delta_b^a := \BN$ otherwise. Then $k \in \Delta_b^a$ if and only if $A_{b-k}^{-1}$ is a factor of an $\ell$-weight in $\nqc(L_b^a)$.
\end{example}
For $a, b \in \BC$, the $Y(sl_2)$-module $\SL_b^a$ can be obtained as the pullback by an evaluation morphism $Y(sl_2) \longrightarrow U(sl_2)$ of an $sl_2$-module $M$ of co-highest weight $b - a$; more precisely, $M$ is the graded Hopf dual of the Verma module of lowest weight $a-b$.  
 
In the special case $a = b-1$, comparing with Example \ref{ex: two-dim} we get a module isomorphism $N_b \cong L_b^{b-1}$ sending $e_1$ to $v_0$ and $e_2$ to $v_1$.

We recall the following result of Tarasov \cite{Tarasov0, Tarasov} on $Y(sl_2)$-modules with detailed proof in \cite[Prop.3.6]{M}. It was stated for the larger Yangian $Y(gl_2)$ which contains  $Y(sl_2)$ as a Hopf subalgebra. Irreducible highest $\ell$-weight modules over $Y(gl_2)$ remain irreducible when restricted to $Y(sl_2)$. 

\begin{theorem}\cite{Tarasov0,Tarasov,M} \label{thm: Tarasov}
The $Y(sl_2)$-module $L_{b_1}^{a_1} \otimes L_{b_2}^{a_2} \otimes \cdots \otimes L_{b_n}^{a_n}$
is irreducible if and only if $b_i - a_j \notin \Delta_{b_i}^{a_i} \cap \Delta_{b_j}^{a_j}$ for any $1\leq i, j \leq n$.
\end{theorem}

The tensor product factorization in category $\BGG_0$ of $Y(sl_2)$-modules is not unique:
$$ L_0^9 \otimes L_2^3 \cong L\left(\frac{(u-9)(u-3)}{u(u-2)}\right) \cong L_0^3 \otimes L_2^9.  $$
Similar example appeared for $U_q(\widehat{sl}_2)$ in \cite[Remark 4.3]{MY}. The non-uniqueness issue will be resolved in the larger category $\BGG^{sh}$; see Theorem \ref{rem: uniqueness}.

\begin{example} \cite[Prop.23]{Z} \label{ex: - prefund sl2}
In Example \ref{ex: asym sl2}, fix $b$, divide the right-hand sides of $x^-(u) v_i$ and $\xi(u)v_i$ by $b-a$, and take the limit as $a$ goes to infinity. In this way, we obtain the negative prefundamental module $L_b^-$ over $Y_{-1}(sl_2)$:
\begin{gather*}
x^+(u) v_i = \frac{1}{u-b+i-1}  v_{i-1}, \quad x^-(u) v_i = \frac{i+1}{u-b+i} v_{i+1}, \\
\xi(u) v_i = \frac{u-b-1}{(u-b+i-1)(u-b+i)} v_i.
\end{gather*}
\end{example}

\section{Tensor products of prefundamental modules}\label{secquatre}
In this section we study two distinguished families of irreducible modules in category $\BGG^{sh}$, the one-dimensional positive prefundamental modules and the infinite-dimensional negative prefundamental modules. We prove cyclicity and co-cyclicity properties for tensor products of these modules 
(Theorem \ref{thm: cyclicity  cocyclicity}),  which motivate our definitions of Weyl modules and standard modules (Definition \ref{defi: Weyl standard modules}). In the end we identify these two modules when $\Glie$ is not of type $E_8$ (Theorem \ref{thm: Weyl standard}).

\subsection{One-dimensional modules}
Let $\CD$ be the submonoid of $\CR$ generated by the $\Psi_{i,a}$ for $(i,a) \in I \times \BC$. This is indeed the classifying set for one-dimensional modules in category $\BGG^{sh}$ in the following sense.
 
\begin{lem}  \label{lem: one-dim modules}
Let $\Be \in \CL$. Then $\dim L(\Be) = 1$ if and only if $\Be \in \CD$.
\end{lem}
\begin{proof}
One-dimensional $Y_{\mu}(\Glie)$-modules are necessarily irreducible in category $\BGG_{\mu}$, and they factorize through the quotient of $Y_{\mu}(\Glie)$ by the ideal generated by the $[x,y]$ for $x, y \in Y_{\mu}(\Glie)$. Since such an ideal contains $\xi_{i,n}$ for $(i,n) \in I \times \BN$, in the quotient each $\xi_i(u)$ is a monic polynomial.
\end{proof}
Category $\BGG^{sh}$ does not admit non-trivial invertible object: if $D$ and $E$ are modules such that $D \otimes E \cong L(1)$, then both $D$ and $E$ are isomorphic to $L(1)$. In the case of the Borel algebra \cite{HJ} or shifted quantum affine algebras \cite{H0}, there are infinitely-many invertible objects, the one-dimensional weight modules. 

By definition (compare with \cite[Theorem 4.1]{FH})
\begin{equation*}   
\qc(\Bs) =  \Bs \quad \mathrm{for}\ \Bs \in \CD.
\end{equation*}
The generalized Baxter's relations for representations of the Borel algebra \cite[Theorem 4.8]{FH} and its proof hold true in category $\BGG^{sh}$. Recall from Lemma \ref{lem: rationality Drinfeld} that the q-character of a finite-dimensional module in category $\BGG^{sh}$ lies in $\BZ[\CR]$, which is the ring of Laurent polynomials in the $\Psi_{i,a}$. 
\begin{cor} 
Let $V$ be a finite-dimensional module in category $\BGG^{sh}$. Replace in $\qc(V)$ each variable $\Psi_{i,a}$ by $[L_{i,a}^+]$ and $\qc(V)$ by $[V]$. Then multiplying by denominators, we get a relation in the Grothendieck ring of $\BGG^{sh}$.
\end{cor}

Let us apply the corollary to the module $N_{i,a}$ of Example \ref{ex: two-dim}:
\begin{gather*}  
\qc(N_{i,a}) = \frac{\Psi_{i,a-d_i}}{\Psi_{i,a}} \prod_{j: c_{ij} < 0} \Psi_{j,a-d_{ij}} + \frac{\Psi_{i,a+d_i}}{\Psi_{i,a}} \prod_{j: c_{ij} < 0} \Psi_{j,a+d_{ij}}, \\
[L_{i,a}^+][N_{i,a}] = \prod_{j: c_{ij} \neq 0}[L_{j,a-d_{ij}}^+] + \prod_{j: c_{ij} \neq 0}[L_{j,a+d_{ij}}^+].
\end{gather*}
 
For $\Bs \in \CD$ of coweight $\mu$, let $\rho_{\Bs}: Y_{\mu}(\Glie) \longrightarrow \BC$ denote the representation of the one-dimensional module $L(\Bs)$. Let $\nu$ be another coweight. Define the algebra morphisms $\iota_1^{\Bs}$ and $\iota_2^{\Bs}$, both from $Y_{\mu+\nu}(\Glie)$ to $Y_{\nu}(\Glie)$, as follows (we omit the dependence of these morphisms on $\nu$ which will always be clear from the context):
\begin{align*}
& \iota_1^{\Bs} := (\rho_{\Bs} \otimes 1) \Delta_{\mu,\nu},\quad \iota_2^{\Bs} := (1\otimes \rho_{\Bs}) \Delta_{\nu,\mu}.
\end{align*}
Then for $V$ a $Y_{\nu}(\Glie)$-module, the tensor product modules $L(\Bs) \otimes V$ and $V\otimes L(\Bs)$ are pullbacks of $V$ by $\iota_1^{\Bs}$ and $\iota_2^{\Bs}$ respectively. Based on Lemma \ref{lem: coproduct estimation} we have the following precise formulas for the algebra morphisms:
\begin{equation}  \label{equ: tensor product one-dim}
\begin{split}
 &\iota_1^{\Bs}:\quad x_i^+(u) \mapsto \langle \Bs_i(u) x_i^+(u) \rangle_+, \quad x_i^-(u) \mapsto x_i^-(u),\quad \xi_i(u) \mapsto \Bs_i(u) \xi_i(u), \\ 
& \iota_2^{\Bs}:\quad x_i^+(u) \mapsto x_i^+(u), \quad x_i^-(u) \mapsto \langle \Bs_i(u)x_i^-(u)\rangle_+,\quad \xi_i(u) \mapsto \Bs_i(u) \xi_i(u).
\end{split}
\end{equation}

\begin{rem}  \label{rem: truncation one-dim}
Let $(\nu, \Br) \in \BP^{\vee} \times \CL$ be truncatable and let $\Bs \in \CD$ be of coweight $\mu$. Then a $Y_{\nu}(\Glie)$-module $V$ factorizes through the truncated shifted Yangian $Y_{\nu}^{\Br}(\Glie)$ if and only if $L(\Bs) \otimes V$ factorizes through $Y_{\nu+\mu}^{\Br\Bs}(\Glie)$. Indeed, the uniqueness of factorization in Eq.\eqref{equ: truncation} shows that $\iota_1^{\Bs}(A_i(u)) = A_i(u)$ for $i \in I$, where the first GKLO series is taken in the shifted Yangian $Y_{\nu+\mu}(\Glie)$ with respect to the truncatable pair $(\nu+\mu, \Br\Bs)$ and the second in $Y_{\nu}(\Glie)$ with respect to $(\nu, \Br)$. 
 \end{rem} 

\subsection{Cyclicity and cocyclicity}
The main result of this subsection is the cyclicity and cocyclicity properties of tensor product modules. By cyclicity we mean the module is generated by a highest $\ell$-weight vector. Let us explain cocyclicity.

\begin{defi}  \label{def: cocyclicity}
Call a $Y_{\mu}(\Glie)$-module $V$ {\it of co-highest $\ell$-weight} if it is top graded and its top weight space is contained in all nonzero submodules of $V$.
\end{defi}

It follows that the submodule of $V$ generated by the top weight space is isomorphic to $L(\Be)$, where $\Be \in \CL$ is the top $\ell$-weight of $V$. We will also say that $V$ is co-generated by a vector of highest $\ell$-weight $\Be$.

\begin{rem}  \label{rem: cyclicity to cocyclicty}
Suppose that $V$ is module of highest $\ell$-weight $\Be$ and $W$ of co-highest $\ell$-weight $\Be$. Then there exists a nonzero module morphism $V \longrightarrow W$ which factorizes through $L(\Be)$. Such a map is unique up to homothety. It is surjective if and only if $W$ is irreducible, injective if and only if $V$ is irreducible.
\end{rem}

\begin{lem} \label{lem: cocyclicity}
Let $V$ be a top graded $Y_{\mu}(\Glie)$-module. Then $V$ is of co-highest $\ell$-weight if and only if its top weight space equals the subspace of vectors in $V$ annihilated by the $x_{i,n}^+$ for all $(i, n) \in I \times \mathbb{N}$.
\end{lem}
\begin{proof}
Let $\lambda \in \Hlie^*$ be the top weight of $V$. Let $S \subset V$ be the subspace of vectors annihilated by all the $x_{i,n}^+$. Then $V_{\lambda} \subset S$ because $\lambda+\alpha_i$ is not a weight of $V$ by assumption and $x_{i,n}^+ V_{\lambda} \subset V_{\lambda+\alpha_i}$. Moreover $S$ is weight graded.

Assume $V$ is of co-highest $\ell$-weight. Let $\beta \in \lambda+\BQ_-$ be any weight of $S$. The nonzero submodule $S' = Y_{\mu}(\Glie) S_{\beta}$ contains $V_{\lambda}$. In particular, $\lambda \in \wt(S')$. Applying the triangular decomposition to $S_{\beta}$ gives $\wt(S') \subset \beta + \BQ_-$ and $\lambda \in \beta + \BQ_-$. So $\beta  = \lambda$. This proves $\wt(S) = \{\lambda\}$ and $S = V_{\lambda}$.

Assume $S = V_{\lambda}$. Let $T$ be a nonzero submodule of $V$. Since $T$ is weight graded by $\lambda + \BQ_-$, there exists $\beta\in \wt(T)$, such that $\beta + \alpha_i \notin \wt(T)$ for all $i \in I$. This implies $x_{i,n}^+ T_{\beta} = \{0\}$ for all $(i,n) \in I \times \BN$ and therefore $\{0\} \neq T_{\beta}  \subset S = V_{\lambda}$. Since $V_{\lambda}$ is one-dimensional, $V_{\lambda} = T_{\beta}  \subset T$.
\end{proof}

Recall from \eqref{rel: spectral shift formal} the algebra homomorphism for $\mu$ a coweight
$$ \tau_z: Y_{\mu}(\Glie) \longrightarrow Y_{\mu}(\Glie) \otimes \BC[z]. $$
Let $V$ be a $Y_{\nu}(\Glie)$-module and $W$ be a $Y_{\mu}(\Glie)$-module. The vector space $W \otimes V \otimes \BC[z]$ is a module over the tensor product algebra $Y_{\mu+\nu}(\Glie) \otimes \BC[z]$: the tensor factor $\BC[z]$ acts by polynomial multiplication; the tensor factor $Y_{\mu+\nu}(\Glie)$ acts by 
$$(1\otimes \tau_z) \Delta_{\mu,\nu} : Y_{\mu+\nu}(\Glie) \longrightarrow Y_{\mu}(\Glie) \otimes Y_{\nu}(\Glie) \longrightarrow Y_{\mu}(\Glie) \otimes Y_{\nu}(\Glie) \otimes \BC[z]. $$  
Similarly, the $Y_{\mu+\nu}(\Glie)\otimes \BC[z]$-module $V \otimes \BC[z] \otimes W$ is defined using $(\tau_z\otimes 1) \Delta_{\nu,\mu}$.
\begin{rem} \label{rem: polynomiality Yangian}
For $x \in Y_{\mu+\nu}(\Glie)$ and $w \otimes v \in W \otimes V$, the action of $x$ on $w\otimes v$ in the $Y_{\mu+\nu}(\Glie)$-module $W\otimes V(a)$ is the evaluation at $z=a$ of the vector-valued polynomial $x(w\otimes v)$ computed in the $Y_{\mu+\nu}(\Glie) \otimes \BC[z]$-module $W \otimes V \otimes \BC[z]$. Similar statement holds for the $Y_{\mu+\nu}(\Glie)\otimes \BC[z]$-module $V \otimes \BC[z] \otimes W$.
\end{rem} 

\begin{theorem}  \label{thm: cyclicity cocyclicity}
 Let $\Bs \in \CD$ be of coweight $\mu$ and $\Be \in \CL$ be of coweight $\nu$. 
\begin{itemize}
\item[(i)] If $V$ is a $Y_{\nu}(\Glie)$-module of co-highest $\ell$-weight, then so are the $Y_{\nu+\mu}(\Glie)$-module $V \otimes L(\Bs)$ and the $Y_{\nu-\mu}(\Glie)$-module $L(\Bs^{-1}) \otimes V$.
\item[(ii)] The assignment $\omega_{\Bs} \otimes \omega_{\Be} \mapsto \omega_{\Bs\Be}$ extends uniquely to a module isomorphism 
$$L(\Bs) \otimes M(\Be) \cong M(\Bs\Be).$$ 
\item[(iii)] The $Y_{\nu-\mu}(\Glie) \otimes \BC[z]$-module $M(\Be) \otimes \BC[z] \otimes L(\Bs^{-1})$ is generated by
$\omega_{\Be} \otimes \omega_{\Bs^{-1}}$.
\end{itemize}
 Therefore, if a $Y_{\nu}(\Glie)$-module $V$ is of highest $\ell$-weight, then so are the $Y_{\nu+\mu}(\Glie)$-module $L(\Bs) \otimes V$ and the $Y_{\nu-\mu}(\Glie)$-module $V\otimes L(\Bs^{-1})$.
\end{theorem}
%Notice that one does not require $V$ to be in category $\BGG^{sh}$.
\begin{proof}
Assume (ii)--(iii). If $V$ is of highest $\ell$-weight, then $V$ is a quotient of a Verma module $M(\Be)$. So $L(\Bs) \otimes V$ is a quotient of $L(\Bs) \otimes M(\Be)$, which is another Verma module $M(\Bs \Be)$ by (ii). Therefore $L(\Bs) \otimes V$ is of highest $\ell$-weight. Similarly, by evaluating the $Y_{\nu-\mu}(\Glie) \otimes \BC[z]$-module $M(\Be) \otimes \BC[z] \otimes L(\Bs^{-1})$ at $z = 0$, which makes sense because of Remark \ref{rem: polynomiality Yangian}, we obtain from (iii) that the $Y_{\nu-\mu}(\Glie)$-module $M(\Be) \otimes L(\Bs^{-1})$ is of highest $\ell$-weight, so is its quotient $V \otimes L(\Bs^{-1})$.

We shall prove (i)--(iii) for $L(\Bs)$ and $L(\Bs^{-1})$ separately.

\medskip

\noindent {\bf First half of part (i).} Let $\lambda$ be the top weight of $V$.  The tensor product $V \otimes L(\Bs)$ is top graded with $V_{\lambda} \otimes \omega_{\Bs}$ being the top weight space. From the formula $\iota_2^{\Bs}(x_i^+(u)) = x_i^+(u)$ of Eq.\eqref{equ: tensor product one-dim} we get $x(v \otimes \omega_{\Bs}) = xv \otimes \omega_{\Bs}$ for $v \in V$ and $x \in Y^>(\Glie) \cong Y_{\nu+\mu}^>(\Glie) \cong Y_{\nu}^>(\Glie)$. By Lemma \ref{lem: cocyclicity}, the module $V$ is of co-highest $\ell$-weight if and only if the module $V \otimes L(\Bs)$ is of co-highest $\ell$-weight. 

\medskip

\noindent {\bf Part (ii).} By Example \ref{example: tensor product highest weight}, $\omega_{\Bs} \otimes \omega_{\Be}\in L(\Bs) \otimes M(\Be)$ is a vector of highest $\ell$-weight $\Bs\Be$. This induces a module morphism $F: M(\Bs\Be) \longrightarrow  L(\Bs) \otimes M(\Be)$ sending $\omega_{\Bs\Be}$ to $\omega_{\Bs} \otimes \omega_{\Be}$. From the formula $\iota_1^{\Bs}(x_i^-(u)) = x_i^-(u)$ of Eq.\eqref{equ: tensor product one-dim} we get $F(x \omega_{\Bs\Be}) = \omega_{\Bs} \otimes x \omega_{\Be}$ for $x \in Y^{<}(\Glie) \cong Y_{\nu+\mu}^<(\Glie) \cong Y_{\nu}^<(\Glie)$. Identifying the underlying space of Verma modules with  $Y^<(\Glie)$, we see that $F$ is an isomorphism.

\medskip

From now on fix $W := L(\Bs^{-1})$ and $\omega := \omega_{\Bs^{-1}}$.  For $i \in I$ setting $P(u) = 1$ and $Q(u) = \Bs_i(u) \in \BC[u]$ in the proof of Theorem \ref{thm: classification}, we get
\begin{equation} \label{equ: recurrent relation irreducible}
 \langle \Bs_i(u) x_i^-(u) \omega \rangle_+ = 0 \quad \mathrm{for}\ i \in I.
\end{equation}
Applying $x_i^+(u) x_{i,n}^- = x_{i,n}^- x_i^+(u) + \langle u^n \xi_i(u) \rangle_+$ to the highest $\ell$-weight vector $\omega$ gives
\begin{equation}  \label{equ: recurrent negative}
\langle \Bs_i(u) x_i^+(u) x_{i,n}^-\omega \rangle_+ = 0 \quad \mathrm{for}\ (i,n) \in I \times \BN.
\end{equation}
Indeed, the term $x_{i,n}^- x_i^+(u)$ annihilates $\omega$ and
$$ \langle\Bs_i(u) \langle u^n \xi_i(u) \omega \rangle_+ \rangle_+  = \langle u^n \Bs_i(u) \Bs_i(u)^{-1}  \omega \rangle_+ = \langle u^n\rangle_+ \omega = 0. $$

\medskip

\noindent {\bf Part (iii).} The Verma module $M(\Be)$ is $\BN$-graded $M(\Be) = \oplus_{n\in \BN} M(\Be)_n$ by declaring $M(\Be)_n$ to be the subspace spanned by the weight vectors $x_{i_1,m_1}^- x_{i_2,m_2}^- \cdots x_{i_n,m_n}^- \omega_{\Be}$ where $(i_k,m_k) \in I \times \BN$ for $1\leq k \leq n$. Similarly the $\BN$-grading on the highest $\ell$-weight module $W$ is defined.
Let $S$ be the $Y_{\nu-\mu}(\Glie) \otimes \BC[z]$-submodule  of $M(\Be) \otimes \BC[z] \otimes W$ generated by $\omega_{\Be} \otimes \omega$. It suffices to show that $M(\Be) \otimes W \subset S$.

\medskip

\noindent {\it Step 1.} Prove $M(\Be)_n \otimes \omega \subset S$ by induction on $n \in \BN$.
 
 The initial case $n = 0$ is trivial because  $M(\Be)_0 = \BC \omega_{\Be}$ and $\omega_{\Be} \otimes \omega \in S$ by definition. Let $n > 0$ and $v \in M(\Be)_n$. By linearity one may assume $v = x_{i,m}^- v'$ for certain $v' \in M(\Be)_{n-1}$ and $(i,m) \in I \times \BN$. By the induction hypothesis we have $v' \otimes \omega \in S$.

In the module $M(\Be) \otimes \BC[z] \otimes W$ we have by Example \ref{example: V W} and Eqs.\eqref{rel: spectral shift}--\eqref{rel: spectral shift formal}:
$$ x_i^-(u)(v' \otimes \omega) = x_i^-(u-z) v' \otimes \Bs_i(u)^{-1} \omega + v' \otimes x_i^-(u) \omega.  $$
Here one views $x_i^-(u-z)$ as the Laurent series $\sum_{k\geq 0} \tau_z(x_{i,k}^-) u^{-k-1}$ whose coefficients belong to $Y_{\nu}(\Glie) \otimes \BC[z]$ and act as linear maps $M(\Be) \longrightarrow M(\Be) \otimes \BC[z]$. The principal part at the right-hand side is unnecessary because $\Bs_i(u)^{-1} x_i^-(u-z)$ is a power series of $u^{-1}$. Multiply the above equation by the polynomial $\Bs_i(u)$ and then take the principal part. We obtain from Eq.\eqref{equ: recurrent relation irreducible} that
\begin{equation}  \label{rel: key}
\langle\Bs_i(u) x_i^-(u)\rangle_+ (v' \otimes \omega) = x_i^-(u-z) v' \otimes \omega \in S[[u^{-1}]].
\end{equation}
For $p \in \BN$, let $C_p$ be the coefficient of $u^{-p-1}$ at the right-hand side. We have 
$$ C_p = \sum_{k=0}^p \binom{p}{k} x_{i,p-k}^- v' \otimes z^k \otimes w \in S. $$
It follows from the Newton formula $\tau_{-a}\tau_a = \mathrm{Id}$ that
$$v \otimes w = x_{i,m}^- v' \otimes w = \sum_{k=0}^m \binom{m}{k} (-z)^{m-k} C_k \in S.   $$
Therefore $M(\Be)_n \otimes \omega \subset S$ for all $n \in \BN$ and $M(\Be) \otimes \omega \subset S$.

\medskip
 
\noindent {\it Step 2.} Prove $M(\Be) \otimes W_n \subset S$ by induction on $n \in \BN$.

The initial case $n = 0$ follows from Step 1 because $W_0 = \BC \omega$. Let $n > 0$ and $w \in W_n$. By linearity assume $w = x_{i,m}^- w'$ for certain weight vector $w' \in W_{n-1}$ and $(i,m) \in I \times \BN$. By the induction hypothesis we have $v \otimes w' \in S$. In the module $M(\Be)\otimes \BC[z] \otimes W$, take an arbitrary vector $v \in M(\Be)$ and apply $x_{i,m}^- \in Y_{\nu-\mu}^<(\Glie)$ to $v \otimes w'$. From the coproduct formula $\Delta_{\nu,-\mu}(x_{i,m}^-)$ of Lemma \ref{lem: coproduct estimation} we get
 \begin{align*}
 S &\ni x_{i,m}^- (v \otimes w') = (\tau_z \otimes 1) \Delta_{\nu,-\mu}(x_{i,m}^-) (v \otimes \omega) \\
& \equiv v \otimes x_{i,m}^- w'  \ \mathrm{mod}. \sum_{w'' \in W: \wt(w'') - \wt(w') \in \BQ_+}  M(\Be) \otimes \BC[z] \otimes w''.  
 \end{align*}
Since $\wt(w) = \wt(w') - \alpha_i$, any $w''$ in the summation belongs to $W_{n'}$ for certain $0 \leq n' < n$. The induction hypothesis applied to $w''$ together with the fact that $S$ is stable by $\BC[z]$ gives $M(\Be) \otimes \BC[z] \otimes w'' \subset S$. So $v \otimes w \in S$ for all $v \in M(\Be)$ and $w \in W_n$. Therefore, $M(\Be) \otimes W_n \subset S$ for all $n \in \BN$ and $M(\Be) \otimes W \subset S$.

\medskip

\noindent {\bf Second half of part (i).}  As in the first half of part (i), it suffices to show that if $g \in W \otimes V$ is annihilated by the $x_i^+(u)$, then $g \in \omega \otimes V_{\lambda}$. Assume $0 \neq g$ is a weight vector of weight $\beta$. Choose a weight basis $\mathcal{B}_V$ of $V$ and write $g = \sum_{v\in \mathcal{B}_V} g_v \otimes v$. Then each $g_v \in W$ is a weight vector and $g_v \neq 0$ only if $\wt(g_v) + \wt(v) = \beta$. Moreover, $g_v = 0$ for all but finitely many $v$. Choose $\gamma \in \wt(V)$ such that: 
\begin{itemize}
\item[(A)] there exists $v_1 \in \mathcal{B}_V$ of weight $\gamma$ such that $g_{v_1} \neq 0$;
\item[(B)] if $g_v \neq 0$ and $\wt(v) \neq \gamma$ then $\wt(v) \notin \gamma + \BQ_-$.
\end{itemize}
Let $i \in I$. By Lemma \ref{lem: coproduct estimation}, $x_i^+(u) (g_v \otimes v)$ is $x_i^+(u) g_v \otimes v$ plus a linear combination of vectors in $W \otimes v'$ where $v' \in \mathcal{B}_V$ satisfies $\wt(v') \in \wt(v) + \alpha_i + \BQ_+$. We get from assumption (B) that the component of $W_{\beta-\gamma+\alpha_i} \otimes V_{\gamma}$ in $x_i^+(u) g = 0$ is 
$$ \sum_{v\in \mathcal{B}_V: \wt(v) = \gamma} x_i^+(u) g_v \otimes v = 0. $$
Since the second tensor factors are linearly independent, for each $v$ in the summation, we have $x_i^+(u) g_v = 0$ for $i \in I$. Since $W$ is co-generated by the highest $\ell$-weight vector $\omega$, there exists $c_v \in \BC$ with $g_v = c_v \omega$. By assumption (A), $c_{v_1} \neq 0$ and so $\beta-\gamma = \wt(g_{v_1}) = \wt(\omega)$. 

Next we consider the component $W_{\wt(\omega)} \otimes V_{\gamma+\alpha_i}$ of in $x_i^+(u) g$. This comes from two parts by the coproduct estimation of Lemma \ref{lem: coproduct estimation} and assumption (B):
$$ 0 = \sum_{v\in \mathcal{B}_V: \wt(v) = \gamma} \xi_i(u) c_v \omega \otimes x_i^+(u) v  + \sum_{v'\in \mathcal{B}_V: \wt(v') = \gamma+\alpha_i} x_i^+(u) g_{v'} \otimes v'. $$
In the first summation, $\xi_i(u) \omega = \Bs_i(u)^{-1} \omega$. Multiply the above equality by $\Bs_i(u)$ and take the principal part. We obtain 
$$ 0 = \omega \otimes  x_i^+(u) \sum_{v\in \mathcal{B}_V: \wt(v) = \gamma}  c_v v  + \sum_{v'\in \mathcal{B}_V: \wt(v') = \gamma+\alpha_i} \langle \Bs_i(u) x_i^+(u) g_{v'} \rangle_+ \otimes v'. $$
For each $v'$ in the second summation, $\wt(g_{v'}) = \beta-\gamma-\alpha_i = \wt(\omega)-\alpha_i$. So $g_{v'}$ is a linear combination of the $x_{i,n}^- \omega$ for $n \in \BN$ and $\langle \Bs_i(u) x_i^+(u) g_{v'} \rangle_+ = 0$ by Eq.\eqref{equ: recurrent negative}. The vector $g':= \sum_{v \in \mathcal{B}_V: \wt(v) = \gamma} c_v v \in V_{\gamma}$ is annihilated by all the $x_i^+(u)$. Since $V$ is of co-highest $\ell$-weight and since $c_{v_1} \neq 0$, we obtain $0 \neq g' \in V_{\lambda}$ and so $\gamma = \lambda$. In particular, $V$ is weight graded by $\gamma + \BQ_-$. Assumption (B) forces $g_v = 0$ if $\wt(v) \neq \gamma$. So $g =  \omega \otimes g' \in \omega \otimes V_{\lambda}$. 
\end{proof}

Part (i) of Theorem \ref{thm: cyclicity  cocyclicity} was known for $V \otimes L$ where $V$ is an irreducible $U_q(\hat{\Glie})$-module in category $\BGG$ of \cite[\S 4.3]{H4}  and $L$ is a tensor product of positive prefundamental modules over the Borel algebra \cite[Lemma 5.7]{Jimbo}. Part (iii) can be seen as an integral version of cyclicity results in the fusion constructions, over the field $\BC(z)$, of representations of current algebras \cite[Prop.1.1]{FL} and quantum affinizations \cite[Theorem 6.2]{H3}. In the non-shifted case the field $\BC(z)$ is necessary because cyclicity holds true only for generic spectral parameters; see \cite{AK,FM,VV,Kashiwara,C1} for $U_q(\hat{\Glie})$ and \cite{TG,Tan,GW} for $Y(\Glie)$.

\begin{cor} \label{cor: tensor product - pre}
Let $\Br, \Bs \in \CD$. The assignment $\omega_{\Br^{-1}} \otimes \omega_{\Bs^{-1}} \mapsto \omega_{\Br^{-1}\Bs^{-1}}$ extends uniquely to a module isomorphism 
$$L(\Br^{-1}) \otimes L(\Bs^{-1}) \cong L(\Br^{-1}\Bs^{-1}).$$
\end{cor}
\begin{proof}
It suffices to prove the irreducibility of the tensor product $ L(\Br^{-1}) \otimes L(\Bs^{-1})$. By Theorem \ref{thm: cyclicity  cocyclicity}, the tensor product is at the same time of highest $\ell$-weight and of co-highest $\ell$-weight. So it must be irreducible. 
\end{proof}

For the Borel algebra and shifted quantum affine algebras, a tensor product of negative prefundamental modules is shown to be irreducible by realizing it as a limit of an inductive system of irreducible tensor products of KR modules over $U_q(\hat{\Glie})$; see \cite[Theorem 4.11]{FH}, \cite[Theorem 7.6]{HL} and \cite[Theorem 5.5]{H0}.
A similar limit procedure was carried out in \cite{Bittmann} with KR-modules replaced by finite-dimensional standard modules \cite[\S 13.2]{Nak}, resulting in modules outside of the category $\BGG$ for the Borel algebra \cite{HJ}.

\subsection{Weyl modules and standard modules} We introduce two families of highest $\ell$-weight modules in category $\BGG^{sh}$ based on the properties of tensor product modules in the previous subsection. Their definitions resemble those in the category of finite-dimensional modules over the quantum affine algebra $U_q(\hat{\Glie})$ \cite{CP3, Nak, VV}.

\begin{defi} \label{defi: Weyl standard modules}
For $\Br, \Bs \in \CD$, define the {\it standard module} $\CW(\Br, \Bs)$ to be the tensor product of irreducible modules $L(\Br) \otimes L(\Bs^{-1})$. Define the {\it Weyl module} $W(\Br,\Bs)$ to be the quotient of the Verma module $M( \Bs^{-1} \Br)$ by the relations 
$$ \langle \Bs_i(u) x_i^-(u) \rangle_+ \omega_{\Bs^{-1}\Br} = 0\quad \mathrm{for}\ i \in I. $$
\end{defi}
For $\Br = \Psi_{i_1,a_1} \Psi_{i_2,a_2} \cdots \Psi_{i_M,a_M}$ and $\Bs= \Psi_{j_1,b_1} \Psi_{j_2,b_2} \cdots \Psi_{j_N,b_N}$ we have the following factorization of the standard module (one need not care about non-associativity because each tensor product in the parentheses is irreducible)
$$ \CW(\Br,\Bs) \cong (L_{i_1,a_1}^+ \otimes L_{i_2,a_2}^+ \otimes \cdots \otimes L_{i_M,a_M}^+) \otimes (L_{j_1,b_1}^-\otimes L_{j_2,b_2}^- \otimes \cdots \otimes L_{j_N,b_N}^-). $$
This resembles the case of finite-dimensional standard modules over $U_q(\hat{\Glie})$ in \cite[Corollary 7.17]{VV} and \cite[Corollary 6.13]{Nak4}. It implies the following equation in $K_0(\BGG^{sh})$:
\begin{equation}  \label{rel: standard multiplicative}
 [\CW(\Br,\Bs) \otimes \CW(\Bm,\Bn)]  = [\CW(\Br\Bm, \Bs\Bn)]\quad \mathrm{for}\ \Br, \Bs, \Bm, \Bn \in \CD.
\end{equation}

Our definition of Weyl module is similar to the one  in the categories of finite-dimensional modules over $U_q(\hat{\Glie})$ \cite[\S 4]{CP3} and over the quantum affine superalgebra $U_q(\widehat{sl}(m,n))$  \cite[\S 4.1]{Z3}. The difference from \cite{CP3} is that apart from the highest $\ell$-weight $\Bs^{-1}\Br$ we have to introduce a new parameter $\Bs$. This is because in category $\BGG^{sh}$ and related category for the quantum affine superalgebra among the highest $\ell$-weight modules of a fixed highest $\ell$-weight there is no universal one.

\begin{rem}
For dominantly shifted Yangians of type A, there is another definition of a standard module in \cite[(7.1)]{BK} as an ordered tensor product of irreducible modules.
These modules are parametrized by their highest $\ell$-weight. We do not know whether they are particular cases of our standard modules.
\end{rem}

\begin{prop} \label{prop: Weyl vs standard}
Let $\Bm, \Br, \Bs \in \CD$.
\begin{itemize}
\item[(i)] There exists a unique surjective module morphism $W(\Br,\Bs) \longrightarrow \CW(\Br,\Bs)$ sending $\omega_{\Bs^{-1}\Br}$ to $\omega_{\Br} \otimes \omega_{\Bs^{-1}}$.
\item[(ii)] The map $\omega_{\Bm} \otimes \omega_{\Bs^{-1}\Br} \mapsto \omega_{\Bs^{-1}\Bm\Br}$ extends uniquely to a module isomorphism $$L(\Bm) \otimes W(\Br, \Bs) \longrightarrow W(\Bm\Br, \Bs).$$
\item[(iii)] Weyl modules are in category $\BGG^{sh}$. A highest $\ell$-weight module in category $\BGG^{sh}$ is necessarily a quotient of a Weyl module.
\end{itemize}
\end{prop}
\begin{proof}
{\bf Part (i).} By Theorem \ref{thm: cyclicity  cocyclicity} the module $\CW(\Br,\Bs)$ is of highest $\ell$-weight $\Bs^{-1}\Br$. Combining Eq.\eqref{equ: recurrent relation irreducible} in $L(\Bs^{-1})$ with Eq.\eqref{equ: tensor product one-dim}, we get for $j \in I$:
 $$\langle \Bs_j(u) x_j^-(u)\rangle_+ (\omega_{\Br} \otimes \omega_{\Bs^{-1}}) = \omega_{\Br} \otimes \langle \Bs_j(u) x_j^-(u) \rangle_+ \omega_{\Bs^{-1}} = 0.$$
All the defining relations of the Weyl module $W(\Br,\Bs)$ are realized in $\CW(\Br,\Bs)$.

\medskip 
 
\noindent {\bf Part (ii).} Let $\mu = \sum_{j\in I} k_j \varpi_j^{\vee}$ be the coweight of $\Bs^{-1}\Br$. Set $\tilde{x}_i^-(u) := \langle \Bs_i(u) x_i^-(u)\rangle_+$. 
 Let $V(\Br,\Bs)$ be the subspace of the Verma module $M(\Bs^{-1}\Br)$ spanned by the coefficients of $\tilde{x}_i^-(u) \omega_{\Bs^{-1}\Br}$ for $i \in I$, and let $K(\Br,\Bs)$ be the submodule generated by this subspace. Then $W(\Br,\Bs) = M(\Bs^{-1}\Br)/K(\Br,\Bs)$. For $(j,m) \in I \times \BN$, as in the proof of Theorem \ref{thm: classification} we have $x_{j,m}^+ V(\Br,\Bs) = \{0\}$. Next we prove by induction on $p \geq -k_j -1$ that $\xi_{j,p} V(\Br,\Bs) \subset V(\Br,\Bs)$. The initial case is trivial because $\xi_{j,-k_j-1} = 1$. Applying the following relation to the highest $\ell$-weight vector $\omega_{\Bs^{-1}\Br}$, which is a common eigenvector of the $\xi_{j,q}$ for $q \in \BZ$, 
\begin{align*}
\xi_{j,p+1} \tilde{x}_i^-(u) &= \tilde{x}_i^-(u) \xi_{j,p+1} + \xi_{j,p}\langle u  \tilde{x}_i^-(u) \rangle_+ \\
&\quad - \langle u \tilde{x}_i^-(u)\rangle_+ \xi_{j,p} - d_{ij} \xi_{j,p} \tilde{x}_i^-(u) - d_{ij} \tilde{x}_i^-(u) \xi_{j,p},
\end{align*}
we derive the case of $p+1$ from the case of $p$.
From the triangular decomposition of Theorem \ref{thm: triangular decomposition} we get $K(\Br,\Bs) = Y^<(\Glie) V(\Br,\Bs)$. 

Theorem \ref{thm: cyclicity  cocyclicity} (ii) affords a module isomorphism $F: M(\Bs^{-1}\Bm\Br) \longrightarrow L(\Bm) \otimes M(\Bs^{-1}\Br)$ which sends $x \omega_{\Bs^{-1}\Bm\Br}$ to $\omega_{\Bm} \otimes x \omega_{\Bs^{-1}\Br}$ for $x \in Y^<(\Glie)$. In particular $F$ maps the subspace $V(\Bm\Br,\Bs)$ of  the Verma module $M(\Bs^{-1}\Bm\Br)$ onto $\omega_{\Bm} \otimes V(\Br,\Bs)$. Furthermore,
$$ F(K(\Bm\Br,\Bs)) = F(Y^<(\Glie) V(\Bm\Br,\Bs)) = \omega_{\Bm} \otimes Y^<(\Glie) V(\Br,\Bs) = \omega_{\Bm} \otimes K(\Br,\Bs). $$
This induces the desired module isomorphism $W(\Bm\Br,\Bs) \cong L(\Bm) \otimes W(\Br,\Bs)$.

\medskip

\noindent {\bf Part (iii).} Let $W$ be a Weyl module generated by a highest $\ell$-weight vector $\omega$, so that $\wt(W) \subset \wt(\omega) + \BQ_-$. By definition  $W_{\wt(\omega)-\alpha_i}$ is finite-dimensional for $i \in I$. One can copy \cite[\S 5, PROOF OF (b)]{CP} to show that all weight spaces of $W$ are finite-dimensional. Therefore $W$ is in category $\BGG^{sh}$.

Let $V$ be a module in category $\BGG^{sh}$ generated by a highest $\ell$-weight vector $v$ of $\ell$-weight $\Bn$. Fix $i \in I$. The $x_{i,m}^- v$ for $m \in \BN$ span a finite-dimensional weight space. We get a monic polynomial $Q_i(u) = \sum_{k=0}^n c_k u^k$ with $\sum_{k=0}^n c_k x_{i,k}^-  v = 0$. Applying $x_{i,m}^+$ for $m \in \BN$ to this equality we get that $Q_i(u) \Bn_i(u)$ is a monic polynomial of $u$. Apply repeatedly $\xi_{i,-k_i+1} - \frac{1}{2} \xi_{i,-k_i}^2$ to the equality. From
\begin{equation*}  
 [\xi_{i,-\langle \mu, \alpha_i \rangle+1} - \frac{1}{2} \xi_{i,-\langle \mu, \alpha_i \rangle}^2, x_{i,m}^-] = - 2d_i x_{i,m+1}^-,
\end{equation*}
we get $\sum_{k=0}^n c_k x_{i,k+m}^-v = 0$ for all $m \in \BN$, namely, $\langle Q_i(u) x_i^-(u) \rangle_+v = 0$. Let us set $\Bs := (Q_i(u))_{i\in I}$ and $\Br := \Bn \Bs$. Then $\Br, \Bs \in \CD$ and $V$ is a quotient of $W(\Br,\Bs)$.
\end{proof}

In the rest of this subsection we show that Weyl modules are standard modules when $\Glie$ is not of type $E_8$. Recall from Subsection \ref{ss: ordinary Yangian} the root vectors $x_{\gamma}^{\pm} \in \Glie_{\pm\gamma}$ and the PBW variables $x_{\gamma, n}^- \in Y(\Glie)^{\leq n}$  for $(\gamma,n) \in R \times \BN$. Identify the associated grading $\mathrm{gr}_{\BN} Y(\Glie)$ with $U(\Glie[t])$ via the isomorphism of \eqref{iso: filtration}. Then for $i \in I$ and $n \in \BN$,
$$ x_{i,n}^- + Y(\Glie)^{\leq n-1} = x_{\alpha_i,n}^- + Y(\Glie)^{\leq n-1} = x_{\alpha_i}^- \otimes t^n\ \text{and}\ d_i \alpha_i^{\vee} = \xi_{i,0}. $$
In general, $x_{\gamma,n}^- + Y(\Glie)^{\leq n-1}$ is proportional to $x_{\gamma}^- \otimes t^n \in \Glie[t]$ for $\gamma \in R$.

Let $\Bs \in \CD$. Consider the Weyl module $W(\Bs, \Bs)$ over $Y(\Glie)$. Its zero weight space $W(\Bs,\Bs)_0$ is spanned by a highest $\ell$-weight vector $\omega$. The $\BN$-filtration of $Y(\Glie)$ descends to the module $W(\Bs, \Bs)$ by setting $$W(\Bs, \Bs)^{\leq m} := Y(\Glie)^{\leq m} \omega \text{ for $m \in \BN$.}$$
The associated grading $\mathrm{gr}_{\BN} W(\Bs,\Bs)$ is then naturally an $\BN$-graded $U(\Glie[t])$-module.  This is referred to as the {\it classical limit} of $W(\Bs, \Bs)$, denoted by $\overline{W(\Bs,\Bs)}$. 

\begin{lem} \label{lem: classical limit}
Let $\Bs \in \CD$. As a module over $\Hlie \subset \Glie[t]$, the classical limit $\overline{W(\Bs,\Bs)}$ is semi-simple whose character is equal to $\chi(W(\Bs,\Bs))$. Moreover, the $\Glie[t]$-module $\overline{W(\Bs,\Bs)}$ is generated by $\omega$ and satisfies the relations
$$ (x_{\alpha_i}^+ \otimes t^n) \omega = 0 = (\alpha_i^{\vee} \otimes t^n) \omega = (x_{\alpha_i}^- \otimes t^{\langle \varpi^{\vee}(\Bs), \alpha_i \rangle}) \omega \quad \mathrm{for}\ (i,n) \in I \times \BN. $$
\end{lem}
\begin{proof}
The first part comes from the compatibility of $\BN$-filtration and weight grading on $Y(\Glie)$. For the second part, fix $i \in I$ and set $N := \langle \varpi^{\vee}(\Bs), \alpha_i \rangle$. By definition $\Bs_i(u) = \sum_{k=0}^N c_k u^k \in \BC[u]$ with $c_N = 1$. In the Weyl module $W(\Bs,\Bs)$ we have 
$$ x_{i,N}^- \omega = - \sum_{k=0}^{N-1} c_k x_{i,k}^- \omega \in W(\Bs,\Bs)^{\leq N-1}. $$
In the associated grading, the left-hand side becomes $(x_{\alpha_i}^- \otimes t^N) \omega$ and is of degree $N$, while the right-hand side is of degree $N-1$. So both sides vanish and $(x_{\alpha_i}^- \otimes t^N) \omega = 0$ in the classical limit. The first two relations are proved in the same way.
\end{proof}

Let $\mathfrak{a}_{\Bs}$ be the Lie subalgebra of $\Glie[t]$ generated by the elements $x_{\alpha_i}^+ \otimes t^n,\ \alpha_i^{\vee} \otimes t^n$ and $x_{\alpha_i}^- \otimes t^{\langle \varpi^{\vee}(\Bs), \alpha_i \rangle}$ for $(i, n) \in I \times \BN$. Define the $U(\Glie[t])$-module
$$ P_{\Bs} := U(\Glie[t]) \otimes_{U(\mathfrak{a}_{\Bs})} \BC $$
where $\mathfrak{a}_{\Bs}$ acts on $\BC$ as zero. Lemma \ref{lem: classical limit} shows that $\overline{W(\Bs,\Bs)}$ is a quotient of $P_{\Bs}$.
\begin{lem} \label{lem: character classical limit}
The action of $\Hlie \subset \Glie[t]$ on $P_{\Bs}$ is semi-simple with character
$$ \chi(P_{\Bs}) = \prod_{\gamma \in R} \left( \frac{1}{1-e^{-\gamma}}\right)^{\langle\varpi^{\vee}(\Bs), \gamma\rangle}. $$
\end{lem}
\begin{proof}
We claim that the subalgebra $\mathfrak{a}_{\Bs}$ of $\Glie[t]$ is spanned by 
$$\mathrm{(B)}:\quad \quad  x_{\gamma}^+ \otimes t^n,\quad \alpha_i^{\vee} \otimes t^n, \quad x_{\gamma}^- \otimes t^m $$
where $(i,\gamma, m, n) \in I \times R \times \BN^2$ and $m \geq \langle\varpi^{\vee}(\Bs), \gamma\rangle$. Assume the claim. Take $V$ to be the subspace of $\Glie[t]$ with basis $x_{\gamma}^- \otimes t^k$ where $(\gamma, k) \in R \times \BN$ and $k < \langle\varpi^{\vee}(\Bs), \gamma\rangle$. Then $\Glie[t] = \mathfrak{a}_{\Bs} \oplus V$ is a direct sum of semsimple $\Hlie$-modules. By the PBW theorem for $U(\Glie[t])$, the $\Hlie$-module $P_{\Bs}$ is isomorphic to the symmetric algebra $\mathrm{Sym}(V)$ whose $\Hlie$-module structure is induced from that of $V$. Each basis vector $x_{\gamma}^- \otimes t^k$ of $V$ is of weight $-\gamma$ and gives rise to a factor $\frac{1}{1-e^{-\gamma}}$ in the character $\chi(\mathrm{Sym} (V)) = \chi(P_{\Bs})$. Taking their products gives the desired product character formula.

First we show that each vector in (B) belongs to $\mathfrak{a}_{\Bs}$. This is clear for the first two families of vectors because $x_{\gamma}^+ \otimes t^n$ is proportional to a commutator of the $x_{\alpha_i}^+ \otimes t^k$ for $(i,k) \in I \times \BN$. For the third family, one may assume $m = \langle\varpi^{\vee}(\Bs), \gamma\rangle$, as the other cases can be deduced from adjoint actions of the $\alpha_i^{\vee} \otimes t$. Write $\gamma = \sum_{i\in I} m_i \alpha_i$. Then $x_{\gamma}^-$ is proportional to a commutator of the $x_{\alpha_i}^-$ where each $x_{\alpha_i}^-$ appears exactly $m_i$ times. Replacing in the commutator formula of $x_{\gamma}^-$ each $x_{\alpha_i}^-$ with $x_{\alpha_i}^- \otimes t^{\langle \varpi^{\vee}(\Bs), \alpha_i \rangle}$, we obtain $x_{\gamma}^- \otimes t^m$ in the Lie subalgebra $\mathfrak{a}_{\Bs}$.

It remains to show that the subspace of $\Glie[t]$ spanned by (B) is a Lie subalgebra. The only non-trivial case is to check that the following vector belongs to this subspace for $(\delta, \gamma, n, m) \in R^2 \times \BN^2$ with $m \geq \langle\varpi^{\vee}(\Bs), \gamma\rangle$:
$$[x_{\delta}^+ \otimes t^n, x_{\gamma}^- \otimes t^m] = [x_{\delta}^+, x_{\gamma}^-] \otimes t^{m+n}.$$
If $\gamma-\delta \notin R$, then $[x_{\delta}^+, x_{\gamma}^-]$ is spanned by the $x_{\beta}^+$ and $\alpha_i^{\vee}$ for $(i,\beta) \in I \times R$, so we get the first two families of vectors in (B). If $\gamma -\delta \in R$, then $[x_{\delta}^+, x_{\gamma}^-] \in \BC x_{\gamma-\delta}^-$ and
$$ m+n \geq m \geq \langle\varpi^{\vee}(\Bs), \gamma\rangle > \langle\varpi^{\vee}(\Bs), \gamma-\delta\rangle. $$
So the commutator does belong to the third family of vectors in (B). 
 \end{proof}

\begin{theorem}  \label{thm: Weyl standard}
Assume that $\Glie$ is not of type $E_8$. Then for $\Br, \Bs \in \CD$, the surjective morphism $W(\Br,\Bs) \longrightarrow \CW(\Br,\Bs)$ of Proposition \ref{prop: Weyl vs standard} (i) is an isomorphism. Furthermore, the ordered monomials in the $x_{\gamma,n}^-$ for $(\gamma,n) \in R \times \BN$ and $n < \langle \varpi^{\vee}(\Bs),\gamma\rangle$ applied to $\omega_{\Bs^{-1}\Br}$ form a basis of the Weyl module $W(\Br,\Bs)$.
\end{theorem}
\begin{proof}
By Proposition \ref{prop: Weyl vs standard} (i), $\chi(\CW(\Bs,\Bs))$ is bounded above by $\chi(W(\Bs,\Bs))$. 
Lemma \ref{lem: character classical limit} implies that $\chi(W(\Bs,\Bs)) = \chi(\overline{W(\Bs,\Bs)})$ is bounded above by $\chi(P_{\Bs})$. In view of Corollary \ref{cor: tensor product - pre} and Theorem \ref{thm: character formula}, if we write $\Bs = \Psi_{i_1,a_1} \Psi_{i_2,a_2} \cdots \Psi_{i_N,a_N}$, then 
$$ \chi(\CW(\Bs,\Bs)) = \prod_{k=1}^N \chi(\CW(\Psi_{i_k,a_k}, \Psi_{i_k,a_k})) = \prod_{k=1}^N \prod_{\gamma \in R}  \left( \frac{1}{1-e^{-\gamma}}\right)^{\langle\varpi_{i_k}^{\vee}, \gamma\rangle}  = \chi(P_{\Bs}). $$
So all the characters are equal and the quotient map $P_{\Bs} \longrightarrow \overline{W(\Bs,\Bs)}$ is bijective.  From the tensor product factorizations of Proposition \ref{prop: Weyl vs standard} (ii) we get 
$$\nqc(W(\Br,\Bs)) = \nqc(W(\Bs,\Bs)) = \nqc(\CW(\Bs,\Bs)) = \nqc(L(\Bs^{-1})) = \nqc(\CW(\Br,\Bs)). $$
This implies $W(\Br,\Bs) \cong \CW(\Br,\Bs)$.

For the second part, the case $\Br = \Bs$ follows from the PBW basis of the classical limit $\overline{W(\Bs,\Bs)} \cong P_{\Bs}$ obtained in the proof of Lemma \ref{lem: character classical limit}. The rest follows from the fact that the isomorphism $L(\Bm) \otimes W(\Bn,\Bs) \cong W(\Bm\Bn,\Bs)$ of Proposition \ref{prop: Weyl vs standard} (ii) for $\Bm, \Bn \in \CD$ identifies the actions of $Y^<(\Glie)$ on the Weyl modules $W(\Bn,\Bs)$ and $W(\Bm\Bn, \Bs)$.
\end{proof}

 \begin{rem}
 Our proof of Theorem \ref{thm: Weyl standard} in category $\BGG^{sh}$ is simpler than the finite-dimensional case \cite[Theorem 7.5]{CM}. The reason is that we take the quotient of $U(\Glie[t])$ by a left ideal generated by elements of the Lie algebra $\Glie[t]$. In the finite-dimensional case, depending on a dominant integral weight $\sum_{i\in I} k_i \varpi_i$, the left ideal of $U(\Glie[t])$ to define the classical limit is generated by \cite[\S 2]{CP3}:
 $$ x_{\alpha_i}^+ \otimes t^n,\quad \alpha_i^{\vee} \otimes t^n - \delta_{n0} k_i,\quad (x_{\alpha_i}^- \otimes 1)^{k_i + 1}  \quad \mathrm{for}\ (i,n) \in I \times \BN. $$
The third family of generators is not in $\Glie[t]$. This makes it highly nontrivial to find a basis for the quotient, even in type A \cite{CL}.  
 \end{rem}
 
\section{Properties of R-matrices}  \label{sec: one-dim R-matrix}
In this section we study R-matrices, which are module morphisms
$$V \otimes W \longrightarrow W \otimes V$$ where $V$ and $W$ are suitably chosen highest $\ell$-weight modules. 
 We establish first properties of the R-matrices: existence, uniqueness, factorization and polynomiality (Theorem \ref{thm: R-matrix}, Propositions \ref{prop: polynomiality R-} and \ref{prop: one-dim R+}). 
We also compute the eigenvalues of certain of these $R$-matrices in Proposition  \ref{prop: l-weight R-matrix+}.

\begin{defi}
A module in category $\BGG^{sh}$ is called {\it negative} if it is irreducible and its highest $\ell$-weight is of the form $\Bs^{-1}$ with $\Bs \in \CD$. 
\end{defi}

\begin{theorem} \label{thm: R-matrix}
Let $V$ be a $Y_{\nu}(\Glie)$-module generated by a highest $\ell$-weight vector $v_0$ and $W$ be a $Y_{\mu}(\Glie)$-module generated by a highest $\ell$-weight vector $\omega$. Then the assignment $v_0 \otimes \omega \mapsto \omega \otimes v_0$ extends uniquely to a $Y_{\mu+\nu}(\Glie)$-module morphism 
$$\check{R}_{V,W}: V \otimes W \longrightarrow W \otimes V$$
under one of the following conditions:
\begin{itemize}
\item[(i)] The module is irreducible, and $W$ is negative.
\item[(ii)] The module $V$ is one-dimensional, and $W$ is either a Verma module, or a highest $\ell$-weight irreducible module, or a Weyl module.
\end{itemize}
\end{theorem}
\begin{proof}
{\bf Part (i).} By Theorem \ref{thm: cyclicity  cocyclicity},  $V \otimes W$ is of highest $\ell$-weight, and $W \otimes V$ is of co-highest $\ell$-weight. Their highest $\ell$-weights coincide by Example \ref{example: tensor product highest weight}. The existence and uniqueness of $\check{R}_{V,W}$ follows from Remark \ref{rem: cyclicity to cocyclicty}.

\medskip

\noindent {\bf Part (ii).} The same arguments as above work when $\dim V = 1$ and $W$ is irreducible.

Suppose $V = L(\Bs)$ and $W = M(\Be)$ with $\Bs \in \CD$ and $\Be \in \CL$. Then $V \otimes W$ is isomorphic to the Verma module $M(\Bs\Be)$ by Theorem \ref{thm: cyclicity  cocyclicity} (ii). Since $\omega \otimes v_0 \in W \otimes V$ is of highest $\ell$-weight $\Bs\Be$, the existence and uniqueness of $\check{R}_{V,W}$ follows.

Suppose $V = L(\Bm)$ and $W = W(\Br,\Bs)$ with $\Bm, \Br, \Bs \in \CD$. Then $V \otimes W$ is isomorphic to the Weyl module $W(\Bm\Br,\Bs)$ by Proposition \ref{prop: Weyl vs standard} (ii). The uniqueness and existence of $\check{R}_{V,W}$ follows if the highest $\ell$-weight vector $\omega \otimes v_0 \in W \otimes V$ satisfies the defining relations of the Weyl module $W(\Bm\Br,\Bs)$. 
We have $\langle \Bs_i(u) x_i^-(u) \omega \rangle_+ = 0$ in the Weyl module $W(\Br,\Bs) = W$ and $x_i^-(u) (\omega \otimes v_0) = \langle \Bm_i(u) x_i^-(u) \omega \rangle_+ \otimes v_0$ in the tensor product module $W \otimes L(\Bm) = W \otimes V$ by Eq.\eqref{equ: tensor product one-dim}. So
\begin{align*}
&\langle \Bs_i(u) x_i^-(u)  (\omega \otimes v_0) \rangle_+ = \langle \Bs_i(u) \langle \Bm_i(u) x_i^-(u) \omega \rangle_+ \rangle_+ \otimes v_0  \\
&= \langle \Bm_i(u) \langle \Bs_i(u) x_i^-(u) \omega \rangle_+ \rangle_+ \otimes v_0 = 0.
\end{align*}
The second equality used twice Eq.\eqref{asso: truncation}. 
\end{proof}
$\check{R}_{V,W}$ is independent of the choice of the highest $\ell$-weight vectors $v_0$ and $\omega$ because both of them span a one-dimensional weight space. It is normalized as in Theorem \ref{thm: normalized R fin dim}.

As a first application, we consider the situation of Theorem \ref{thm: R-matrix} (i). For $a \in \BC$ there exists a unique $Y_{\mu+\nu}(\Glie)$-module morphism 
$$ \check{R}_{V,W}(a): V(a) \otimes W \longrightarrow W \otimes V(a),\quad v_0 \otimes \omega \mapsto \omega \otimes v_0. $$

\begin{prop} \label{prop: polynomiality R-}
Let $V$ be a highest $\ell$-weight irreducible $Y_{\nu}(\Glie)$-module and let $W$ be a negative module. Then there exists a unique linear map 
$$ \check{R}_{V,W}(u): V \otimes W \longrightarrow W \otimes V \otimes \BC[u] $$
whose evaluation at $u = a$ is $\check{R}_{V,W}(a)$ for any $a\in \BC$.
\begin{itemize}
\item[(i)] Assume $\nu$ is antidominant. For $\Br, \Bs \in \CD$, we have 
\begin{equation*} 
\check{R}_{V, L(\Br^{-1}\Bs^{-1})} = (1 \otimes \check{R}_{V,L(\Bs^{-1})}) (\check{R}_{V,L(\Br^{-1})}  \otimes 1)
\end{equation*}
where we have identified $L(\Br^{-1}) \otimes L(\Bs^{-1})$ with $L(\Br^{-1}\Bs^{-1})$ as in Corollary \ref{cor: tensor product - pre}.
\item[(ii)] Let $U$ and $V$ be finite-dimensional irreducible $Y(\Glie)$-modules. Then we have the following {\it quantum Yang--Baxter equation} 
\begin{equation} \label{YBE} 
 \check{R}_{U,V}^{23}(u-v) \check{R}_{U,W}^{12}(u) \check{R}_{V,W}^{23}(v) = \check{R}_{V,W}^{12}(v) \check{R}_{U,W}^{23}(u) \check{R}_{U,V}^{12}(u-v).
\end{equation}
It is an equality of linear maps from $U \otimes V \otimes W$ to $W \otimes V \otimes U \otimes \BC(u,v)$, where $R^{23}  = 1 \otimes R,\  R^{12} = R \otimes 1$, and $\check{R}_{U,V}(u-v)$ is the normalized R-matrix of Theorem \ref{thm: normalized R fin dim}.
\end{itemize}
\end{prop}
\begin{proof} 
Suppose that the negative module $W$ is defined over the shifted Yangian $Y_{\mu}(\Glie)$ and fix highest $\ell$-weight vectors $\omega$ in $W$ and $v_0$ in $V$. We need to show that for any vectors $v \in V$ and $w \in W$, the function $a \mapsto \check{R}_{V,W}(a) (v\otimes w)$ is polynomial, in the sense that it is the evaluation at $z = a$ of polynomial of $z$ taking values in $W\otimes V$. 
By Remark \ref{rem: polynomiality Yangian}, the $Y_{\mu+\nu}(\Glie)$-module $V(a) \otimes W$ is the evaluation at $z=a$ of the $Y_{\mu+\nu}(\Glie) \otimes \BC[z]$-module $V \otimes \BC[z] \otimes W$. By Theorem \ref{thm: cyclicity  cocyclicity} (iii), there exists a polynomial $\sum_{s=0}^N X_s z^s$ with coefficients $X_ s \in Y_{\mu+\nu}(\Glie)$  such that in the module $V(a) \otimes W$:
$$ v \otimes w = \sum_{s=0}^N a^s X_s (v_0 \otimes \omega) \in V(a) \otimes W. $$ 
Applying $\check{R}_{V,W}(a)$ to the relation yields
$$\check{R}_{V,W}(a) (v\otimes w) = \sum_{s=0}^N a^s X_s (\omega \otimes v_0) \in W \otimes V(a).  $$
Conclude from the polynomial actions of the $X_s$ on $W \otimes V(a)$.

\medskip

\noindent {\bf Part (i).} Both sides are module morphisms because all the tensor factors $V, L(\Br^{-1})$ and $L(\Bs^{-1})$ are modules over antidominantly shifted Yangians and belong to the monoidal category $\BGG_-^{sh}$ of Remark \ref{rem: monoidal}. Both sides send $v_0 \otimes \omega_{\Br^{-1}} \otimes \omega_{\Bs^{-1}}$ to  $\omega_{\Br^{-1}} \otimes \omega_{\Bs^{-1}} \otimes v_0$. They have to be equal by uniqueness of R-matrix in Theorem \ref{thm: R-matrix}.

\medskip

\noindent{\bf Part (ii).}  Let $\omega_1, \omega_2$ and $\omega_3$ denote highest $\ell$-weight vectors of $U, V$ and $W$. By the rationality of $\check{R}_{U,V}(u)$ and polynomiality of $\check{R}_{U,W}(u)$ and $\check{R}_{V,W}(u)$,  it suffices to prove Eq.\eqref{YBE} evaluated at $u = a$ and $v = b$ for $a, b \in \BC$ satisfying the conditions of Theorem \ref{thm: normalized R fin dim} (ii)--(iii). Both sides are module morphisms from $U(a) \otimes V(b) \otimes W$ to $W \otimes V(b) \otimes U(a)$ since the tensor factors are modules over antidominantly shifted Yangians and belong to the monoidal category $\BGG_-^{sh}$ of Remark \ref{rem: monoidal}. The source module is generated by $\omega_1 \otimes \omega_2 \otimes \omega_3$ by Theorem \ref{thm: cyclicity  cocyclicity} and the irreducibility of $U(a) \otimes V(b)$. Both sides send the generator to $\omega_3 \otimes \omega_2 \otimes \omega_1$, so they must coincide.
\end{proof}
For $v$ and $w$ weight vectors in the modules $V$ and $W$ respectively, the vector $v \otimes w$ in the module $V(a) \otimes W$ is of weight $\wt(v)+\wt(w)-a\tilde{\nu}$ by Remark \ref{rem: spectral weight}; recall that $V$ is defined over $Y_{\nu}(\Glie)$. So $\check{R}_{V,W}(a)(v\otimes w)$ belongs to the weight space of the module $W\otimes V(a)$ of the same weight, which by Remark \ref{rem: spectral weight} coincides with the finite-dimensional weight space $(W\otimes V)_{\wt(v)+\wt(w)}$ of the module $W\otimes V$. It follows that $\check{R}_{V,W}(u)$ restricts to polynomials taking values in finite-dimensional vector spaces:
$$ \check{R}_{V,W}(u) \in \mathrm{Hom}((V\otimes W)_{\gamma}, (W\otimes V)_{\gamma}) \otimes \BC[u] \quad \mathrm{for}\ \gamma \in \wt(V\otimes W). $$

\begin{defi} \label{def: diagonal entries}
In the situation of Theorem \ref{thm: R-matrix} (i), the {\it highest diagonal entry} of $\check{R}_{V,W}(u)$ is the linear operator $s_{V,W}(u) \in \mathrm{Hom}(V, V \otimes \BC[u])$ such that for $v \in V$:
\begin{equation}  \label{matrix coefficient highest}
\check{R}_{V,W}(u) (v \otimes \omega) \equiv \omega \otimes s_{V,W}(u) v \ \mathrm{mod.}\ \sum_{\wt(\omega) \neq \gamma \in \wt(W)} W_{\gamma} \otimes V \otimes \BC[u].
\end{equation}
If we assume furthermore that $V$ is finite-dimensional with $v_-$ being a lowest $\ell$-weight vector, then the {\it lowest diagonal entry} of $\check{R}_{V,W}(u)$ is defined to be the linear operator $t_{V,W}(u) \in \mathrm{Hom}(W, W \otimes \BC[u])$ such that for $w \in W$:
\begin{equation}  \label{matrix coefficient lowest}
\check{R}_{V,W}(u) (v_- \otimes w) \equiv  t_{V,W}(u) w \otimes v_- \ \mathrm{mod.}\ \sum_{\wt(v_-) \neq \gamma \in \wt(V)} W \otimes  V_{\gamma} \otimes \BC[u].
\end{equation}
Define the polynomial $\lambda_{V,W}(u) \in \BC[u]$ to be the eigenvalue of $t_{V,W}(u)$ associated to $\omega$. 
\end{defi}
The above definition makes sense because both of the vectors $\omega$ and $v_-$ span a one-dimensional weight space. 

\begin{rem}  \label{rem: factorization}
Let $W$ be a highest $\ell$-weight module over the quantum affine algebra $U_q(\hat{\Glie})$ and $V$ be a lowest $\ell$-weight module over the Borel subalgebra. We specialize the universal R-matrix $\mathcal{R}(z)$ of $U_q(\hat{\Glie})$ to obtain $\check{R}_{V,W}(z): V \otimes W \longrightarrow (W \otimes V)[[z]]$. Then we attach the highest diagonal entry $s_{V,W}(z): V \longrightarrow V[[z]]$ to a highest $\ell$-weight vector $\omega$ of $W$ and the lowest diagonal entry $t_{V,W}(z): W \longrightarrow W[[z]]$ to a lowest $\ell$-weight vector $v_-$ of $V$ as above. The proofs of \cite[Lemma 2.6]{FM} and \cite[Prop.5.5]{FH} indicate that
$$ s_{V,W}(z) = \mathcal{R}^0(z) \mathcal{R}^{\infty}|_{V \otimes \omega},\quad t_{V,W}(z) = \mathcal{R}^0(z)\mathcal{R}^{\infty}|_{v_- \otimes W}.  $$
Here the abelian part $\mathcal{R}^0(z)$ and the Cartan part $\mathcal{R}^{\infty}$ of $\mathcal{R}(z)$ are written in terms of Drinfeld--Cartan generators. Therefore both diagonal entries are determined by the action of the Drinfeld--Cartan generators. Our computation of the diagonal entries in the Yangian situation (Proposition \ref{prop: highest diagonal entry} and Theorem \ref{thm: TQ pre R}) is guided by this principle, although we do not have a universal R-matrix for shifted Yangians.
\end{rem}

\begin{example} \label{example: polynomiality sl2}
Fix $i\in I$. Let $V = N_{i,0}$ and $v_0$ be the highest $\ell$-weight vector $e_1$ in Example \ref{ex: two-dim}. Write $\Bs_i(u) = \sum_{k=0}^m c_k u^k$ with $c_m =1$. Apply $\check{R}_{V,W}(a)$ to Eq.\eqref{rel: key} with $v' = e_1$ and $z = a$ and then use Example \ref{example: V W} to compute $x_i^-(u)(\omega \otimes e_1)$. We obtain
\begin{align*}
 \frac{d_i}{u-a} \check{R}_{V,W}(a)( e_2 \otimes \omega)	&= \check{R}_{V,W}(a) (x_i^-(u) e_1 \otimes \omega) = \langle \Bs_i(u) x_i^-(u) (\omega \otimes e_1) \rangle_+ \\
&=\langle \Bs_i(u) x_i^-(u) \omega \otimes \xi_i(u) e_1 \rangle_+ + \langle \Bs_i(u) \omega \otimes x_i^-(u) e_1 \rangle_+\\
&= \left\langle \Bs_i(u) \frac{u-a+d_i}{u-a} x_i^-(u)\right\rangle_+ \omega \otimes e_1 + \left\langle \Bs_i(u) \frac{d_i}{u-a} \right\rangle_+ \omega \otimes e_2, \\
\check{R}_{V,W}(a)( e_2 \otimes \omega) &= \Bs_i(a) \omega \otimes e_2 + \sum_{n=0}^{m-1} \sum_{k=n+1}^m c_k a^{k-n-1} x_{i,n}^-\omega \otimes e_1.
\end{align*}
Since $e_2$ is a lowest $\ell$-weight vector, $\lambda_{V,W}(u) = \Bs_i(u)$.
\end{example}

Next consider Theorem \ref{thm: R-matrix} (ii). Let $\Bs \in \CD$ denote the highest $\ell$-weight of $V = \BC v_0$. We have a unique linear operator $R_{\Bs}^W$ on $W$ such that  $R_{\Bs}^W(\omega) = \omega$ and 
 $$ \check{R}_{V, W}(v_0 \otimes w) = R_{\Bs}^W(w) \otimes v_0\quad \mathrm{for}\ w \in W. $$

\begin{prop}  \label{prop: one-dim R+}
Let $W$ be either a Verma module, or its irreducible quotient, or a Weyl module, of highest $\ell$-weight $\Be$.
\begin{itemize}
\item[(i)] We have $R_{\Br}^W \circ R_{\Bs}^W = R_{\Br\Bs}^W$ for $\Br, \Bs \in \CD$.  In particular, the $R_{\Bs}^W$ for $\Bs \in \CD$ form a commuting family of linear endomorphisms on $W$. 
\item[(ii)] For $i \in I$, there exists a unique linear map 
$$ R_i^W(u): W \longrightarrow W \otimes \BC[u] $$
whose evaluation at $u = a$, for $a \in \BC$, is $R_{\Psi_{i,a}}^W$. Furthermore, for $\beta \in \varpi(\Be) + \BQ_-$, restricted to the weight space $W_{\beta}$, the operator $R_i^W(-u)$ is an $\mathrm{End}(W_{\beta})$-valued monic polynomial of degree $ \langle \varpi_i^{\vee}, \varpi(\Be)-\beta\rangle$.
\end{itemize}
\end{prop}
\begin{proof}
Suppose $W = M(\Be)$ is a Verma module and write $\omega = \omega_{\Be}$. 

\medskip

\noindent {\bf Part (i).} Comparing the actions of the shifted Yangian on the two tensor products, we see that the linear map $R_{\Bs}^{W}: W \longrightarrow W$ for $\Bs \in \CD$ is uniquely characterized by the equations $R_{\Bs}^W (\omega) = \omega$ and $R_{\Bs}^{W} \circ \iota_1^{\Bs} = \iota_2^{\Bs} \circ R_{\Bs}^{W}$, namely:
\begin{gather}
R_{\Bs}^{W} (\omega) = \omega,\quad R_{\Bs}^{W} x_i^-(u) = \langle \Bs_i(u) x_i^-(u) \rangle_+ R_{\Bs}^{W}, \label{Baxter: recursion}  \\
R_{\Bs}^{W} \xi_i(u) = \xi_i(u) R_{\Bs}^{W},\quad R_{\Bs}^{W} \langle \Bs_i(u) x_i^+(u)\rangle_+ = x_i^+(u) R_{\Bs}^{W}. \label{Baxter: affine Cartan}
\end{gather}
Eq.\eqref{Baxter: recursion} already determines $R_{\Bs}^{W}$, because $W$ is obtained from $\omega$ by applying repeatedly the $x_{i,n}^-$. We need to check Eq.\eqref{Baxter: recursion} for $R_{\Br}^{W} \circ R_{\Bs}^{W}$ with $\Bs$ replaced by $\Br\Bs$. The first half is evident. For the second half, 
\begin{align*}
& R_{\Br}^{W} \circ R_{\Bs}^{W} x_i^-(u) = R_{\Br}^{W}  \langle \Bs_i(u) x_i^-(u) \rangle_+ R_{\Bs}^{W} \\
 =& \langle \Bs_i(u) \langle \Br_i(u) x_i^-(u) \rangle_+ \rangle_+ R_{\Br}^{W} \circ R_{\Bs}^{W} = \langle \Bs_i(u)\Br_i(u) x_i^-(u) \rangle_+ R_{\Br}^{W} \circ R_{\Bs}^{W}.
\end{align*}
The second half of part (i) follows from the commutativity of $\CD$. 

\medskip

\noindent {\bf Part (ii).} For $(j,n) \in I \times \BN$ we have by Eq.\eqref{Baxter: recursion}: 
$$ R_{\Psi_{i,-a}}^{W} x_{i,n}^- = (x_{i,n+1}^- + a x_{i,n}^-) R_{\Psi_{i,-a}}^W,\quad   R_{\Psi_{i,-a}}^{W} x_{j,n}^- = x_{j,n}^- R_{\Psi_{i,a}}^{W}\quad \mathrm{if}\ j \neq i.  $$
Write $\varpi(\Be) - \beta = \sum_{j \in I} h_j \alpha_j$ so that $h_j =  \langle \varpi_j^{\vee}, \varpi(\Be)-\beta\rangle$. By the triangular decomposition $W_{\beta}$ is spanned by the vectors of the form $ x_{j_1,n_1}^- \cdots x_{j_K,n_K}^- \omega$ where each $j \in I$ appears exactly $h_j$ times. Applying $R_{\Psi_{i,-a}}^{W}$ to such a vector gives
$$ R_{\Psi_{i,-a}}^{W}( x_{j_1,n_1}^- \cdots x_{j_K,n_K}^- \omega) =  \prod_{s=1}^K (\delta_{j_si} x_{i,n_s+1}^- + a^{\delta_{j_si}} x_{j_s,n_s}^-) \times \omega. $$
The right-hand side is the evaluation at $u=a$ of an $W_{\beta}$-valued polynomial whose dominant term is $u^{h_i} x_{j_1,n_1}^- \cdots x_{j_K,n_K}^- \omega$. We have therefore proved for each $v \in W_{\beta}$ the existence of $h_i$ vectors $v_0, v_1, \cdots, v_{h_i-1} \in W_{\beta}$ such that 
$$ R_{\Psi_{i,-a}}^W (v) = a^{h_i} v + a^{h_i-1} v_{h_i-1} + \cdots + a v_1 + v_0 \quad \mathrm{for}\ a \in \BC.  $$
By standard argument of Vandermonde determinant,  each $v \mapsto v_s$ defines a linear operator $Q_s$ on $W_{\beta}$ for $0\leq s < h_i$. The $\mathrm{End}(W_{\beta})$-valued monic polynomial 
$$u^{h_i} \mathrm{Id} + u^{h_i-1} Q_{h_i-1} + \cdots + u Q_1 + Q_0$$
defines the restriction of $R_i^W(-u)$ to $W_{\beta}$.

\medskip

Suppose $W$ is either the irreducible quotient of $M(\Be)$, or a Weyl module. Let $\pi: M(\Be) \longrightarrow W$ denote the quotient map. The following diagram
\begin{gather*} 
  \xymatrixcolsep{6pc} \xymatrix{
L(\Bs) \otimes M(\Be) \ar[d]^{1\otimes \pi} \ar[r]^{\check{R}_{L(\Bs), M(\Be)}} & M(\Be) \otimes L(\Bs) \ar[d]^{\pi \otimes 1}  \\
L(\Bs) \otimes W \ar[r]^{\check{R}_{L(\Bs), W}} & W\otimes L(\Bs)
 } 
\end{gather*}
is commutative because the top-left module is generated by $\omega_{\Bs} \otimes \omega_{\Be}$ which is sent to $\omega_{\Be} \otimes \omega_{\Bs}$ by both paths. So (i)--(ii) for the Verma module descend to $W$.
\end{proof}

In the rest of this section, we compute the eigenvalues of $R_i^W(u)$ for $W$ a Weyl module or its irreducible quotient. From Eq.\eqref{Baxter: affine Cartan} we get $R_i^W(u) \xi_{j,p} = \xi_{j,p} R_i^W(u)$. So $R_i^W(u)$ restricts to an $\mathrm{End} (W_{\Bf})$-valued polynomial for each $\ell$-weight $\Bf$ of $W$. 

\begin{prop}  \label{prop: l-weight R-matrix+}
Let  $(\mu, \Br) \in \BP^{\vee} \times \CL$ be a truncatable pair and $A_i(u)$ for $i \in I$ be the GKLO series in $Y_{\mu}(\Glie)$.  Let
$W$ be either a Weyl module or an irreducible module, generated by a vector $\omega$ of highest $\ell$-weight $\Be \in \CR_{\mu}$. Let $g_i(u) \in \BC((u^{-1}))^{\times}$ be the eigenvalue of $A_i(u)$ associated to $\omega$ and normalize $\overline{A}_i(u) := g_i(u)^{-1} A_i(u)$.
\begin{itemize}
\item[(i)] We have an additive difference equation
\begin{equation} \label{equ: difference}
R_i^W(u+d_i) = R_i^W(u) \overline{A}_i(u) \in \mathrm{Hom}(W, W \otimes \BC[u]). 
\end{equation}
\item[(ii)] Each $\ell$-weight $\Bf$ of $W$ has a unique decomposition 
$$\Bf = \Be \prod_{j\in I} \prod_{s=1}^{h_j} A_{j,a_{j,s}}^{-1}$$
where $h_j = \langle \varpi_j^{\vee}, \varpi(\Bf^{-1}\Be) \rangle$ and $a_{j,s} \in \BC$ for $1\leq s \leq h_j$. Furthermore, both of the operators $R_i^W(u)$ and $A_i(u)$ acting on $W_{\Bf}$ have a unique eigenvalue, respectively 
$$\prod_{s=1}^{h_i} (a_{i,s}-u)\quad \text{ and }\quad g_i(u) \prod_{s=1}^{h_i} \frac{u-a_{i,s}+d_i}{u-a_{i,s}}.$$
\end{itemize}
As a consequence, the normalized q-character of an arbitrary highest $\ell$-weight module in category $\BGG^{sh}$ is a power series in $\BN[[A_{i,a}^{-1}]]_{i\in I, a \in \BC}$ with leading term $1$.
\end{prop}
\begin{proof} {\bf Part (i).} The left-hand side of Eq.\eqref{equ: difference} sends $W$ to $W \otimes \BC[u]$ by Proposition \ref{prop: one-dim R+}, and the right-hand side sends each finite-dimensional weight space $W_{\beta}$ of $W$ to $W_{\beta} \otimes \BC((u^{-1}))$. Since $W$ is obtained from $\omega$ by repeatedly applying the $x_{j,n}^-$ for $j \in I$ and $n \in \BN$, and since both sides send $\omega$ to $\omega$, it suffices to show that both sides have the same commutation relations with the $x_{j,n}^-$. 

By Lemma \ref{lem: affine coweight} and Eq.\eqref{Baxter: recursion}, if $j \neq i$, then $R_i^W(u), A_i(u)$ and both sides of Eq.\eqref{equ: difference} commute with $x_{j,n}^-$. If $j = i$, then 
\begin{align*}
& R_i^W(u+d_i) x_{i,n}^- = (x_{i,n+1}^- - u x_{i,n}^- - d_i x_{i,n}^-) R_i^W(u+d_i),  \\
& R_i^W(u) \overline{A}_i(u) x_{i,n}^- = R_i^W(u) (x_{i,n}^- + d_i \sum_{k\geq 0} x_{i,n+k}^- u^{-k-1})\overline{A}_i(u)  \\
& \qquad =  (x_{i,n+1}^- - u x_{i,n}^- + d_i \sum_{k\geq 0} (x_{i,n+k+1}^- - u x_{i,n+k}^-) u^{-k-1}) R_i^W(u) \overline{A}_i(u)  \\
& \qquad = (x_{i,n+1}^- - u x_{i,n}^- - d_i x_{i,n}^-)R_i^W(u) \overline{A}_i(u).
\end{align*}
This proves Eq.\eqref{equ: difference}.

\medskip

\noindent {\bf Part (ii).} The finite-dimensional $\ell$-weight space $W_{\Bf}$ admits mutually commuting actions of the series $R_j^W(u), \xi_j(u),\overline{A}_i(u)$ for $j \in I$. One can choose a basis $B$ of $W_{\Bf}$ with respect to which the matrices of these series are upper triangular. Fix a basis vector $b \in B$. For $X(u)$ any of these series, let $[X(u)]_b$ denote the $b$-th diagonal of the matrix of $X(u)$. Then $[\xi_j(u)]_b = \Bf_j(u)$ by definition of the $\ell$-weight space. 

By Proposition \ref{prop: one-dim R+}, the linear map $R_j^W(-u): W_{\varpi(\Bf)} \longrightarrow W_{\varpi(\Bf)} \otimes \BC[u]$, viewed as an $\mathrm{End}(W_{\varpi(\Bf)})$-valued polynomial in $u$ is monic of degree $h_j$. Its restriction to the $\ell$-weight space $W_{\Bf}$ as an $\mathrm{End}(W_{\Bf})$-valued polynomial is also monic of degree $h_j$, so is its diagonal entry $[R_j^W(-u)]_b \in \BC[u]$ associated to the basis vector $b$. Let $-a_{j,s}$ for $1\leq s \leq h_j$ denote the roots of the eigenvalue, which may depend on $b$. Then $[\overline{A}_j(u)]_b$ can be computed from the difference equation \eqref{equ: difference}:
$$ [R_j^W(u)]_b = \prod_{s=1}^{h_j} (a_{j,s}-u),\quad  [\overline{A}_j(u)]_b = \frac{[R_j^W(u+d_j)]_b}{[R_j^W(u)]_b} = \prod_{s=1}^{h_j} \frac{u-a_{j,s}+d_j}{u-a_{j,s}}. $$
By definition $\overline{A}_i(u)$ is the normalization of $A_i(u)$ by its eigenvalue associated to $\omega$. Applying \eqref{equ: truncation} to $\omega$ and then to $b \in B$ we get
\begin{align*}
\frac{\Bf_i(u)}{\Be_i(u)}  &= \frac{1}{[\overline{A}_i(u)]_b[\overline{A}_i(u-d_i)]_b} \prod_{j: c_{ji} < 0} \prod_{t=1}^{-c_{ji}} [\overline{A}_j(u-d_{ij}-t d_j)]_b \\
&= \prod_{s=1}^{h_i} \frac{(u-a_{i,s})(u-a_{i,s}-d_i)}{(u-a_{i,s}+d_i)(u-a_{i,s})} \times  \prod_{j: c_{ji} < 0}  \prod_{s'=1}^{h_j} \prod_{t=1}^{-c_{ji}}  \frac{u-d_{ij}-td_j-a_{j,s'}+d_j}{u-d_{ij}-td_j - a_{j,s'}}  \\
&= \prod_{s=1}^{h_i} \frac{u-a_{i,s}-d_i}{u-a_{i,s}+d_i} \times \prod_{j: c_{ji} < 0} \prod_{s'=1}^{h_j}  \frac{u-d_{ij}-a_{j,s'}}{u-d_{ij}+c_{ji}d_j - a_{j,s'}} \\
&= \prod_{j\in I} \prod_{s=1}^{h_j} \frac{u-a_{j,s}-d_{ij}}{u-a_{j,s}+d_{ij}} \quad (\mathrm{because}\ c_{ji}d_j = 2 d_{ij}). 
\end{align*}
From Eq.\eqref{def: simple root} we get $\Bf = \Be \prod_{j\in I}  \prod_{s=1}^{h_j} A_{j,a_{j,s}}^{-1}$.
Since the generalized simple roots generate a free abelian subgroup of $\CR$, the $a_{i,s}$ for $i \in I$ and $1\leq s \leq h_i$ are uniquely determined by $\Bf^{-1}\Be$ and they are independent of the basis vector $b \in B$. This shows that both $R_i^W(u)$ and $A_i(u)$ have a single eigenvalue of the desired form.
\end{proof}
\begin{rem}  \label{rem: transfer matrix}
 (i) Recall from \eqref{equ: isomorphism} that $Y^-_{\mu}(\Glie) \cong Y_0^-(\Glie)$. As in \cite[\S 2.6]{GTL0}, there are {\it shift operators} $\sigma_i$ for $i \in I$, which are algebra endomorphisms defined by
$$\sigma_i: Y_{\mu}^-(\Glie) \longrightarrow Y_{\mu}^-(\Glie),\quad \xi_{j,p} \mapsto \xi_{j,p},\quad x_{j,n}^- \mapsto x_{j,n+\delta_{ij}}^-. $$
It follows from Lemma \ref{lem: affine coweight} and Eq.\eqref{Baxter: recursion} that:
\begin{gather}  
A_i(u) x_{j,n}^-  =  \frac{u-\sigma_i + d_i \delta_{ij}}{u-\sigma_i}  x_{j,n}^- A_i(u),     \label{rel: A x-}     \\
R_i^W(u) x_{j,n}^- = (-u+\sigma_i)^{\delta_{ij}} x_{j,n}^- R_i^W(u). \label{rel: T x-}
\end{gather}
This also illustrates the additive difference equation \eqref{equ: difference}.

(ii)  In the Borel situation, there is a universal solution to Eq.\eqref{rel: T x-}, denoted by $T_i(z)$ in \cite[Prop.5.5]{FH}, whose action on a $U_q(\hat{\Glie})$-module $W$ is the lowest diagonal entry of the specialization $\mathcal{R}(z)|_{L \otimes W}$, where $L$ is a lowest $\ell$-weight module over the Borel subalgebra obtained as the graded Hopf dual of a negative prefundamental module; see \cite[\S 3.4, \S 7.2]{FH} and Remark \ref{rem: factorization}. The polynomiality of $T_i(z)$ follows from the stronger one for the transfer matrix; see \cite[Theorems 5.9, 5.17]{FH}. Our difference equation \eqref{equ: difference} corresponds to \cite[(9.18), (10.20)]{H0} for shifted quantum affine algebras.

(iii) The ordinary Yangian $Y(\Glie)$ processes the abelian part $\mathcal{R}^0(u)$ of the universal R-matrix. While it is divergent as a formal infinite product, its specialization on a tensor product of finite-dimensional $Y(\Glie)$-modules makes sense by viewing it as a solution to a difference equation \cite[\S 5.8]{GTL1}. Proposition \ref{prop: l-weight R-matrix+} is close to this approach. We expect that a suitably shifted version of $\mathcal{R}^0(u)$ can be specialized to a tensor product of modules in category $\BGG^{sh}$ and that it recovers the highest/lowest diagonal entries as in the case of quantum affine algebras in Remark \ref{rem: factorization}.
\end{rem}
\begin{cor}  \label{cor: irre one }
Let $$\Bs = \Psi_{i_1,a_1} \Psi_{i_2,a_2} \cdots \Psi_{i_N,a_N} \in \CD$$ and $W$ be an irreducible module in category $\BGG^{sh}$. The module $L(\Bs) \otimes W$ is irreducible if and only if: for all $1\leq s \leq N$ and $\Bf \in \lwt(W)$ we must have $A_{i_s,a_s}^{-1} \Bf \notin \lwt(W)$.
\end{cor} 
\begin{proof}
From the proof of Theorem \ref{thm: R-matrix} we get a module morphism $\check{R}_{L(\Bs),W}$ from the highest $\ell$-weight module $L(\Bs) \otimes W$ to the co-highest $\ell$-weight module $W \otimes L(\Bs)$. It is injective if and only if $L(\Bs) \otimes W$ is irreducible. We have $$R_{\Bs}^{W} = R_{i_1}^W(a_1) R_{i_2}^W(a_2) \cdots R_{i_N}^W(a_N),$$ which is a product of mutually commuting operators on $W$. $R_{\Bs}^W$ is injective if and only if $0$ is not an eigenvalue of any of the operators $R_{i_s}^W(a_s)$. The rest follows from Proposition \ref{prop: l-weight R-matrix+} (ii). 
\end{proof}
The \lq\lq if\rq\rq\ part of the corollary was known \cite[Lemma 5.9]{Jimbo} for $L \otimes V$ where $L$ is a tensor product of positive prefundamental modules over the Borel algebra and $V$ is an irreducible $U_q(\hat{\Glie})$-module in category $\BGG$ of \cite[\S 4.3]{H4}.

\begin{example}  \label{example: short exact sequence}
Let $W = N_{i,a}$ as in Example \ref{ex: two-dim} and write $L_{i,a}^+ =  \BC \mathbf{1}$. Then $e_1$ and $e_2$ are eigenvectors of $R_i^W(u)$ of eigenvalues $1$ and $a-u$ respectively. Consider the module morphism $\check{R}_{L_{i,a}^+,W}$ from $L_{i,a}^+ \otimes W$ to $W \otimes L_{i,a}^+$: its image is spanned by the vector $e_1 \otimes \mathbf{1}$ of $\ell$-weight $\prod_{j: c_{ij} \neq 0} \Psi_{j,a-d_{ij}}$; its kernel is spanned by $\mathbf{1} \otimes e_2$ of $\ell$-weight $\prod_{j: c_{ij} \neq 0} \Psi_{j,a+d_{ij}}$. We obtain a short exact sequence of modules in category $\BGG^{sh}$:
\begin{equation*} 
0 \longrightarrow \bigotimes_{j: c_{ij} \neq 0} L_{i,a+d_{ij}}^+ \longrightarrow L_{i,a}^+ \otimes N_{i,a} \longrightarrow \bigotimes_{j: c_{ij} \neq 0} L_{j,a-d_{ij}}^+ \longrightarrow 0.
\end{equation*}
Similar short exact sequence appeared in the category $\BGG$ of the Borel algebra \cite[Theorem 5.16]{H1}, whose proof also made use of R-matrices.
\end{example}

\section{Tensor product factorization in the $sl_2$-case}   \label{sec: sl2}
In this section $\Glie$ is fixed to be $sl_2$. We prove existence and uniqueness of factorization for {\it all} irreducible modules in category $\BGG^{sh}$ into tensor products of prefundamental modules and KR modules (Theorem \ref{rem: uniqueness}). This result will be crucial in the proof of Jordan--H\"older property in Section \ref{sec: truncation}.

In our situation, $\CR$ is the subgroup of the multiplicative group of the field $\BC(u)$ generated by the $u-a$ for $a \in \BC$. Recall from Subsection \ref{ss: sl2} the subset $\Delta_b^a \subset \BN$ and the irreducible modules $L_b^a$ and $L_a^{\pm}$ for $a, b \in \BC$.  

\begin{defi} \label{def: rational factorisation}
Let $\Be \in \CR$. A {\it standard factorization} of $\Be$ is 
$$ \Be = \prod_{r=1}^m (u-x_r) \times \prod_{s=1}^n \frac{u-y_s}{u-z_s} \times \prod_{t=1}^k \frac{1}{u-w_t}  $$
where for $1\leq r \leq m,\ 1\leq s, l \leq n$ and $1\leq t \leq k$:
\begin{gather*}
0 \neq z_s - y_s \in \BN,\quad z_s - y_l \notin \Delta_{z_s}^{y_s} \cap \Delta_{z_l}^{y_l}, \\
z_s - x_r  \notin \Delta_{z_s}^{y_s}, \quad w_t - x_r \notin \BN, \quad w_t - y_s \notin \Delta_{z_s}^{y_s}.
\end{gather*}
\end{defi}
For example, $(u-3) (u-9) \times \frac{u-5}{u-6} \times \frac{1}{u}  \frac{1}{u-2}$ is a standard factorization, while this is false for $(u-5) (u-9) \times \frac{u-3}{u-6} \times \frac{1}{u}  \frac{1}{u-2}$. 
\begin{prop} \label{prop: factorization sl2}
Consider the following factorization and tensor product:
\begin{align*}
&(F): \qquad \Be =  \prod_{r=1}^m (u-x_r) \times \prod_{s=1}^n \frac{u-y_s}{u-z_s} \times \prod_{t=1}^k \frac{1}{u-w_t}, \\
&T= (L_{x_1}^+\otimes \cdots \otimes L_{x_m}^+) \otimes (L_{z_1}^{y_1} \otimes \cdots \otimes L_{z_n}^{y_n}) \otimes (L_{w_1}^- \otimes \cdots \otimes L_{w_k}^-).
\end{align*}
Suppose $0 \neq z_s - y_s \in \BN$ for all $1\leq s \leq n$. Then $(F)$ is a standard factorization of $\Be$ if and only if $T$ is an irreducible module isomorphic to $L(\Be)$.
\end{prop}
\begin{proof}
Note that the irreducibility of $T$ would force it to be isomorphic to $L(\Be)$, by comparing highest $\ell$-weights. Write $T = T_1 \otimes T_2$ where  $T_1$ denotes the tensor product of the first $m$ factors, and $T_2$ the remaining part. First, notice that $T_1$, being one-dimensional, is isomorphic to $L((u-x_1)\cdots (u-x_m))$. 

 For $1\leq t \leq k$ let us choose $w_t' \in \BC$ in such a way that $w_t' - w_l \notin \BZ$ and $w_t' - z_s \notin \BZ$ for all $1\leq l \leq k$ and $1\leq s \leq n$. Set $\Bs := (u-w_1') \cdots (u-w_k')$. 
 
 \medskip

\noindent {\it Claim 1.} The module $T_2$ is irreducible if and only if $T_2 \otimes L(\Bs)$ is irreducible. 
 
 The $\Leftarrow$ part is trivial since tensor product is exact. For the $\Rightarrow$ part, assume the irreducibility of $T_2$. By Corollary \ref{cor: irre one }, it suffices to prove that none of the $A_{w_t'}^{-1}$ for $1\leq t \leq k$ appears as a factor of any $\ell$-weight in the normalized q-character of $T_2$: 
 \begin{align*}
 \nqc(T_2) &= \prod_{s=1}^n \nqc(L_{z_s}^{y_s}) \prod_{l=1}^k\nqc(L_{w_l}^-).
 \end{align*}
 $\nqc(L_{z_s}^{y_s})$ only admits the $A_{z_s-c}^{-1}$ for $c \in \Delta_{z_s}^{y_s}$ as factors. $\nqc(L_{w_l}^-)$ only admits the $A_{w_l-c}^{-1}$ for $c \in \BN$ as factors. The assumption $w_l - w_t'  \notin \BZ$ and $z_s - w_t' \notin \BZ$ guarantees the condition of Corollary \ref{cor: irre one }, hence the irreducibility of $T_2 \otimes L(\Bs)$.
 
 \medskip
 
\noindent {\it Claim 2.} The module $T_2 \otimes L(\Bs)$ is irreducible if and only if the tensor product 
 $$ T_2' := (L_{z_1}^{y_1} \otimes \cdots \otimes L_{z_n}^{y_n}) \otimes (L_{w_1}^{w_1'} \otimes \cdots \otimes L_{w_k}^{w_k'})  $$
 is irreducible. This is because the two modules have the same q-character.
 
 Applying Theorem \ref{thm: Tarasov} to $T_2'$, in view of our choice of the additional parameters $w_t'$, we get that $T_2$ is irreducible if and only if
 $$ z_l - y_s  \notin \Delta_{z_s}^{y_s} \cap \Delta_{z_l}^{y_l}, \quad w_t - y_s \notin \Delta_{z_s}^{y_s}\quad \mathrm{for}\ 1\leq s, l \leq n,\ 1\leq t \leq k.  $$
Note that the irreducibility of $T$ is equivalent to the irreducibility of the module $T_2$ and the tensor product $L((u-x_1)\cdots (u-x_m)) \otimes T_2$.  The latter is again in the situation of Corollary \ref{cor: irre one }. It is irreducible if and only if for $1\leq r \leq m, 1\leq s \leq n$ and $1\leq t \leq k$:
\begin{itemize}
\item $A_{x_r}^{-1}$ does not appear in $\nqc(L_{z_s}^{y_s})$, which means $z_s - x_r  \notin \Delta_{z_s}^{y_s}$;
\item $A_{x_r}^{-1}$ does not appear in $\nqc(L_{w_t}^-)$, which means $w_t - x_r  \notin \BN$.
\end{itemize}
Therefore, $T$ is irreducible if and only if all the conditions from Definition \ref{def: rational factorisation} on the $x_r, y_s, z_s, w_t$ are satisfied, meaning that $(F)$ is a standard factorization.
\end{proof}

If $k \leq m$ in the standard factorization then the irreducibility of the tensor product follows from \cite[Theorem 7.7]{BK}, by first identifying $L_{x_t}^+ \otimes L_{w_t}^-$ with $L_{w_t}^{x_t}$ for $1\leq t \leq k$ (so that there is no negative prefundamental module), and then replacing $L_b^a$ and $L_b^+$ with the modules $L(_b^a)$ and $L(_b)$ respectively in {\it loc.cit}.

If $\Be$ is a product of the $\frac{u-a+1}{u-a}$ for $a \in \BC$, then a standard factorization is equivalent to writing a finite set of complex numbers with multiplicities as a union of pairwise non-interacting strings \cite[Prop.3.5]{CP1}. 

\begin{lem} \label{lem: uniqueness factorization}
Let $\Be \in \CR$. Standard factorizations of $\Be$ as in Definition \ref{def: rational factorisation} exist, and they are unique in the sense that the two polynomials $(u-x_1)\cdots (u-x_m)$ and $(u-w_1) \cdots (u-w_k)$, and the pairs $(y_s,z_s)$ for $1\leq s \leq n$ up to $\mathfrak{S}_n$-permutations are completely determined by $\Be$.
\end{lem}
\begin{proof}
We begin with some easy observations.

\noindent {\it Observation 1.} A standard factorization gives rise to the reduced form of the rational function $\Be$: its numerator and denominator as monic polynomials are 
$$ (u-x_1)\cdots (u-x_m) (u-y_1) \cdots (u-y_n),\quad (u-z_1) \cdots (u-z_n) (u-w_1) \cdots (u-w_n). $$
This is because the two polynomials are coprime by Definition \ref{def: rational factorisation}.

\noindent {\it Observation 2.} If for all zero $a$ and pole $b$ of $\Be$ we have $b-a \notin \BN$, then the reduced form of $\Be$ is the unique standard factorization.

\noindent {\it Observation 3.} In Definition \ref{def: rational factorisation}, deleting a factor of the form $u-x_r, \frac{u-y_s}{u-z_s}$ or $\frac{1}{u-w_t}$, one gets another standard factorization.

\medskip

We prove the existence and uniqueness of standard factorization by induction on the number, denoted by $d(\Be)$, of zeroes and poles of $\Be$ counted with multiplicities. Namely, $d(\Be)$ is the degree of the numerator plus that of the denominator. The initial case $d(\Be) = 0$ is trivial since $\Be = 1$. 

Suppose $d(\Be) > 0$. If $\Be$ satisfies the hypothesis of Observation 2, then we conclude. Assume that there exist a zero $y_0$ and a pole $z_0$ of $\Be$ such that $z_0 - y_0 \in \BN$. Since there are finitely many such pairs, we assume further:
\begin{itemize}
\item[(H1)] If $y$ is a zero of $\Be$ and $z$ a pole, then $z-y \in \BN$ implies $z-y \geq z_0 - y_0$.
\end{itemize}
Set $\Bf := \frac{u-z_0}{u-y_0} \Be$. Since $\frac{u-y_0}{u-z_0}$ appears in the reduced form of $\Be$, it cancels with the factor $\frac{u-z_0}{u-y_0}$ and we get $d(\Bf) = d(\Be) - 2$. By induction hypothesis, the existence and uniqueness of standard factorization holds for $\Bf$. Fix such a factorization:
$$(F1):\qquad \Bf = \prod_{r=1}^m (u-x_r) \times \prod_{s=1}^n \frac{u-y_s}{u-z_s} \times \prod_{t=1}^k \frac{1}{u-w_t}. $$

\medskip

\noindent {\bf Step 1: existence.} We show that the following is a standard factorization:
$$ (F2): \qquad \Be =  \prod_{r=1}^m (u-x_r) \times \prod_{s=0}^n \frac{u-y_s}{u-z_s} \times \prod_{t=1}^k \frac{1}{u-w_t}. $$
In view of Definition \ref{def: rational factorisation}, it suffices to show that none of the following complex numbers belongs to $\Delta_{z_0}^{y_0}$ for $1\leq r \leq m,\ 1\leq s \leq n$ and $1\leq t \leq k$:
$$ z_0 - x_r,\quad z_0 - y_s,\quad  z_s - y_0 , \quad w_t - y_0.  $$
Let us prove it for the first number, the other three being parallel. Applying Observation 1 to the standard factorization of $\Be$, and noting that $d(\Be) = d(\Bf) + 2$, we see that the above factorization of $\Be$ is reduced. In particular, $x_r$ is a zero of $\Be$. If $z_0 - x_r  \in \Delta_{z_0}^{y_0}$, then by definition of $\Delta_{z_0}^{y_0} \subset \BN$ we have $z_0 - x_r < z_0 - y_0$, in contradiction with hypothesis (H1) which forces $z_0 - x_r \geq z_0 - y_0$.

\medskip

\noindent {\bf Step 2: uniqueness.} Let the following be a standard factorization:
$$(F3):\qquad \Be =  \prod_{r=1}^{m'} (u-x_r') \times \prod_{s=0}^{n'} \frac{u-y_s'}{u-z_s'} \times \prod_{t=1}^{k'} \frac{1}{u-w_t'}. $$
We claim that $(y_0, z_0) = (y_l', z_l')$ for certain $0\leq l \leq n'$. When this is the case, by Observation 3 we have another standard factorization of $\Bf$:
$$(F4):\qquad \Bf = \prod_{r=1}^{m'} (u-x_r') \times \prod_{0\leq s\leq n', s\neq l'}\frac{u-y_s'}{u-z_s'} \times \prod_{t=1}^{k'} \frac{1}{u-w_t'}.  $$
Applying the induction hypothesis to $(F1)$ and $(F4)$, we get that 
$$ \prod_{r=1}^{m'} (u-x_r') = \prod_{r=1}^m (u-x_r),\quad \prod_{t=1}^k (u-w_t) = \prod_{t=1}^{k'} (u-w_t') $$
and the pairs $(y_s', z_s')$ for $0\leq s \leq n', s \neq l'$ are in one-to-one correspondence with the $(y_s, z_s)$ for $1\leq s \leq n$. In other words, 
the standard factorizations $(F2)$ and $(F3)$ are the same after permutation.

To prove the claim, notice by Observation 1 that
$$ y_0 \in \{x_1',\cdots, x_{m'}', y_0', y_1', \cdots, y_{n'}' \},\quad z_0 \in \{z_0', z_1', \cdots, z_{n'}', w_1', \cdots, w_{k'}' \}. $$
Applying (H1) to the zero $y_s'$ and pole $z_s'$ we get:
\begin{itemize}
\item[(H2)] For $0 \leq s \leq n'$ we have $z_0 - y_0 \leq z_s' - y_s'$.
\end{itemize} 
In the standard factorization $(F3)$ we have $w_t' - x_r' \notin \BN$. So $z_0 - y_0 \in \BN$ forces $(y_0, z_0) \neq (x_r', w_t')$ for $1\leq r \leq m'$ and $1\leq t \leq k'$. There remain three cases.

Case 1: we have $(y_0, z_0) = (x_r', z_s')$ for certain $1\leq r \leq m'$ and $0 \leq s \leq n'$. Again in the standard factorization $(F3)$ we have  $z_s' - x_r' \notin \Delta_{z_s'}^{y_s'}$. Together with the assumption $z_s' - x_r'  = z_0 - y_0 > 0$ and the definition of $\Delta_{z_s'}^{y_s'}$ we have $z_s' - x_r'  = z_0 - y_0 \geq z_s' - y_s'$.  In view of (H2), equality holds. Now $z_0 = z_s'$ forces $y_0 = y_s'$.

Case 2: we have $(y_0, z_0) = (y_s', w_t')$ for certain $0 \leq s \leq n'$ and $1\leq t \leq k'$. Similar arguments as in Case 1 show $(y_0, z_0) = (y_s', z_s')$.

Case 3: we have $(y_0, z_0) = (y_s', z_l')$ for certain $0 \leq s, l \leq n'$. From the condition $z_0 - y_0 = z_l' - y_s' \notin \Delta_{z_s'}^{y_s'} \cap \Delta_{z_l'}^{y_l'}$ imposed by $(F3)$ we get either $z_0 - y_0\geq z_s' - y_s'$ or $z_0 - y_0  \geq z_l' - y_l'$. In both situations, equality holds by (H2). Applying $y_0 = y_s'$ to the first situation and $z_0 = z_l'$ to the second situation, we get either $(y_0, z_0) = (y_s', z_s')$ or $(y_0, z_0) = (y_l', z_l')$.
\end{proof}
We point out that the existence arguments follow closely \cite[Prop.3.6]{M}. It is the uniqueness that is the key point of Lemma \ref{lem: uniqueness factorization}.

Proposition \ref{prop: factorization sl2} together with Lemma \ref{lem: uniqueness factorization} 
implies the following.

\begin{theorem}\label{rem: uniqueness} All simple module $L(\Be)$ in category $\BGG^{sh}$ factorizes uniquely as a tensor product of prefundamental modules $L_{a}^{\pm}$ ($a\in \BC$) and of Kirillov--Reshetikhin modules $L_b^a$ ($a,b\in \BC$, $0 < a-b \in \BN$). 
\end{theorem}

\begin{example} Let us revisit the example after Theorem \ref{thm: Tarasov},
$$L_0^9 \otimes L_2^3  \cong L_2^9 \otimes L_0^3  \cong L_9^+ \otimes L_3^+ \otimes L_0^- \otimes L_2^-. $$ 
In the subcategory $\BGG_0$ of $\BGG^{sh}$, the irreducible module $L_b^a$ for $b-a\notin \BN$ is prime in the sense that if $L_b^a \cong V \otimes W$ in category $\BGG_0$ then either $V$ or $W$ is the one-dimensional trivial module. The first isomorphism forms two non-equivalent factorizations into primes of the same irreducible module. The issue of non-uniqueness is resolved in category $\BGG^{sh}$ by further factorizing $L_b^a \cong L_a^+ \otimes L_b^-$. 
\end{example}

\section{Computation of diagonal entries}  \label{sec: pre R}
In this section we compute the diagonal entries introduced in Definition \ref{def: diagonal entries} for the R-matrix $\check{R}_{V,W}(u)$, where $V$ is a finite-dimensional irreducible module and $W$ is a negative module (Proposition \ref{prop: highest diagonal entry} and Theorem \ref{thm: TQ pre R}). 
A technical point in the proofs is a refined estimation of the coproduct that we establish in Lemma \ref{lem: coproduct estimation refine}. 

\subsection{Second coproduct estimation} As a preparatory step, we refine the coproduct estimation of Lemma \ref{lem: coproduct estimation} for the Drinfeld--Cartan series $\xi_i(u)$.
In the ordinary Yangian $Y(\Glie)$ hold the relations:
\begin{align}
[x_j^+(u), x_{j,0}^-] &= [x_{j,0}^+, x_j^-(u)] = \langle\xi_j(u)\rangle_+ = \xi_j(u) - 1,\label{GTL: x+ et x-} \\
 [\xi_i(u), x_{j,0}^-] &= - 2 d_{ij} \xi_i(u) x_j^-(u-d_{ij}) = - 2 d_{ij} x_j^-(u+d_{ij}) \xi_i(u). \label{GTL: xi et x-}
\end{align}
The first relation follows from \eqref{rel: Cartan}. The second is obtained by taking specializations $v = u + d_{ij}$ and $v = u - d_{ij}$ of the relation \cite[\S 2.4]{GTL}: 
$$(u-v+d_{ij}) \xi_i(u) x_j^-(v) - (u-v- d_{ij}) x_j^-(v) \xi_i(u) = - [\xi_i(u), x_{j,0}^-].$$

\begin{lem}  \label{lem: coproduct estimation refine}
For all coweights $\mu$ and $\nu$, the coproduct $\Delta_{\mu,\nu}$ satisfies:
\begin{align*}
\Delta_{\mu,\nu}(\xi_i(u)) \equiv & \ \xi_i(u) \otimes \xi_i(u) - \sum_{j\in I} 2d_{ij}x_j^-(u+d_{ij}) \xi_i(u) \otimes \xi_i(u) x_j^+(u+d_{ij}) \\
& \quad \mathrm{mod}.\ \sum_{h(\beta) \geq 2} Y_{\mu}^-(\Glie)_{-\beta} \otimes Y_{\nu}^+(\Glie)_{\beta}.
\end{align*}
\end{lem}
\begin{proof}
One adapts the zigzag arguments of \cite[Theorem 4.12]{coproduct} to reduce to the case $\mu = \nu = 0$, as in the proof of Lemma \ref{lem: coproduct estimation}. For $\alpha \in \BQ$, let $\pi_{\alpha}$ denote the projection of $Y(\Glie)$ onto the weight space $Y(\Glie)_{\alpha}$. It suffices to prove for $i, j \in I$:
\begin{equation} \label{equ: coproduct xi} 
(\pi_{-\alpha_j} \otimes \pi_{\alpha_j}) \circ \Delta (\xi_i(u)) = - 2d_{ij} x_j^-(u+d_{ij}) \xi_i(u) \otimes \xi_i(u) x_j^+(u+d_{ij}).
\end{equation}
The strategy is to produce a system of linear equations which will have a unique solution, given by both sides of this equation. Let 
$$A(u) := \sum_{p\geq -1} A_p u^{-p-1}$$ 
denote the power series at the left-hand side. We view the coefficients $A_p$ as elements in $ Y(\Glie)_{-\alpha_j} \otimes Y(\Glie)_{\alpha_j}$. From Eq.\eqref{eq: coproduct Yangian} we get $A_{-1} = A_0 = 0$.

Next, applying $\Delta$ to the second formula of $[\xi_i(u), x_{j,0}^-]$ from Eq.\eqref{GTL: xi et x-}, taking into account $\Delta(x_{j,0}^-) = 1 \otimes x_{j,0}^- + x_{j,0}^- \otimes 1$ and Lemma \ref{lem: coproduct estimation}:
\begin{align*}
\Delta(x_j^-(u)) \equiv 1 \otimes x_j^-(u) + x_j^-(u) \otimes \xi_j(u) \ \mathrm{mod}.\ \sum_{h(\beta) > 0} Y(\Glie)_{-\beta-\alpha_j} \otimes Y(\Glie)_{\beta},
\end{align*}
after projection onto the weight space of bi-weight $(-\alpha_j,0)$ we obtain:
\begin{align*}
& [\xi_i(u)\otimes \xi_i(u), x_{j,0}^- \otimes 1] + [A(u), 1 \otimes x_{j,0}^-] \\
& = -2d_{ij} (1 \otimes x_j^-(u+d_{ij})) A(u) -2d_{ij} x_j^-(u+d_{ij}) \xi_i(u) \otimes \xi_j(u+d_{ij}) \xi_i(u).
\end{align*}
Making use of Eqs.\eqref{GTL: x+ et x-}--\eqref{GTL: xi et x-} we simplify the equality as follows:
\begin{multline}  \label{recursion}
\quad \quad [A(u), 1 \otimes x_{j,0}^-] + (2d_{ij}\otimes x_j^-(u+d_{ij})) A(u) \\
= -2d_{ij} x_j^-(u+d_{ij}) \xi_i(u) \otimes \langle\xi_j(u+d_{ij})\rangle_+ \xi_i(u).
\end{multline}
Eq.\eqref{recursion} forms a linear system whose unknown variables are the $A_p$ for $p \geq -1$. It expresses $[A_p, 1 \otimes x_{j,0}^-]$ in terms of the $A_m$ for $m < p$.  Therefore, the system has a unique solution provided that the following linear map is injective:
$$Y(\Glie)_{-\alpha_j} \otimes Y(\Glie)_{\alpha_j} \longrightarrow Y(\Glie)_{-\alpha_j} \otimes Y(\Glie)_0, \quad a \mapsto [a, 1 \otimes x_{j,0}^-]. $$
This map is the restriction of $-\mathrm{Id} \otimes \mathrm{ad}_{x_{j,0}^-}$. It suffices to establish the injectivity of $\mathrm{ad}_{x_{j,0}^-}$ restricted to $Y(\Glie)_{\alpha_j}$. Note that $x_{j,0}^{-}, \frac{1}{d_j} \xi_{j,0}, \frac{1}{d_j} x_{j,0}^+$ span a sub-Lie-algebra of $Y(\Glie)$ isomorphic to $sl_2$. The adjoint action of $sl_2$ on $Y(\Glie)$ is integrable by the Serre relation \eqref{rel: Serre}. If $w \in \ker (\mathrm{ad}_{x_{j,0}^-}) \cap Y(\Glie)_{\alpha_j}$ is nonzero, then $w$ is a vector of lowest weight $\frac{1}{d_j} (\alpha_j,\alpha_j) = 2$, contradicting the integrable representation theory of $sl_2$. 

It remains to show that the right-hand side is a solution to Eq.\eqref{recursion}. This follows from Eqs.\eqref{GTL: x+ et x-}--\eqref{GTL: xi et x-} and commutativity of the $\xi_i(u)$.
\end{proof}
When $\Glie = sl_2$, Lemma \ref{lem: coproduct estimation refine} agrees  with the term $k = 1$ of the coproduct formula of $\Delta(h(u))$ in \cite[Definition 2.24]{Molev} and \cite[(6.9)]{coproduct}.

\subsection{Highest diagonal entry} 
Let $V$ be an irreducible module in category $\BGG^{sh}$ and $W$ be a negative module as in Definition \ref{def: diagonal entries}. We identify the highest diagonal entry $s_{V,W}(u)$ with R-matrices of Proposition \ref{prop: one-dim R+}. When $V$ is finite-dimensional, this leads to a formula for the polynomial $\lambda_{V,W}(u)$ in terms of $\ell$-weights of $V$ and $W$.

\begin{prop} \label{prop: highest diagonal entry}
Let $\Br = \Psi_{i_1,a_1} \Psi_{i_2,a_2} \cdots \Psi_{i_N,a_N} \in \CD$ and $V$ be a highest $\ell$-weight irreducible module over $Y_{\nu}(\Glie)$.  Then the highest diagonal entry of $\check{R}_{V,L(\Br^{-1})}(u)$ is 
$$ s_{V,L(\Br^{-1})}(u) = R_{i_1}^V(a_1-u) R_{i_2}^V(a_2-u) \cdots R_{i_N}^V(a_N-u). $$
\end{prop}
\begin{proof}
Write $W = L(\Br^{-1})$ and $\omega = \omega_{\Br^{-1}}$. Chose a highest $\ell$-weight vector $v_0$ of $V$. 
Since both sides are linear maps from $V$ to $V \otimes \BC[u]$, it suffices to prove the equality specialized at an arbitrary complex number $a \in \BC$. 

Set $v = v_0$ in Eq.\eqref{matrix coefficient highest}. Since $\check{R}_{V,W}(a)$ sends $v_0 \otimes \omega$ to $\omega \otimes v_0$, we get 
$$ s_{V,W}(a) v_0 = v_0. $$
Next, for $i \in I$ and $v' \in V$, from Eq.\eqref{rel: key} we obtain the following relation in the module $V(a) \otimes W$:
$$ x_i^-(u) v' \otimes \omega =  \langle \Br_i(u) x_i^-(u)\rangle_+ (v' \otimes \omega).  $$ 
Applying the module morphism $\check{R}_{V,W}(a): V(a) \otimes W \longrightarrow W \otimes V(a)$ gives 
\begin{equation*} 
\check{R}_{V,W}(a) (x_i^-(u) v' \otimes \omega) =  \langle \Br_i(u) x_i^-(u)\rangle_+ \check{R}_{V,W}(a) (v' \otimes \omega).
\end{equation*}
We compute the components of $\omega \otimes V(a)$ at both sides. By definition the left-hand side is $\omega \otimes s_{V,W}(a) x_i^-(u) v'$. At the right-hand side, $\check{R}_{V,W}(a)(v' \otimes \omega)$ is $\omega \otimes s_{V,W}(a)v'$ plus a linear combination of vectors in $W_{\gamma} \otimes V(a)$ where $\gamma \in \wt(\omega) + \BQ_-$ and $\gamma \neq \wt(\omega)$. By Lemma \ref{lem: coproduct estimation}, the coproduct of $x_i^-(u)$ is $1 \otimes x_i^-(u)$ plus a linear combinations tensor products of elements of shifted Yangians such that the weight of each first tensor factor is in $\BQ_-\setminus \{0\}$. For weight reason the desired component is 
$$ \omega \otimes \langle \Br_i(u) x_i^-(u)  \rangle_+ s_{V,W}(a) v'.  $$
So we have the following equality in the module $V(a)$:
$$ s_{V,W}(a) x_i^-(u) v' = \langle \Br_i(u) x_i^-(u) \rangle_+ s_{V,W}(a) v'.  $$
In the module $V$ the equality becomes 
$$ s_{V,W}(a) x_i^-(u-a) v' = \langle \Br_i(u) x_i^-(u-a) \rangle_+ s_{V,W}(a) v'. $$
Replacing $u$ by $u+a$, we obtain the following commutation relation in the module $V$:
$$ s_{V,W}(a) x_i^-(u) = \langle \Br_i(u+a) x_i^-(u)\rangle_+ s_{V,W}(a) \quad \mathrm{for}\ i \in I. $$
Combining with $s_{V,W}(a) v_0 = v_0$ we obtain all the defining properties of the operator $R_{\tau_{-a}(\Br)}^V$ in Eq.\eqref{Baxter: recursion}. Here we recall from Remark \ref{rem: spectral weight} the one-parameter family of group automorphisms $\tau_b: \CL \longrightarrow \CL$ for $b \in \BC$. By uniqueness $s_{V,W}(a) = R_{\tau_{-a}(\Br)}^V$ and 
$$ s_{V,W}(a) = R_{\Psi_{i_1,a_1-a}\Psi_{i_2,a_2-a} \cdots \Psi_{i_N,a_N-a}}^V = R_{i_1}^V(a_1-a) R_{i_2}^V(a_2-a) \cdots R_{i_N}^V(a_N-a). $$ 
This is exactly the equality of the proposition evaluated at $u = a$.
\end{proof}
Proposition \ref{prop: highest diagonal entry} implies that $R_i^V(-u)$ is the highest diagonal entry of $\check{R}_{V, L_{i,0}^-}(u)$. This is dual to the Borel situation of Remark \ref{rem: transfer matrix} (ii).
\begin{theorem}  \label{thm: truncation polynomial}
Let $\Bs \in \CD$ and $V$ be a finite-dimensional irreducible module in category $\BGG^{sh}$. Write the ratio of the highest $\ell$-weight to the lowest $\ell$-weight of $V$ as a monomial of the $A_{j,b}$, and replace each $A_{j,b}$ with the polynomial $\Bs_j(u+b)$, then we obtain $\lambda_{V,L(\Bs^{-1})}(u)$. In particular, the polynomial $\lambda_{V,L(\Bs^{-1})}(u) \in \BC[u]$ is monic.
\end{theorem}
\begin{proof}
Setting $v$ in Eq.\eqref{matrix coefficient highest} to be a lowest $\ell$-weight vector $v_-$ of $V$, and $w = \omega_{\Bs^{-1}}$ in Eq.\eqref{matrix coefficient lowest}, we see that $\lambda_{V,L(\Bs^{-1})}(u)$ is also the eigenvalue of the highest diagonal entry $s_{V,L(\Bs^{-1})}(u): V \longrightarrow V \otimes \BC[u]$ associated to the lowest $\ell$-weight vector $v_-$ of $V$.

Let $\prod_{i\in I} \prod_{t=1}^{h_i} A_{i,a_{i,t}}$ be the ratio of highest $\ell$-weight to lowest $\ell$-weight of $V$. By Proposition \ref{prop: highest diagonal entry}, each factor $u-b$ of the polynomial $\Bs_i(u)$ gives rise to a factor $R_i^V(b-u)$ of $s_{V,L(\Bs^{-1})}(u)$. The eigenvalue of $R_i^V(b-u)$ associated to $v_-$ is $\prod_{t=1}^{h_i}(a_{i,t}-b+u)$ by Proposition \ref{prop: l-weight R-matrix+} (ii). Each component $\Bs_i(u)$ of $\Bs$ gives rise to a factor $\prod_{t=1}^{h_i} \Bs_i(u+a_{i,t})$ of the eigenvalue of $s_{V,L(\Bs^{-1})}(u)$ . After taking product over all $i \in I$, we get 
$$ \lambda_{V,L(\Bs^{-1})}(u) = \prod_{i\in I} \prod_{t=1}^{h_i} \Bs_i(u+a_{i,t}). $$
\end{proof}

\subsection{Lowest diagonal entry}
We express the lowest diagonal entry $t_{V,W}(u)$ from Definition \ref{def: diagonal entries} in terms of one-dimensional R-matrices from Proposition \ref{prop: one-dim R+}, assuming that $V$ is a finite-dimensional irreducible  module over the ordinary Yangian. The idea is to reduce to the case of a fundamental module by a fusion construction of R-matrices, which appeared for example in \cite[Corollary 5.5]{H2}.

\begin{theorem} \label{thm: TQ pre R}
Let $W$ be a negative module, and $V$ be a finite-dimensional irreducible $Y(\Glie)$-module whose lowest $\ell$-weight is 
$$ Y_{i_1,a_1+\frac{1}{2}d_{i_1}}^{-1} Y_{i_2,a_2+\frac{1}{2}d_{i_2}}^{-1} \cdots Y_{i_m,a_m+\frac{1}{2}d_{i_m}}^{-1}.  $$
Then we have the following equality in $\mathrm{Hom}(W, W \otimes \BC[u])$:
\begin{equation}  \label{TQ pre R}
 t_{V,W}(u) \prod_{s=1}^m R_{i_s}^W(u+a_s) = \lambda_{V,W}(u) \prod_{s=1}^m R_{i_s}^W(u+a_s+d_{i_s}).
\end{equation}
Moreover, $t_{V,W}(u)$ is an $\mathrm{End} (W)$-valued monic polynomial of degree $\deg \lambda_{V,W}(u)$.
\end{theorem}
\begin{proof}
Let $\omega$ denote a highest $\ell$-weight vector of the module $W$ defined over $Y_{\mu}(\Glie)$. Since $W$ is an irreducible module in category $\BGG_{\mu}$, by Proposition \ref{prop: l-weight R-matrix+} (ii) there exists a countable subset $\Gamma \subset \BC$ such that the normalized q-character of $W$ is power series in the $A_{j,b}^{-1}$ with $(j,b) \in I \times \Gamma$. Notably, $a \in \BC \setminus \Gamma$ and $\Bf \in \lwt(W)$ imply $ A_{i,a}^{-1} \Bf \notin \lwt(W)$. 

Restricted to each weight space $W_{\beta}$ of $W$, by Propositions \ref{prop: polynomiality R-} and \ref{prop: one-dim R+}, $R_i^W(-u)$ is a monic polynomial of $u$ and $t_{V,W}(u)$ is a polynomial. Eq.\eqref{TQ pre R} implies that the degree of $t_{V,W}(u)|_{W_{\beta}}$ is $\deg \lambda_{V,W}(u)$, which is independent of the weight $\beta$. This proves the second part of the theorem assuming Eq.\eqref{TQ pre R}.
By polynomiality, it suffices to prove Eq.\eqref{TQ pre R} evaluated at $u = a$ for $a \in \BC$  such that $a+a_s \notin \Gamma$ for all $1\leq s \leq m$. 
 
 \medskip

\noindent {\it Step 1: reduction to the fundamental case}. For $1\leq s \leq m$, let $v_s$ and $v_-^s$ denote a highest $\ell$-weight vector and a lowest $\ell$-weight vector of the fundamental module $V_{i_s}$ of Eq.\eqref{def: fund module}. By Theorem \ref{thm: normalized R fin dim} (i), one may assume, after a permutation of the pairs $(i_s,a_s)$, that $V$ is the irreducible submodule of the following tensor product
$$ T := V_{i_1}(a_1) \otimes V_{i_2}(a_2) \otimes \cdots \otimes V_{i_m}(a_m) $$
generated by $v_0 := v_1 \otimes v_2 \otimes \cdots \otimes v_m$ and $v_- := v^1_- \otimes v_-^2  \otimes \cdots \otimes v_-^m$. In particular, $v_0$ and $v_-$ are highest $\ell$-weight vector and lowest $\ell$-weight vector of $V$.

From the fact that $\tau_b$ is a Hopf algebra automorphism of $Y(\Glie)$ and from the equation $\tau_b \tau_c  = \tau_{b+c}$, for $b, c \in \BC$, we get an identification of modules
$$ T(a) = V_{i_1}(a_1+a) \otimes V_{i_2}(a_2+a) \otimes \cdots \otimes V_{i_m}(a_m+a). $$
Consider the following composite map
\begin{align*}
R(a) =  \prod_{s=1}^m (1^{\otimes s-1} \otimes \check{R}_{V_{i_s},W}(a+a_s) \otimes 1^{\otimes m-s}): T(a) \otimes W \longrightarrow W \otimes T(a).
\end{align*}
Since all modules are defined over antidominantly shifted Yangians, it follows from the monoidality of category $\BGG_-^{sh}$ in Remark \ref{rem: monoidal} that $R(a)$ is a module morphism from $T(a) \otimes W$ to $W \otimes T(a)$. It restricts to a module morphism from $V(a) \otimes W$ to $W \otimes V(a)$ because the former is generated by the highest $\ell$-weight vector $v_0 \otimes \omega$ by Theorem \ref{thm: cyclicity cocyclicity} (iii) and $R(a)$ sends $ v_0 \otimes \omega$ to $\omega \otimes v_0$ by definition of the $\check{R}_{V_{i_s},W}(a+a_s)$. From uniqueness of R-matrix in Theorem \ref{thm: R-matrix}, we obtain 
$$R(a)|_{V(a) \otimes W} = \check{R}_{V,W}(a). $$ 
For $w \in W$, by definition $t_{V,W}(a) w \otimes v_-$ is the projection to $W \otimes v_-^1 \otimes v_-^2 \otimes \cdots \otimes v_-^m$ of the vector $R(a) (v_-^1 \otimes v_-^2 \otimes \cdots \otimes v_-^m \otimes w)$, which by the factorization of $R(a)$ is 
$$t_{V_{i_1},W}(a+a_1)  t_{V_{i_2},W}(a+a_2)  \cdots t_{V_{i_m}, W}(a+a_m) w \otimes v_-^1 \otimes v_-^2 \otimes \cdots \otimes v_-^m. $$
We obtain therefore a factorization:
$$t_{V,W}(a) = t_{V_{i_1},W}(a+a_1)  t_{V_{i_2},W}(a+a_2)  \cdots  t_{V_{i_m}, W}(a+a_m) \in \mathrm{End} (W).  $$
We are reduced to compute the lowest diagonal entry $t_{V_i,W}(u)$ for a fixed $i \in I$. Assume from now on that $m = 1,\ (i_1,a_1) = (i,0)$ and $a \in \BC \setminus \Gamma$. 

 By Lemma \ref{lem: lowest weight fund}, the finite-dimensional $Y(\Glie)$-module $V$ contains a $\Glie$-submodule of lowest weight $-\varpi_i$ and $V_{-\varpi_i} = \BC v_-$.  By Weyl group symmetry $v_+ := x_{i,0}^+ v_-$ spans the weight space $V_{\alpha_i-\varpi_i}$. Complete $v_{\pm}$ to a weight basis $\mathcal{B}_V$ of $V$. If $v \in \mathcal{B}_V \setminus \{v_-, v_+\}$, then $0 \neq \wt(v) - \alpha_i+\varpi_i \in \BQ_+$. Since $V$ is defined over the ordinary Yangian, by Remark \ref{rem: spectral weight} it has the same weight grading as $V(a)$. We have the following relations in the module $V(a)$ similar to Example \ref{ex: two-dim}:
\begin{gather}  \label{fund module action}
\begin{split}
x_j^+(u) v_- &= \frac{\delta_{ij}}{u-a} v_+,\quad \xi_j(u) v_+ = \frac{u-a-d_i \delta_{ij}}{u-a} \frac{u-a+d_{ij}}{u-a-d_{ij}} v_+, \\
x_j^-(u) v_+ &= \frac{d_i \delta_{ij}}{u-a} v_-, \quad \xi_j(u) v_- = \frac{u-a-d_i \delta_{ij}}{u-a} v_-.
\end{split}
\end{gather} 
From the paragraph above Definition \ref{def: diagonal entries}, we get a family of linear operators $f_v(u): W \longrightarrow W\otimes \BC[u]$ for $v \in \mathcal{B}_V$ such that for $a \in \BC$ and $w \in W$ we have
$$ \check{R}_{V,W}(a) (v_- \otimes w) = \sum_{v \in \mathcal{B}_V} f_v(a) w \otimes v $$ 
In particular $f_{v_-}(u) = t_{V,W}(u)$ by Eq.\eqref{matrix coefficient highest}. To simplify notations, write
\begin{gather*}
R_a := \check{R}_{V,W}(a), \quad \lambda_a  := \lambda_{V,W}(a), \quad  C_a := f_{v_+}(a), \quad D_a := f_{v_-}(a), \\
\tilde{C}_a := R_{\Psi_{i,a}}^W C_a,\quad \tilde{D}_a := R_{\Psi_{i,a}}^W D_a,\quad \tilde{E}_a := \tilde{C}_a + x_{i,0}^- \tilde{D}_a. 
\end{gather*}
Our goal is to show that for $a \in \BC \setminus \Gamma$ we have $\tilde{D}_a = \lambda_a R_{\Psi_{i,a+d_i}}^W$ as linear operators on $W$. By polynomiality, this would imply that $R_i^W(u) t_{V,W}(u) = \lambda_{V,W}(u) R_i^W(u+d_i)$. From Proposition \ref{prop: one-dim R+} we get that the operators $t_{V,W}(u), R_i^W(u)$ for $i \in I$ acting on $W$ mutually commute, so that the order of the products in Eq.\eqref{TQ pre R} does not matter.

\medskip

\noindent {\it Step 2: projection formulas}. The $Y(\Glie)$-modules $V$ and $V(a)$ have the same weight grading. With respect to the weight basis $\mathcal{B}_V$ of $V(a)$, let 
$$\phi: W \otimes V(a) \longrightarrow W\text{ and }\psi: W \otimes V(a) \longrightarrow W$$ denote linear maps which send $\sum_{v\in \mathcal{B}_V} g_v \otimes v$ to $g_{v_-}$ and to $g_{v_+}$ respectively.
We apply these maps to the following relations in $W\otimes V(a)$ for $j \in I$ and $w\in W$:
\begin{align*}
   R_a x_j^-(u) (v_- \otimes w) &= x_j^-(u) R_a(v_- \otimes w), \quad  R_a \xi_j(u) (v_- \otimes w)  = \xi_j(u) R_a(v_- \otimes w).
\end{align*}
By Example \ref{example: V W}, the left-hand sides of the above equations are $R_a(v_- \otimes x_j^-(u) w)$ and $R_a(\xi_j(u) v_- \otimes \xi_j(u) w)$ respectively.
Based on the coproduct estimations of $\Delta_{\mu,0}(x_j^-(u))$ from Lemma \ref{lem: coproduct estimation} and of $\Delta_{\mu,0}(\xi_j(u))$ from Lemma \ref{lem: coproduct estimation refine}, we have
\begin{align*}
   \phi(R_a x_j^-(u) (v_- \otimes w)) &= \phi(R_a(v_- \otimes x_j^-(u) w)) = D_a x_j^-(u) w, \\
    \phi( x_j^-(u) R_a(v_- \otimes w)) &= \phi(x_j^-(u)(\sum_{v\in \mathcal{B}_V} f_v(a)w \otimes v)) \\
    &= \phi(x_j^-(u) D_a(w) \otimes \xi_j(u) v_-) + \phi(C_a(w) \otimes x_j^-(u) v_-) \\
    &= \frac{u-a-d_i \delta_{ij}}{u-a} x_j^-(u) D_a(w) + \frac{d_i \delta_{ij}}{u-a} C_a(w), \\ 
    \phi(R_a \xi_j(u) (v_- \otimes w)) &= \phi(R_a(\xi_j(u) v_- \otimes \xi_j(u) w)) = \frac{u-a-d_i\delta_{ij}}{u-a} D_a \xi_j(u) w, \\
    \phi(\xi_j(u) R_a (v_- \otimes w)) &= \phi(\xi_j(u)(\sum_{v\in \mathcal{B}_V} f_v(a)w \otimes v)) \\
    &= \phi(\xi_j(u) D_a(w) \otimes \xi_j(u)v_-) = \frac{u-a-d_i\delta_{ij}}{u-a} \xi_j(u) D_a (w), \\
    \psi(R_a \xi_j(u) (v_- \otimes w)) &=  \psi(R_a(\xi_j(u) v_- \otimes \xi_j(u) w)) = \frac{u-a-d_i\delta_{ij}}{u-a} C_a \xi_j(u) w,\\
    \psi(\xi_j(u) R_a (v_- \otimes w)) &= \psi(\xi_j(u)(\sum_{v\in \mathcal{B}_V} f_v(a)w \otimes v)) = \psi(\xi_j(u)C_a(w) \otimes \xi_j(u) v_+) \\
    &\qquad - 2d_{ij} \psi( x_i^-(u+d_{ij}) \xi_j(u) D_a(w) \otimes \xi_j(u) x_i^+(u+d_{ij}) v_-) \\
    &= \frac{u-a-d_i\delta_{ij}}{u-a}  \frac{u-a+d_{ij}}{u-a-d_{ij}}\xi_j(u) C_a(w) \\
    &\quad -\frac{u-a-d_i\delta_{ij}}{u-a} \frac{2d_{ij}}{u-a-d_{ij}} x_i^-(u+d_{ij}) \xi_j(u) D_a(w). 
\end{align*}

\medskip

\noindent {\it Step 3: commutation relations}. It follows from the projection formulas that 
\begin{align}  
D_a x_j^-(u) &=  \frac{u-a-d_i\delta_{ij}}{u-a} x_j^-(u) D_a + \frac{d_i\delta_{ij}}{u-a} C_a,  \label{1} \\
D_a \xi_j(u) &= \xi_j(u) D_a, \label{2} \\
C_a \xi_j(u) &= \frac{u-a+d_{ij}}{u-a-d_{ij}}\xi_j(u) C_a - \frac{2d_{ij}}{u-a-d_{ij}} x_i^-(u+d_{ij}) \xi_j(u) D_a. \label{3}
\end{align}
Notice that $x_i^-(u+d_{ij}) \xi_j(u) = \xi_j(u) x_i^-(u-d_{ij})$ as a rewriting of the second half Eq.\eqref{GTL: xi et x-}. Left multiplying Eq.\eqref{3} by $\xi_j(u)^{-1}$, we obtain
\begin{equation*}
    \xi_j(u)^{-1} C_a \xi_j(u) = \frac{u-a+d_{ij}}{u-a-d_{ij}}C_a - \frac{2d_{ij}}{u-a-d_{ij}} x_i^-(u-d_{ij})D_a.
\end{equation*}
Left multiplying the above equation by $R_{\Psi_{i,a}}^W$, which commutes with $\xi_j(u)$, we get 
\begin{align*}
\xi_j(u)^{-1} \tilde{C}_a \xi_j(u) &=  \frac{u-a+d_{ij}}{u-a-d_{ij}} \tilde{C}_a - \frac{2d_{ij}}{u-a-d_{ij}} \langle (u-a-d_{ij}) x_i^-(u-d_{ij}) \rangle_+ \tilde{D}_a \\
&= \frac{u-a+d_{ij}}{u-a-d_{ij}} (\tilde{C}_a + x_{i,0}^- \tilde{D}_a) - 2 d_{ij} x_i^-(u-d_{ij}) \tilde{D}_a - x_{i,0}^- \tilde{D}_a.
\end{align*}
Combining with Eq.\eqref{2} and the  relation $\xi_j(u)^{-1} x_{i,0}^- \xi_j(u) = x_{i,0}^- + 2 d_{ij} x_i^-(u-d_{ij})$,
which rewrites the first half of Eq.\eqref{GTL: xi et x-}, we obtain
\begin{equation} \label{CD auxiliary}
\xi_j(u)^{-1} \tilde{E}_a \xi_j(u) = \frac{u-a+d_{ij}}{u-a-d_{ij}} \tilde{E}_a.
\end{equation}

Recall that $a \in \BC \setminus \Gamma$. Let $w \in W$ be a vector of $\ell$-weight $\Bf$.  If the vector $\tilde{E}_a(w)$ is nonzero, then it is of $\ell$-weight $ A_{i,a}^{-1}\Bf$ by Eqs.\eqref{CD auxiliary} and \eqref{def: simple root}, contradicting our choice of $\Gamma$. So $\tilde{E}_a(w) = 0$ for all $\ell$-weight vectors $w \in W$ and $\tilde{E}_a  = 0$.

Let us multiply Eq.\eqref{1} by $R_{\Psi_{i,a}}^W$. Making use of the defining properties Eq.\eqref{Baxter: recursion}, we recover all the defining properties of $\lambda_a R_{\Psi_{i,a+d_i}}^W$:
\begin{gather*}
 \tilde{D}_a(\omega) = \lambda_a\omega,\quad  \tilde{D}_a x_j^-(u) = x_j^-(u) \tilde{D}_a \quad \mathrm{for}\ j \neq i, \\
\tilde{D}_a x_i^-(u) = \langle (u-a-d_i) x_i^-(u) \rangle_+ \tilde{D}_a + \frac{d_i}{u-a}\tilde{E}_a = \langle (u-a-d_i) x_i^-(u) \rangle_+ \tilde{D}_a. 
\end{gather*}
This proves the desired identity $\tilde{D}_a = \lambda_a R_{\Psi_{i,a+d_i}}^W$ for $a \in \BC \setminus \Gamma$. 
\end{proof}

\begin{example}  \label{ex: R pre sl2}
Let $\Glie = sl_2$. Consider $\check{R}_{V,W}(a)$ with $V = N_0$ and $W = L_0^-$. Example \ref{example: polynomiality sl2} gives $\lambda_{V,W}(u) = u$. Furthermore, $v_i$ is an eigenvector of $R_1^W(u)$ of eigenvalue $ -u(-u-1)(-u-2) \cdots (-u-i+1)$ by Proposition \ref{prop: l-weight R-matrix+} (ii) and Example \ref{ex: - prefund sl2}. From Eq.\eqref{TQ pre R} and its proof we obtain that 
\begin{align*}
 D_a(v_i) &= a \frac{R_1^W(a+1)}{R_1^W(a)} v_i = a \frac{(a+1)(a+2)\cdots (a+i)}{a(a+1) \cdots (a+i-1)} v_i = (a+i) v_i, \\
C_a(v_i) &= - R_1^W(a)^{-1} x_0^-  R_1^W(a) t_{e_2,e_2}^W(a)  v_i = (i+1) v_{i+1}.
\end{align*}
This gives $\check{R}_{V,W}(a) (e_2 \otimes v_i) = (a+i) v_i \otimes e_2 + (i+1) v_{i+1} \otimes e_1$. Apply $x_0^+$ to the equality and notice that $\Delta_{0,-1}(x_0^+) = x_0^+ \otimes 1+1 \otimes x_0^+$ and $\Delta_{-1,0}(x_0^+) = x_0^+ \otimes 1$:
\begin{gather*}
 x_0^+ (e_2 \otimes v_i) = x_0^+ e_2 \otimes v_i +  e_2 \otimes x_0^+  v_i = e_1 \otimes v_i + e_2 \otimes v_{i-1}, \\
 x_0^+ (v_i \otimes e_2) = x_0^+ v_i \otimes e_2 = v_{i-1} \otimes e_2,  \quad x_0^+ (v_{i+1} \otimes e_1) = x_0^+v_{i+1} \otimes e_1 =v_i \otimes e_1.
\end{gather*}
We obtain from the commutativity of $\check{R}_{V,W}(a)$ with $x_0^+$ that 
\begin{align*}
&  \check{R}_{V,W}(a) (v_i \otimes e_1 + e_2 \otimes v_{i-1})  = (a+i) v_{i-1} \otimes e_2 + (i+1) v_i \otimes e_1, \\
& \check{R}_{V,W}(a)(v_i \otimes e_1) = e_1 \otimes v_i + e_2 \otimes v_{i-1}.
\end{align*}
With respect to the basis $(e_1, e_2)$ of $V$ we get 
$$ \check{R}_{V,W}(u) = \begin{pmatrix}
1 & \mathbf{a}^+ \\
\mathbf{a}^- &  u + \mathbf{a}^+ \mathbf{a}^-
\end{pmatrix} \quad \mathrm{where}\ \begin{cases} 
\mathbf{a}^+(v_i) = (i+1)v_{i+1}, \\
\mathbf{a}^-(v_i) = v_{i-1}.
\end{cases} $$
This is a monodromy matrix of Baxter's Q-operator for $Y(gl_2)$; see \cite[(3.38)]{B}. The Yang--Baxter equation \cite[(3.1)]{B} is a particular case of Eq.\eqref{YBE}.
\end{example}
In Definition \ref{def: diagonal entries}, we think of $\check{R}_{V,W}(u)$ as a monodromy matrix. Taking a suitable trace over $V$ gives a transfer matrix acting on $W$, and $t_{V,W}(u)$ is a leading term of the transfer matrix. We expect Eq.\eqref{TQ pre R} to be a leading term of generalized Baxter's relations for transfer matrices \cite[Theorem 5.11]{FH}.   
\section{Truncations of standard modules}\label{sectrunc}
In this section, we prove that any standard module (and so any irreducible module) in category $\BGG^{sh}$ factorizes through a truncated shifted Yangian (Theorem \ref{thm: truncation standard}). Our proof is uniform for all finite types (see Introduction for a discussion of known results).

Recall from Eq.\eqref{def: fund module} and the paragraph above the involution $i \mapsto \overline{i}$ on $I$, the rational number $\kappa$ and the fundamental module $V_i$. Recall from Definition \ref{def: diagonal entries} and Theorem \ref{thm: truncation polynomial} the monic polynomial $\lambda_{V,W}(u)$.
\begin{defi}  \label{def: truncation pref}
Let $\Br \mapsto \tilde{\Br}$ denote the group automorphism on $\CR$  which sends each generator $\Psi_{i,a}$ to $\Psi_{\overline{i}, a+\kappa}$.
For $\Bs \in \CD$, define the polynomial $\ell$-weight $(g_i^{\Bs}(u))_{i\in I} \in \CD$ and the rational $\ell$-weight $\overline{\Bs} = (\overline{\Bs}_i(u))_{i\in I} \in \CR$ as follows
\begin{gather*}
g_i^{\Bs}(u) := \lambda_{V_i,L(\Bs^{-1})}(u), \quad
 \overline{\Bs}_i(u) :=  \frac{ g_i^{\Bs}(u) g_i^{\Bs}(u-d_i) }{\Bs_i(u)} \prod_{j: c_{ji} < 0} \prod_{t=1}^{-c_{ji}} \frac{1}{g_j^{\Bs}(u-d_{ij}-t d_j)}.
\end{gather*}
\end{defi}

\begin{example} \label{Example: B2}
We describe the map $\Bs \mapsto \overline{\Bs}$ of Definition \ref{def: truncation pref} for $\Glie$ of type $B_2$. Let $\alpha_1$ be the long root and $\alpha_2$ the short root so that $d_1 = 2,\ d_2 = 1$ and $d_{12} = -1$. The dual Coxeter number is 3 and so $\kappa = 3$. The Dynkin diagram automorphism $i \mapsto \overline{i}$ is identity. We have $V_1 = L(Y_{1,-2})$ and $V_2 = L(Y_{2,-\frac{5}{2}})$. The ratios of highest to lowest $\ell$-weights for $V_1$ and $V_2$ are given by:
\begin{gather*} 
A_{1,0}A_{1,-1}A_{2,0}A_{2,-1}, \quad A_{1,-1}A_{2,0} A_{2,-2}.
\end{gather*}
We obtain from Theorem \ref{thm: truncation polynomial} and Definition \ref{def: truncation pref} that for $\Bs \in \CD$:
\begin{gather*}
g_1^{\Bs}(u) = \Bs_1(u-1) \Bs_1(u) \Bs_2(u-1) \Bs_2(u), \quad g_2^{\Bs}(u) = \Bs_1(u-1) \Bs_2(u-2) \Bs_2(u), \\
\overline{\Bs}_1(u) =  \frac{g_1^{\Bs}(u)g_1^{\Bs}(u-2)}{\Bs_1(u)g_2^{\Bs}(u)g_2^{\Bs}(u-1)} =  \Bs_1(u-3), \quad
\overline{\Bs}_2(u) = \frac{g_2^{\Bs}(u)g_2^{\Bs}(u-1)}{\Bs_2(u)g_1^{\Bs}(u-1)} =  \Bs_2(u-3).
\end{gather*}
Therefore, in type $B_2$ we have $\overline{\Bs} = \tilde{\Bs} \in \CD$.
\end{example}
\begin{example} \label{Example: G2}
Assume $\Glie$ is of type $G_2$. Let $\alpha_1$ be the long root and $\alpha_2$ the short root, so that $d_1 = 3,\ d_2 = 1,$ and $d_{12} = -\frac{3}{2}$. The dual Coxeter number is 4 and so $\kappa = 6$. The Dynkin diagram automorphism is identity. We have $V_1 = L(Y_{1,-\frac{9}{2}})$ and $V_2 = L(Y_{2,-\frac{11}{2}})$. The ratios of highest to lowest $\ell$-weights for $V_1$ and $V_2$ are given by:
$$A_{1,0}A_{1,-1}A_{1,-2} A_{1,-3} A_{2,\frac{1}{2}}A_{2,-\frac{1}{2}}A_{2,-\frac{3}{2}}^2A_{2,-\frac{5}{2}}A_{2,-\frac{7}{2}}, \quad A_{1,-\frac{3}{2}} A_{1,-\frac{7}{2}}  A_{2,0}A_{2,-2} A_{2,-3} A_{2,-5}. $$
As in the previous example, we have for $\Bs \in \CD$:
\begin{align*}
g_1^{\Bs}(u) &= \Bs_1(u) \Bs_1(u-1)\Bs_1(u-2) \Bs_1(u-3) \\
&\quad \times \Bs_2(u+\frac{1}{2}) \Bs_2(u-\frac{1}{2}) \Bs_2(u-\frac{3}{2})^2 \Bs_2(u-\frac{5}{2}) \Bs_2(u-\frac{7}{2}), \\
g_2^{\Bs}(u) &=  \Bs_1(u-\frac{3}{2}) \Bs_1(u-\frac{7}{2}) \Bs_2(u) \Bs_2(u-2) \Bs_2(u-3) \Bs_2(u-5), \\
\overline{\Bs}_1(u) &= \frac{g_1^{\Bs}(u)g_1^{\Bs}(u-3)}{\Bs_1(u) g_2^{\Bs}(u+\frac{1}{2})g_2^{\Bs}(u-\frac{1}{2}) g_2^{\Bs}(u-\frac{3}{2})} = \Bs_1(u-6), \\
\overline{\Bs}_2(u) &= \frac{g_2^{\Bs}(u)g_2^{\Bs}(u-1)}{\Bs_2(u) g_1^{\Bs}(u-\frac{3}{2})} = \Bs_2(u-6).
\end{align*}
Therefore, in type $G_2$ we have $\overline{\Bs} = \tilde{\Bs} \in \CD$.
\end{example}

%It is unclear to us whether $\overline{\Bs} \in \CD$ in general types. This is why in Definition \ref{def: truncated shifted Yangians} we drop the assumption $\Br \in \CD$.
 
\begin{theorem} \label{thm: truncation standard}
For $\Br, \Bs \in \CD$, the standard module $\CW(\Br,\Bs)$  factorizes through the truncated shifted Yangian $Y_{\varpi^{\vee}(\Bs^{-1}\Br)}^{\Br \overline{\Bs}}(\Glie)$. In particular, any irreducible module in category $\BGG^{sh}$ factorizes through a truncated shifted Yangian.
\end{theorem}
\begin{proof}
We have $\CW(\Br, \Bs) = L(\Br) \otimes L(\Bs^{-1})$ with $L(\Br)$ being one-dimensional. For an irreducible module $L(\Be)$ in category $\BGG^{sh}$, one can write $\Be = \Bn^{-1} \Bm$ with $\Bm, \Bn \in \CD$. Then Theorem \ref{thm: cyclicity  cocyclicity} shows that $L(\Be)$ is a quotient of the standard module $\CW(\Bm,\Bn)$. By Remark \ref{rem: truncation one-dim},  it suffices to show that  $L(\Bs^{-1})$ factorizes through $Y_{-\varpi^{\vee}(\Bs)}^{\overline{\Bs}}(\Glie)$. 

Note that $(-\varpi^{\vee}(\Bs), \overline{\Bs}) \in \BP^{\vee} \times \CR$ is truncatable: $$\varpi^{\vee}(\overline{\Bs}) + \varpi^{\vee}(\Bs) = \sum_{i\in I} \deg(g_i^{\Bs}(u)) \alpha_i^{\vee}.$$
In the situation of Proposition \ref{prop: l-weight R-matrix+} with $W = L(\Bs^{-1})$, we have $g_i(u) = g_i^{\Bs}(u)$ by Definition \ref{def: truncation pref}.  Take $V = V_i$ in Theorem \ref{thm: TQ pre R} and compare with Eq.\eqref{equ: difference}. We get the following equality of $\mathrm{End}(L(\Bs^{-1}))$-valued Laurent series in $u^{-1}$:
\begin{align} \label{equ: GKLO vs lowest diagonal entry}
A_i(u)|_{L(\Bs^{-1})} = t_{V_i,L(\Bs^{-1})}(u) \quad \mathrm{for}\ i \in I. 
\end{align}
 The polynomiality of the lowest diagonal entries in Theorem \ref{thm: TQ pre R} implies in the module $L(\Bs^{-1})$ the defining relation $\langle A_i(u)\rangle_+ = 0$ of the truncated shifted Yangian.
\end{proof}

\begin{rem}   \label{rem: GKLO R-matrix}
A similar identification of GKLO series with matrix entries as in Eq.\eqref{equ: GKLO vs lowest diagonal entry} was given in \cite[Corollary 5.9]{KTWWY0}. As commented in \cite[Remark 5.10]{KTWWY0}, these should be specializations of RTT realizations of shifted Yangians \cite{W}. Some particular cases of such a realization appeared in \cite{BK1,FPT}. Notice from Eq.\eqref{YBE} that our R-matrix $\check{R}_{V,W}(u)$ satisfies an RTT relation when $V$ is a finite-dimensional irreducible module over the ordinary Yangian and $W$ is a negative module.
\end{rem}

\begin{rem}  \label{rem: ADE truncation}
(i) Fix $\Br, \Bs \in \CD$ and set $\mu$ and $\nu$ to be the coweights of $\Bs^{-1}\Br$ and $\Br\tilde{\Bs}$ respectively. Recall from Remark \ref{rem: truncated shifted Yangian} the quotient map $\tilde{Y}_{\mu}^{\nu}(\Br\tilde{\Bs}) \longrightarrow Y_{\mu}^{\nu}(\Br\tilde{\Bs})$. 
If $\Glie$ is simply-laced, then there is a classification of irreducible highest $\ell$-weight modules for the original truncated shifted Yangian $Y_{\mu}^{\nu}(\Br\tilde{\Bs})$ in terms of monomial crystals \cite{KTWWY0,KTWWY}, which translated in the language of q-characters by \cite[Theorem 3.3]{Nak1} implies that the $Y_{\mu}(\Glie)$-module $\CW(\Br,\Bs)$ factorizes through 
 $$Y_{\mu}(\Glie) \longrightarrow Y_{\mu}^{\Br\tilde{\Bs}}(\Glie) =  \tilde{Y}_{\mu}^{\nu}(\Br\tilde{\Bs}) \longrightarrow Y_{\mu}^{\nu}(\Br\tilde{\Bs})$$ 
and is a module over  $Y_{\mu}^{\Br\tilde{\Bs}}(\Glie)$. Theorem \ref{thm: truncation standard} shows that $\CW(\Br,\Bs)$ is a module over $Y_{\mu}^{\Br \overline{\Bs}}(\Glie)$. So we expect that $\tilde{\Bs} = \overline{\Bs}$ for $\Glie$ simply-laced. By Examples \ref{Example: B2}--\ref{Example: G2} such an equality seems to hold for $\Glie$ of arbitrary type. This would imply that $\overline{\Bs} \in \CD$ in Definition \ref{def: truncation pref}.

(ii) For $\Glie$ of non simply-laced types, the classification of irreducible highest $\ell$-weight modules over the original truncated shifted Yangians of Remark \ref{rem: truncated shifted Yangian} can be reduced to the known classification in simply-laced types \cite{KTWWY0,KTWWY} via geometric arguments \cite{NW, Nak2}. Theorem \ref{thm: truncation standard} in non simply-laced types follows from this analysis.

(iii) For shifted quantum affine algebras, the second part of Theorem \ref{thm: truncation standard} was known for finite-dimensional irreducible modules \cite[Theorem 12.9]{H0} (see also the discussion in the Introduction of \cite{H0}). More generally, a conjectural description of highest $\ell$-weights of irreducible modules over truncated shifted quantum affine algebras in terms of Langlands dual q-characters was formulated in \cite[Conjecture 12.3]{H0}.
\end{rem} 

\section{Jordan--H\"older property}  \label{sec: truncation}

In this section, as an application of our study of R-matrices in Sections \ref{sec: one-dim R-matrix}--\ref{sec: pre R}, we establish a property in category $\BGG^{sh}$: finite-length modules are stable under tensor product (Theorem \ref{thm: Jordan-Holder}). 
This is referred to as  Jordan--H\"older property.

We also obtain 
a uniform proof of the following known result (see Remark \ref{reftrunc}) : a truncated Yangian has only a finite number of irreducible representations in the category $\mathcal{O}^{sh}$ (Theorem \ref{thm: finite truncation}).

The following result is well-known in the non-shifted case. We omit its proof as it is a direct consequence of the definition of shifted Yangians.
\begin{lem}  \label{lem: sl2 triple}
Let $\mu = \sum_{j\in I} k_j \varpi_j^{\vee}$ be a coweight of $\Glie$. Then for $i \in I$ there exists a unique algebra homomorphism $f_{\mu,i}: Y_{k_i}(sl_2) \longrightarrow Y_{\mu}(\Glie)$ such that
$$ x^+(u) \mapsto d_i^{-k_i} x_i^+(u d_i),\quad x^-(u) \mapsto d_i x_i^-(ud_i),\quad  \xi(u) \mapsto d_i^{-k_i} \xi_i(ud_i). $$
\end{lem}

For $f(u) \in \BC(u)$ a rational function, by the denominator of $f(u)$ we mean the monic polynomial $q(u)$ of $u$ of smallest degree such that $q(u) f(u) \in \BC[u]$. The numerator of $f(u)$ is the denominator of $f(u)^{-1}$.

\begin{lem} \label{lem: l-weight denominator}
Let $\Be \in \CR$ and $(i, a, m) \in I \times \BC \times \BZ_{>0}$. Then $(u-a)^m$ divides the denominator of $\Be_i(u)$ if and only if $A_{i,a}^{-m} \Be$ is an $\ell$-weight of $L(\Be)$.
\end{lem}
\begin{proof}
Write $\mu = \varpi^{\vee}(\Be) = \sum_{j\in I} k_j \varpi_j^{\vee}$. Set $L := f_{\mu,i}(Y_{k_i}(sl_2)) \omega_{\Be} \subset L(\Be)$; it is a $Y_{k_i}(sl_2)$-module via the pullback by $f_{\mu,i}$. We claim that:
\begin{itemize}
\item[(i)] The $Y_{k_i}(sl_2)$-module $L$ is isomorphic to $L(d_i^{-k_i} \Be_i(ud_i))$. 
\item[(ii)] As subspaces of $L(\Be)$, the $\ell$-weight space of the $Y_{k_i}(sl_2)$-module $L$ of $\ell$-weight $d_i^{-k_i} \Be_i(ud_i) A_{a_1}^{-1} A_{a_2}^{-1} \cdots A_{a_N}^{-1}$ is equal to the $\ell$-weight space of the $Y_{\mu}(\Glie)$-module $L(\Be)$ of $\ell$-weight $\Be A_{i,a_1 d_i}^{-1} A_{i,a_2 d_i}^{-1} \cdots A_{i,a_Nd_i}^{-1}$. 
\end{itemize}
(i) can be proved as in \cite[Lemma 4.3]{CP0} by restriction to diagram subalgebras.

For (ii), write $\varpi(\Be) = \sum_{j\in I} m_j \varpi_j$. Then $\omega_{\Be}$ is of weight $m_i$ in the $Y_{k_i}(sl_2)$-module $L$. From $L(\Be) = Y_{\mu}^<(\Glie) \omega_{\Be}$ and the $\BQ_-$-grading on $Y_{\mu}^<(\Glie)$ we get an identification of weight spaces for $N \in \BN$:
$$L_{m_i - 2N} = L(\Be)_{\varpi(\Be)- N \alpha_i} = \mathrm{Vect}(x_{i,m_1}^-x_{i,m_2}^- \cdots x_{i,m_N}^- \omega_{\Be}\ |\ m_1, m_2, \cdots, m_N \in \BN ). $$
Since a weight space is a direct sum of $\ell$-weight spaces, it suffices to prove that the right-hand side of the equality of (ii) is contained in the left-hand side. This is obvious from $f_{\mu,i}(\xi(u)) = d_i^{-k_i} \xi_i(ud_i)$.

It follows that $ A_{i,a}^{-m} \Be$ is an $\ell$-weight of $L(\Be)$ if and only if $d_i^{-k_i} \Be_i(ud_i) A_{ad_i^{-1}}^{-1}$ is an $\ell$-weight of the $Y_{k_i}(sl_2)$-module $L(d_i^{-k_i} \Be_i(ud_i))$. We are reduced to the case $\Glie = sl_2$. Then this follows from the tensor product factorization of Proposition \ref{prop: factorization sl2} because by Definition \ref{def: rational factorisation} the standard factorization of a rational function fixes its denominator and each factor $u-a$ in the denominator contributes to a factor $A_a^{-1}$ of an $\ell$-weight. 
\end{proof}

For $(\mu, \Br) \in \BP^{\vee} \times \CR$ a truncatable pair define $\CR_{\mu}^{\Br}$ to be the set of $\Be \in \CR_{\mu}$ such that the $Y_{\mu}(\Glie)$-module $L(\Be)$ factorizes through the truncated shifted Yangian $Y_{\mu}^{\Br}(\Glie)$. 

\begin{theorem} \label{thm: finite truncation}
The set $\CR_{\mu}^{\Br}$ is finite for any truncatable pair $(\mu, \Br) \in \BP^{\vee} \times \CR$.
\end{theorem}
\begin{rem}\label{reftrunc}
When $\Br \in \CD$, there is a geometric proof of the finiteness, as explained in \cite[Corollary 3.13]{KTWWY0}, by viewing the truncated shifted Yangian $Y_{\mu}^{\Br}(\Glie)$  as a quantization of a scheme supported on a generalized affine Grassmannian slice $\overline{\mathcal{W}}_{\mu}^{\lambda}$ with $\lambda :=\varpi^{\vee}(\Br)$ (see \cite[Proposition 4.10]{KPW}), and then applying the general result \cite[Proposition 5.1]{BLPW}. If $\Br \notin \CD$, choose $\Bs \in \CD$ such that $\Br\Bs \in \CD$. Remark \ref{rem: truncation one-dim} shows that $\CR_{\mu}^{\Br} \subset \Bs^{-1} \CR_{\mu+\varpi^{\vee}(\Bs)}^{\Br\Bs}$. The finiteness of $\CR_{\mu}^{\Br}$ follows from the known case. Our proof is algebraic and close to the situation of truncated shifted quantum affine algebras \cite[Theorem 11.15]{H0}.
\end{rem}
\begin{proof}
For $\Be \in \CR$ and $i \in I$, let $p_i^{\Be}(u)$ and $q_i^{\Be}(u)$ denote the numerator and denominator of $\Be_i(u)$. Recall from Eq.\eqref{equ: integrality} the coefficient $m_i \in \BN$ of $\alpha_i^{\vee}$ in $\varpi^{\vee}(\Br)-\mu$.

\medskip

\noindent {\bf Step 1.} We shall give a necessary condition for $\Be \in \CR_{\mu}^{\Br}$. As in Proposition \ref{prop: l-weight R-matrix+}, let $g_i(u) \in \BC[u]$ be the eigenvalue of $A_i(u) \in Y_{\mu}^{\Br}(\Glie)[u]$ associated to the eigenvector $\omega_{\Be}$ of $L(\Be)$. Then $g_i(u)$ is a monic polynomial of degree $m_i$. We claim that $g_i(u) \Be_i(u)$ is a polynomial, which by Eq.\eqref{equ: truncation} is equivalent to divisibility of polynomials:
 \begin{equation} \label{truncation necessary}
 g_i(u) q_i^{\Br}(u+d_i) \mid p_i^{\Br}(u+d_i) \times \prod_{j: c_{ji} < 0} \prod_{t=1}^{-c_{ji}} g_j(u+d_i-d_{ij}-t d_j).
 \end{equation}
 We need to prove $q_i^{\Be}(u) \mid g_i(u)$. Namely, for any $(a, m) \in \BC \times \BZ_{>0}$ such that $(u-a)^m \mid q_i^{\Be}(u)$, we must have $(u-a)^m \mid g_i(u)$. By Lemma \ref{lem: l-weight denominator}, $\Be A_{i,a}^{-m}$ is an $\ell$-weight of $L(\Be)$. From Proposition \ref{prop: l-weight R-matrix+} (ii) we see that the eigenvalue of $A_i(u)$ on $L(\Be)_{\Be A_{i,a}^{-m}}$ is $g_i(u) (\frac{u-a+d_i}{u-a})^m$, which must be a polynomial. Since $(u-a+d_i)^m$ and $(u-a)^m$ are coprime, we have $(u-a)^m \mid g_i(u)$, as desired.

\medskip 
 
\noindent {\bf Step 2.} Introduce the following finite subsets of $\BC$ for $s \in \BZ_{>0}$:
\begin{align*}
D_{\Glie} &:=  \{d_i - d_{ij} - td_j \ |\ i,j \in I,\ t \in \BZ, \ c_{ji} < 0,\ 1\leq t \leq -c_{ji} \}, \\
X^0_{\Br} &:= \{ a \in \BC\ |\ p_i^{\Br}(a+d_i) = 0 \ \mathrm{for\ certain}\ i \in I \}, \\
X_{\Br}^s &:=  \{ a-c_1 -c_2 - \cdots - c_s \in  \BC\ |\ a \in X_{\Br}^0,\ c_k \in D_{\Glie}\ \mathrm{for}\ 1\leq k \leq s  \}.
\end{align*}
Clearly, $D_{\Glie}$ depends on $\Glie$, and $X_{\Br}^s$ on the triple $(\Br, s, \Glie)$. Set $m := \sum_{i\in I} m_i$ and $X := \cup_{s=0}^{m-1} X_{\Br}^s$. Then $X$ is a finite set depending on the triple $(\Br, \mu, \Glie)$.

Each $\Be \in \CR_{\mu}^{\Br}$ is uniquely determined by the monic polynomials $g_i(u)$ for $i \in I$ from Step 1. Let us attach a quiver $\Gamma_{\Be}$ to $\Be$ as follows:
\begin{itemize}
\item the set of vertices is $V_{\Be} := \{ (i, a) \in I \times \BC\ |\ g_i(a) = 0\}$;
\item draw an arrow $(i,a) \rightarrow (j,b)$ if $c_{ji} < 0$ and there exists $1\leq t \leq -c_{ji}$ such that $b = a + d_i - d_{ij} - t d_j$ and $p^{\Br}_i(a+d_i) \neq 0$. In particular, $b-a \in D_{\Glie}$.
\end{itemize}
We prove that $V_{\Be} \subset I \times X$. Since $g_i(u)$ is of fixed degree $m_i$, this will imply that there are finitely many choices of $g_i(u)$, from which comes the finiteness of $\CR_{\mu}^{\Br}$. 

Let $V_{\Be}^0$ be the set of sink vertices (namely, vertices with no outgoing arrows) of $\Gamma_{\Be}$. If $(i,a)$ is sink, by Relation \eqref{truncation necessary} we must have $p^{\Br}_i(a+d_i) = 0$, since the product $\prod_{j}$ is nonzero. This means $ V_{\Be}^0 \subset I \times  X_{\Br}^0. $

For $s \in \BZ_{>0}$, define $V_{\Be}^s$ to be the set of vertices $(i,a)$ of $\Gamma_{\Be}$ which can be joint to a sink vertex with $s$ arrows. Namely, $(i,a) \in V_{\Be}^s$ if there exist $s$ vertices $$(i_1, a_1), (i_2,  a_2), \cdots, (i_s, a_s)$$ 
such that $(i_s, a_s)$ is sink and there are arrows  
$$(i,a) \rightarrow (i_1,a_1) \rightarrow (i_2,a_2) \rightarrow \cdots (i_{s-1}, a_{s-1}) \rightarrow (i_s, a_s).$$
From our definition of arrows it follows that $a_{k+1}-a_k \in D_{\Glie},\ a_s \in X_{\Br}^0$ and  
$$ a = a_s - (a_s-a_{s-1}) - \cdots - (a_2 - a_1) - (a_1 - a_0) \in X_{\Br}^s. $$
This gives $V_{\Be}^s \subset I \times X_{\Br}^s$. 

By Claim 1 below, the quiver $\Gamma_{\Be}$ is acyclic and every vertex is connected to a sink vertex by at most $m-1$ arrows. We obtain the desired relation
$$V_{\Be} = \bigcup_{s=0}^{m-1} V_{\Be}^s \subset \bigcup_{s=0}^{m-1} (I \times X_{\Br}^s) = I \times X.  $$ 

\medskip

\noindent {\it Claim 1.} In the quiver $\Gamma_{\Be}$ there does not exist any sequence of vertices $(i_k, a_k)$ for $0 \leq k \leq m$ with arrows $(i_l, a_l) \rightarrow (i_{l+1},a_{l+1})$ for $0 \leq l < m$.

\medskip 

Assume the contrary and fix such a sequence
$$ S: \quad (i_0,a_0) \rightarrow (i_1,a_1) \rightarrow (i_2,a_2) \rightarrow \cdots \rightarrow (i_m,a_m). $$
Since $m+1 > \sum_{j\in I} m_j$, there must exist $j \in I$ which appears in the sequence $i_0i_1i_2\cdots i_m$ at least $m_j+1$ times. Each appearance of $j$, say $i_k = j$, gives a root $a_k$ of $g_j(u)$. If $i_k = i_l$ and $k < l$ then necessarily $l-k > 1$ and from Claim 2 below we get $a_k \neq a_l$. So the polynomial $\prod_{0\leq k \leq m, i_k = j}(u- a_k)$ divides $g_j(u)$; the former is of degree at least $m_j+1$, while the latter of degree $m_j$, contradiction.

\medskip

\noindent {\it Claim 2.} In the sequence $S$, if $0 \leq k < k+1 < l \leq m$ then $a_l - a_k \in \frac{1}{2} \BZ_{>0}$. 

\medskip

If $\Glie$ is not of type $G_2$, then $c_{ji} \geq -2$ for all $i,j \in I$. By Eq.\eqref{equ: truncation}, we have $D_{\Glie} \subset \frac{1}{2} \BN$; moreover, if $0 \in D_{\Glie}$, then it must arise from $0 = d_i  - \frac{1}{2} d_j$ so that $d_i = 1$ and $d_j = 2$. This means that for an arrow $(i,a) \rightarrow (j,b)$ of the quiver $\Gamma_{\Be}$: either $b-a \in \frac{1}{2} \BZ_{>0}$; or $b = a$ and $(d_i, d_j) = (1, 2)$. Apply this to the sequence $S$:
$$ a_1 - a_0,\quad a_2 - a_1,\quad \cdots, a_m - a_{m-1} \in \frac{1}{2} \BN.  $$
So $a_l - a_k \in \frac{1}{2} \BN$. If $a_l = a_k$,  then $a_k = a_{k+1} = a_{k+2}$. From $a_k = a_{k+1}$ we obtain $d_{i_{k+1}} = 2$, while from $a_{k+1} = a_{k+2}$ we obtain $d_{i_{k+1}} = 1$, contradiction.

If $\Glie$ is of type $G_2$ with $d_1 = 3$ and $d_2 = 1$, we check Claim 2 directly. By Eq.\eqref{equ: truncation}, the arrows in the quiver $\Gamma_{\Be}$ are of the form
$$ (1, a) \rightarrow (2, a+  \frac{3}{2}), \quad (1, a) \rightarrow (2, a+  \frac{5}{2}), \quad (1, a) \rightarrow (2, a+  \frac{7}{2}),\quad (2, a) \rightarrow (1, a-  \frac{1}{2}). $$
The sequence $i_0i_1\cdots i_{m-1}i_m$ is alternating of the form  $1212\cdots$ or $2121\cdots$. For an arrow $(1,a) \rightarrow (2,b)$, we have $b-a \geq \frac{3}{2}$, while for an arrow $(2,a) \rightarrow (1,b)$ we have $b - a = -\frac{1}{2}$. It follows that $a_l - a_k = \sum_{t=k}^{l-1} (a_{t+1}-a_t)$ as an half integer is at least $S$, where $S$ is an alternating sum of $l-k > 1$ terms of the following form: 
$$\frac{3}{2} - \frac{1}{2} + \frac{3}{2} - \frac{1}{2} + \cdots, \quad -\frac{1}{2} + \frac{3}{2} - \frac{1}{2} + \frac{3}{2} - \cdots.  $$ 
Clearly $a_l-a_k \geq S > 0$.
\end{proof}

Define $\BGG_{\mu}^{fin}$ to be the full subcategory of $\BGG_{\mu}$ consisting of modules with a finite Jordan--H\"older filtration.
In other words, a module $V$ in category $\BGG_{\mu}$ is in $\BGG_{\mu}^{fin}$ if and only if $V$ admits a finite number of irreducible subquotients, if and only if $[V]$ is a finite sum of irreducible isomorphism classes $[L(\Be)]$ for $\Be \in \CR_{\mu}$. Let $\BGG^{sh}_{fin}$ be the direct sum of the categories $\BGG_{\mu}^{fin}$ over all coweights $\mu$. 

\begin{theorem}  \label{thm: Jordan-Holder}
Category $\BGG^{sh}_{fin}$ is closed under tensor product. 
\end{theorem}
\begin{proof}
We need to show that the tensor product $T$ of two arbitrary irreducible modules belongs to category $\BGG_{fin}^{sh}$, namely, $T$ admits a finite number of irreducible subquotients. From Theorem \ref{thm: cyclicity  cocyclicity} we see that $T$ can be realized as a quotient of a tensor product $T'$ of two standard modules. By Eq.\eqref{rel: standard multiplicative} the module $T'$ has the same isomorphism class as a third standard module $\CW$. By Theorem \ref{thm: truncation standard}, the standard module $\CW$ factorizes through a truncated shifted Yangian $Y_{\nu}^{\Bs}(\Glie)$. Since $\CR_{\nu}^{\Bs}$ is finite, $\CW$ admits a finite number of irreducible subquotients, so do the tensor products $T'$ and $T$. 
\end{proof}
The Grothendieck group of category $K_0(\BGG^{sh}_{fin})$, which is the abelian subgroup of $K_0(\BGG^{sh})$ freely generated by the $[L(\Be)]$ for $\Be \in \CR$, is a therefore a subring.  

\begin{rem} (i) Assume $\Glie$ is not of type $E_8$. Combining Proposition \ref{prop: Weyl vs standard} (iii), Theorem \ref{thm: Weyl standard} and Theorem \ref{thm: truncation standard}, we obtain that any highest $\ell$-weight module in category $\BGG^{sh}$ factorizes through a truncated shifted Yangian and belongs to category $\BGG_{fin}^{sh}$.

(ii) The Jordan--H\"older property is known to be true \cite[Theorem 5.27]{H3} for certain category of integrable modules over a quantum affinization. It fails in the original category $\BGG$ of $\Glie$-modules (counterexample: the tensor product of two irreducible Verma modules over $sl_2$ has infinitely many irreducible subquotients) and in the category $\BGG$ of modules over the Borel algebra \cite{HJ} as observed in \cite[Remark 5.12]{HL}.

(iii) For $\Glie = sl_{r+1}$ there is another proof of Theorem \ref{thm: Jordan-Holder} by extending $Y(sl_{r+1})$ to the Yangian $Y(gl_{r+1})$ and using the big center of $Y(gl_{r+1})$; see \cite[Lemmas 6.13, 7.16]{BK}. This is close to the proof of classical fact in the original category $\BGG$ of $\Glie$-modules that each Verma module admits a finite Jordan--H\"older filtration.
\end{rem}

\medskip

\noindent {\bf Acknowledgments:} We are grateful to Hiraku Nakajima and Alex Weekes for useful correspondences, and to the referee for valuable comments and suggestions. The first author is supported by the ERC under the European Union's Framework Programme H2020 with ERC Grant Agreement number 647353 Qaffine. The second author is supported by the Labex CEMPI (ANR-11-LABX-0007-01).

\end{document}